\documentclass[10pt]{article}
\usepackage{amsmath}
\usepackage{amsfonts}
\usepackage{amssymb}
\usepackage{amsthm}
\usepackage{enumerate}
\usepackage[totalwidth=450pt, totalheight=650pt]{geometry}

\begin{document}

\newcommand{\C}{{\mathbb{C}}}
\newcommand{\R}{{\mathbb{R}}}
\newcommand{\Q}{B}
\newcommand{\Z}{{\mathbb{Z}}}
\newcommand{\N}{{\mathbb{N}}}
\newcommand{\q}{\left}
\newcommand{\w}{\right}
\newcommand{\Ninf}{\N_\infty}
\newcommand{\Vol}[1]{\mathrm{Vol}\q(#1\w)}
\newcommand{\B}[4]{B_{\q(#1,#2\w)}\q(#3,#4\w)}
\newcommand{\CjN}[3]{\q\|#1\w\|_{C^{#2}\q(#3\w)}}
\newcommand{\Cj}[2]{C^{#1}\q( #2\w)}
\newcommand{\grad}{\bigtriangledown}
\newcommand{\sI}[2]{\mathcal{I}\q(#1,#2 \w)}
\newcommand{\Det}[1]{\det_{#1\times #1}}
\newcommand{\sK}{\mathcal{K}}
\newcommand{\sKt}{\widetilde{\mathcal{K}}}
\newcommand{\sA}{\mathcal{A}}
\newcommand{\sB}{\mathcal{B}}
\newcommand{\sC}{\mathcal{C}}
\newcommand{\sD}{\mathcal{D}}
\newcommand{\sS}{\mathcal{S}}
\newcommand{\sF}{\mathcal{F}}
\newcommand{\sQ}{\mathcal{Q}}
\newcommand{\sV}{\mathcal{V}}
\newcommand{\sM}{\mathcal{M}}
\newcommand{\sT}{\mathcal{T}}
\newcommand{\cV}{\q( \sV\w)}
\newcommand{\vsig}{\varsigma}
\newcommand{\vsigt}{\widetilde{\vsig}}
\newcommand{\dil}[2]{#1^{\q(#2\w)}}
\newcommand{\dilp}[2]{#1^{\q(#2\w)_{p}}}
\newcommand{\lA}{-\log_2 \sA}
\newcommand{\eh}{\widehat{e}}
\newcommand{\Ho}{\mathbb{H}^1}
\newcommand{\sd}{\sum d}
\newcommand{\dt}{\tilde{d}}
\newcommand{\dhc}{\hat{d}}
\newcommand{\Span}[1]{\mathrm{span}\q\{ #1 \w\}}
\newcommand{\dspan}[1]{\dim \Span{#1}}
\newcommand{\K}{K_0}
\newcommand{\ad}[1]{\mathrm{ad}\q( #1 \w)}
\newcommand{\LtOpN}[1]{\q\|#1\w\|_{L^2\rightarrow L^2}}
\newcommand{\LpOpN}[2]{\q\|#2\w\|_{L^{#1}\rightarrow L^{#1}}}
\newcommand{\LpN}[2]{\q\|#2\w\|_{L^{#1}}}
\newcommand{\Jac}{\mathrm{Jac}\:}
\newcommand{\kapt}{\widetilde{\kappa}}
\newcommand{\gt}{\widetilde{\gamma}}
\newcommand{\gtt}{\widetilde{\widetilde{\gamma}}}
\newcommand{\gh}{\widehat{\gamma}}
\newcommand{\Sh}{\widehat{S}}
\newcommand{\Wh}{\widehat{W}}
\newcommand{\Ih}{\widehat{I}}
\newcommand{\Wt}{\widetilde{W}}
\newcommand{\Xt}{\widetilde{X}}
\newcommand{\Tt}{\widetilde{T}}
\newcommand{\Dt}{\widetilde{D}}
\newcommand{\Phit}{\widetilde{\Phi}}
\newcommand{\Vh}{\widehat{V}}
\newcommand{\Xh}{\widehat{X}}
\newcommand{\deltah}{\widehat{\delta}}
\newcommand{\sSh}{\widehat{\mathcal{S}}}
\newcommand{\sFh}{\widehat{\mathcal{F}}}
\newcommand{\thetah}{\widehat{\theta}}
\newcommand{\ct}{\widetilde{c}}
\newcommand{\at}{\tilde{a}}
\newcommand{\bt}{\tilde{b}}
\newcommand{\fg}{\mathfrak{g}}
\newcommand{\cC}{\q( \sC\w)}
\newcommand{\cG}{\q( \sC_{\fg}\w)}
\newcommand{\cJ}{\q( \sC_{J}\w)}
\newcommand{\cY}{\q( \sC_{Y}\w)}
\newcommand{\cYu}{\q( \sC_{Y}\w)_u}
\newcommand{\cJu}{\q( \sC_{J}\w)_u}
\newcommand{\Bb}{\overline{B}}
\newcommand{\Qb}{\overline{Q}}
\newcommand{\sP}{\mathcal{P}}
\newcommand{\sN}{\mathcal{N}}
\newcommand{\cH}{\q(\mathcal{H}\w)}
\newcommand{\Omegat}{\widetilde{\Omega}}
\newcommand{\Kt}{\widetilde{K}}
\newcommand{\sMt}{\widetilde{\sM}}
\newcommand{\denum}[2]{#1_{#2}}
\newcommand{\phit}{\tilde{\phi}}
\newcommand{\nuset}{\q\{1,\ldots, \nu\w\}}
\newcommand{\diam}[1]{ {\mathrm{diam}}\q\{#1\w\} }
\newcommand{\Nt}{\widetilde{N}}
\newcommand{\psit}{\widetilde{\psi}}
\newcommand{\sigt}{\widetilde{\sigma}}
\newcommand{\muoS}{\q\{\mu_1\w\}}
\newcommand{\Mjcutoff}{\nu+2}
\newcommand{\LplqOpN}[3]{\q\| #3 \w\|_{L^{#1}\q(\ell^{#2}\q(\N^{\nu} \w) \w)\rightarrow L^{#1}\q( \ell^{#2 }\q( \N^{\nu} \w) \w) }}
\newcommand{\LplqN}[3]{\q\| #3 \w\|_{L^{#1}\q(\ell^{#2 }\q(\N^{\nu} \w) \w) }}
\newcommand{\iinf}{\iota_\infty}
\newcommand{\jz}{\ell}
\newcommand{\Lpp}[2]{L^{#1}\q(#2\w)}
\newcommand{\Lppn}[3]{\q\|#3\w\|_{\Lpp{#1}{#2}}}
\newcommand{\Ewmu}{E\cup \q\{\mu_1\w\}}
\newcommand{\Ewmuc}{\q(\Ewmu\w)^c}
\newcommand{\ip}[2]{\q< #1, #2 \w>}

\newtheorem{thm}{Theorem}[section]
\newtheorem{cor}[thm]{Corollary}
\newtheorem{prop}[thm]{Proposition}
\newtheorem{lemma}[thm]{Lemma}
\newtheorem{conj}[thm]{Conjecture}

\theoremstyle{remark}
\newtheorem{rmk}[thm]{Remark}

\theoremstyle{definition}
\newtheorem{defn}[thm]{Definition}

\theoremstyle{remark}
\newtheorem{example}[thm]{Example}

\numberwithin{equation}{section}

\title{Multi-parameter singular Radon transforms II: the $L^p$ theory}
\author{Elias M. Stein\footnote{Partially supported by NSF DMS-0901040.}  and Brian Street\footnote{Partially supported by NSF DMS-0802587.}}
\date{}

\maketitle

\begin{abstract}
The purpose of this paper is to study the $L^p$ boundedness 
of operators of the form
\[
f\mapsto \psi(x) \int f(\gamma_t(x))K(t)\: dt,
\]
where $\gamma_t(x)$ is a $C^\infty$ function defined on a neighborhood
of the origin in $(t,x)\in \R^N\times \R^n$, satisfying
$\gamma_0(x)\equiv x$, $\psi$ is a $C^\infty$ cutoff function supported
on a small neighborhood of $0\in \R^n$, and
$K$ is a ``multi-parameter singular kernel'' supported on a small
neighborhood of $0\in \R^N$.  We also study associated maximal operators.
The goal is, given an appropriate class of kernels $K$, to give
conditions on $\gamma$ such that every operator of the above
form is bounded on $L^p$ ($1<p<\infty$).  The case when
$K$ is a Calder\'on-Zygmund kernel was studied by
Christ, Nagel, Stein, and Wainger; we generalize their
work to the case when $K$ is (for instance) given by a ``product kernel.''  Even
when $K$ is a Calder\'on-Zygmund kernel, our methods yield some new results.
This is the second paper in a three part series.  The first
paper deals with the case $p=2$, while the third paper
deals with the special case when $\gamma$ is real analytic.

\end{abstract}

\section{Introduction}
The goal of this paper is to prove the $L^p$ boundedness
of (a special case of) the multi-parameter singular
Radon transforms introduced in \cite{StreetMultiParameterSingRadonLt}.
We consider operators of the form
\begin{equation}\label{EqnIntroOp}
T\q( f\w) \q( x\w) = \psi\q( x\w) \int f\q( \gamma_t\q( x\w) \w) K\q( t\w) \: dt,
\end{equation}
where $\psi$ is a $C_0^\infty$ cut off function (supported near, say, $0\in \R^n$),
$\gamma_t\q( x\w) =\gamma\q( t,x\w)$ is a $C^\infty$ function defined
on a neighborhood of the origin in $\R^N\times \R^n$ satisfying
$\gamma_0\q( x\w) \equiv x$, and $K$ is a ``multi-parameter''
distribution kernel, supported near $0$ in $\R^N$.
For instance, one could take $K$ to be a ``product kernel''
supported near $0$.\footnote{Our main theorem applies to classes
of kernels other than product kernels.}
  To define this notion, suppose we have
decomposed $\R^N=\R^{N_1}\times \cdots\times \R^{N_\nu}$,
and write $t=\q( t_1,\ldots, t_{\nu}\w)\in \R^{N_1}\times \cdots\times \R^{N_\nu}$.
A product kernel satisfies
\begin{equation*}
\q|\partial_{t_1}^{\alpha_1}\cdots \partial_{t_\nu}^{\alpha_\nu} K\q( t\w)\w|\lesssim \q|t_1\w|^{-N_1-\q|\alpha_1\w|}\cdots \q|t_{\nu}\w|^{-N_{\nu}-\q|\alpha_\nu\w|},
\end{equation*}
along with certain ``cancellation conditions.''\footnote{The simplest 
example of a product kernel is given by $K\q(t_1,\ldots,t_\nu\w)=K_1\q(t_1\w)\otimes\cdots\otimes K_\nu\q(t_\nu\w)$, where $K_1,\ldots, K_\nu$ are standard
Calder\'on-Zygmund kernels.  That is, $K_j$ satisfies
$\q|\partial_{t_j}^{\alpha_j} K_j\q(t_j\w)\w|\lesssim \q|t_j\w|^{-N_j-\q|\alpha_j\w|}$, again along with certain ``cancellation conditions.''  
When $\nu=1$, the class of product kernels is precisely the class
of Calder\'on-Zygmund kernels.
See
Section 16 of \cite{StreetMultiParameterSingRadonLt} for the 
statement of the cancellation conditions.  We do not make them precise
in this paper, since we will be working with more general kernels $K$.}
The goal is to develop conditions on $\gamma$ such that $T$ is
bounded on $L^p$ ($1<p<\infty$).  In addition, we will
prove (under the same conditions on $\gamma$) $L^p$ boundedness ($1<p\leq \infty$) for the corresponding
maximal operator,
\begin{equation*}
\sM f\q( x\w) = \psi\q(x\w) \sup_{0<\delta_1,\ldots,\delta_\nu\leq a}\frac{1}{\delta_1^{N_1} \delta_2^{N_2}\cdots\delta_\nu^{N_\nu} } \int_{\q|t_\mu\w|\leq \delta_\mu} \q|f\q( \gamma_t\q( x\w)\w)\w|\: dt_1\cdots dt_\nu,
\end{equation*}
where $a>0$ is some small number depending on $\gamma$.
In fact, the $L^p$ boundedness of $\sM$ will be a step towards
proving the $L^p$ boundedness of $T$.

This paper is the second in a three part series.  The first paper in the series
\cite{StreetMultiParameterSingRadonLt}
dealt with the $L^2$ theory, and applied to a larger class
of kernels than the results in this paper do (see Section \ref{SectionKernels} for
a discussion).
However, all of the examples we have in mind (and in particular all
of the examples discussed in \cite{StreetMultiParameterSingRadonLt})
do fall under the theory discussed in this paper.
An $L^2$ result from \cite{StreetMultiParameterSingRadonLt}
will serve as one of the main technical lemmas
of this paper.
The third paper in the series \cite{StreetMultiParameterSingRadonAnal} deals with the special case
when $\gamma$ is assumed to be real-analytic.  In this case
many of the assumptions from this paper take a much simpler
form, and some of our results can be improved.
See \cite{SteinStreetMultiparameterSingularRadonTransformsAnnounce}
for an overview of the series.

For a more detailed introduction to the operators
defined in this paper, we refer the reader to
\cite{StreetMultiParameterSingRadonLt},
which discusses special motivating cases, and gives a number
of examples.

\begin{rmk}
The results in this paper generalize the $L^p$ boundedness results
from the work (in the single-parameter setting, when $K$
is a Calder\'on-Zygmund kernel)
of Christ, Nagel, Stein, and Wainger 
\cite{ChristNagelSteinWaingerSingularAndMaximalRadonTransforms}.
In fact, as discussed in
\cite{StreetMultiParameterSingRadonLt}, the results in this paper
generalize the $L^p$ boundedness results 
in \cite{ChristNagelSteinWaingerSingularAndMaximalRadonTransforms} 
even if one considers only the single parameter setting.\footnote{For instance, in the single parameter setting if $K\q(t\w)$ is a Calder\'on-Zygmund kernel supported near $t=0$,
$\gamma_t\q(x\w)$ is a real analytic function in both variables, defined on a neighborhood of $\q(0,0\w)\in \R^N\times \R^n$ satisfying $\gamma_0\q(x\w)\equiv x$,
and $\psi\q(x\w)\in C_0^\infty$ is supported near $x=0$, then it follows from the results in this paper that the operator $T$ given by \eqref{EqnIntroOp} is 
bounded on $L^p$ ($1<p<\infty$) under no additional hypotheses--see \cite{StreetMultiParameterSingRadonAnal} for the full details on this.  This result
is beyond the methods of \cite{ChristNagelSteinWaingerSingularAndMaximalRadonTransforms}.}
One major difficulty that this paper deals with, which did not appear
in \cite{ChristNagelSteinWaingerSingularAndMaximalRadonTransforms},
is that (unlike the single-parameter setting) there is no
relevant Calder\'on-Zygmund theory to fall back on.  At least
not {\it a priori}.
See Remark \ref{RmkNoCZTheory} and Section \ref{SectionSingularIntegrals}
for further details.
What can be used instead is the idea of controlling operators by square functions;
a procedure that has appeared a number of times before.  See,
for instance,
\cite{FeffermanSteinSingularIntegralsOnProductSpaces,NagelSteinOnTheProductTheoryOfSingularIntegrals}.
\end{rmk}

\subsection{Informal statement of the main results}
In this section, we informally state the special case of our main results
when $K\q(t_1,\ldots, t_\nu\w)$ is a product kernel
relative to the decomposition $\R^N=\R^{N_1}\times \cdots\times \R^{N_\nu}$ (see
the introduction for this notion).
We suppose we are given a $C^\infty$ function, $\gamma\q(t,x\w)=\gamma_t\q(x\w)\in \R^n$ defined
on a small neighborhood of the origin in $\q(t,x\w)\in \R^N\times \R^n$,
satisfying $\gamma_0\q(x\w)\equiv x$.
For $t$ sufficiently small, $\gamma_t$ (thought of of a function in the $x$-variable) is a diffeomorphism onto
its image.  Thus, it makes sense to write $\gamma_t^{-1}$, the inverse mapping. 
We define the vector field
\begin{equation*}
W\q(t,x\w)= \frac{d}{d\epsilon}\bigg|_{\epsilon=1} \gamma_{\epsilon t}\circ \gamma_t^{-1} \q(x\w)\in T_x \R^n.
\end{equation*}
Note that $W\q(0,x\w)\equiv 0$.

For a collection of vector fields $\sV$, let $\sD\q(\sV\w)$ denote
the involutive distribution generated by $\sV$.  I.e., the smallest
$C^\infty$ module containing $\sV$ and such that if $X,Y\in \sD\q(\sV\w)$
then $\q[X,Y\w]\in \sD\q(\sV\w)$.
For a multi-index $\alpha\in \N^N$, write $\alpha=\q(\alpha_1,\ldots, \alpha_\nu\w)$, with $\alpha_\mu\in \N^{N_\mu}$.

Decompose $W$ into a Taylor series in the $t$ variable,
\begin{equation*}
W\q(t,x\w)\sim \sum_{\alpha\ne 0} t^{\alpha}X_\alpha.
\end{equation*}
We call $\alpha=\q(\alpha_1,\ldots, \alpha_\nu\w)\in \N^N$ a pure
power if $\alpha_\mu\ne 0$ for precisely one $\mu$.  Otherwise,
we call it a non-pure power.

We assume that the following conditions hold ``uniformly'' for
$\delta=\q(\delta_1,\ldots, \delta_\nu\w)\in \q(0,1\w]^\nu$,
though we defer making this notion of uniform
precise to Section \ref{SectionCurves}.
\begin{itemize}
\item For every $\delta\in \q(0,1\w]^\nu$,
\begin{equation*}
\sD_{\delta}:=\sD\q(\q\{\delta_1^{\q|\alpha_1\w|}\cdots\delta_\nu^{\q|\alpha_\nu\w|}X_{\alpha_1,\ldots,\alpha_\nu} : \q(\alpha_1,\ldots,\alpha_\nu\w)\text{ is a pure power}\w\}\w)
\end{equation*}
is finitely generated as a $C^\infty$ module, uniformly in $\delta$.

\item For every $\delta\in \q(0,1\w]^\nu$,
\begin{equation}\label{EqnIntroWInsD}
W\q(\delta_1 t_1,\ldots, \delta_\nu t_\nu\w)\in \sD_\delta,
\end{equation}
uniformly in $\delta$.
\end{itemize}

\begin{rmk}
If it were not for the ``uniform'' aspect of the above assumptions, they
would be independent of $\delta$.  Thus it is the uniform part,
which we have not made precise, that is the heart of the above
assumptions.
\end{rmk}

Our main theorems are,
\begin{thm}
Under the above assumptions (which are made
precise in Section \ref{SectionCurves}), the operator given by
\begin{equation*}
f\mapsto \psi\q(x\w) \int f\q(\gamma_t\q(x\w)\w) K\q(t\w)\: dt,
\end{equation*}
is bounded on $L^p$ ($1<p<\infty$), for every product kernel $K\q(t_1,\ldots, t_\nu\w)$,
with sufficiently small support, provided $\psi$ has sufficiently small support.
\end{thm}

\begin{thm}
Under the same assumptions, the maximal operator given by
\begin{equation*}
f\mapsto \psi\q(x\w) \sup_{0<\delta_1,\ldots, \delta_\nu<<1} \int_{\q|t\w|<1} \q| f \q(\gamma_{\delta_1 t_1,\ldots, \delta_\nu t_\nu}\q(x\w) \w) \w|\: dt
\end{equation*}
is bounded on $L^p$ ($1<p\leq \infty$).
\end{thm}

The precise statement of our main result (where more general
kernels are considered) can be
found in Theorems \ref{ThmMainThmFirstPass}, \ref{ThmMainThmSecondPass}, and \ref{ThmMainMaxThm}.

\begin{rmk}
In \cite{ChristNagelSteinWaingerSingularAndMaximalRadonTransforms}, the main conditions were stated in terms of slightly different vector fields.
Namely, it was shown that $\gamma$ could be written in the form 
\begin{equation*}
\gamma_t\q(x\w)\sim \exp\q(\sum_{\alpha\ne 0} t^{\alpha} \Xh_\alpha\w)x,
\end{equation*}
where $\Xh_\alpha$ are $C^\infty$ vector fields and the above notation means $\gamma_t\q(x\w) = \exp\q(\sum_{0<\q|\alpha\w|<L} t^{\alpha} \Xh_\alpha\w)x+O\q(\q|t\w|^L\w)$, $\forall L$.  One can easily solve for the $\Xh_\alpha$ in terms of the $X_\alpha$ and {\it vice versa}.  In fact, one can equivalently state our theorems
by replacing $X_\alpha$ with $\Xh_\alpha$, throughout.  Nevertheless, we still need to introduce the vector field $W$, in order to state \eqref{EqnIntroWInsD}.  The connection
between the $\Xh_\alpha$ and the $X_\alpha$ is spelled out in more detail in Section 8 of \cite{StreetMultiParameterSingRadonAnal}.
See also Section 9 of \cite{StreetMultiParameterSingRadonLt}.
\end{rmk}

\begin{rmk}
For several examples of $\gamma$ which satisfy our hypotheses and several examples which do not, the reader is referred to Section 17 of \cite{StreetMultiParameterSingRadonLt}.
\end{rmk}


\section{Kernels}\label{SectionKernels}
In this section, we will discuss the classes of kernels
$K\q( t\w)$ for which we will study operators
of the form \eqref{EqnIntroOp}.
The kernels which we study will be supported in
$\Q^N\q( a\w)=\q\{x\in \R^N: \q|x\w|<a\w\}$, where $a>0$ is some small number
to be chosen later (depending on $\gamma$).
Fix $\nu\in \N$, we will be studying $\nu$ parameter
operators.  Fix $1\leq \mu_0\leq \nu$, and define,
\begin{equation*}
\sA_{\mu_0} = \q\{\delta=\q( \delta_1,\ldots, \delta_\nu\w)\in \q[0,1\w]^\nu: \delta_{\mu_0}\leq \delta_{\mu_0+1}\leq \cdots\leq \delta_{\nu}\w\}.
\end{equation*}
The class of kernels we define will depend on $\mu_0$ (by
depending on $\sA_{\mu_0}$).
In \cite{StreetMultiParameterSingRadonLt}, the class of kernels
depended on a subset $\sA\subseteq\q[0,1\w]^\nu$.  In that
paper, one could use any subset $\sA$ such that if $\delta_1,\delta_2\in \sA$,
then $\delta_1\vee \delta_2\in \sA$ (where
$\delta_1\vee \delta_2$ is given by taking the coordinatewise maximum
of $\delta_1$ and $\delta_2$).  In this paper, we
restrict our attention to $\sA$ of the form $\sA_{\mu_0}$ for
some $\mu_0$.\footnote{Actually, this is not the most general
form of $\sA$ that can be handled by our methods.  See Section \ref{SectionMoreKernels}
for further details.}
This is the only difference between the setting in
\cite{StreetMultiParameterSingRadonLt}
and the setting in this paper.
Notice, $\sA_{\nu}=\q[0,1\w]^\nu$ and $\sA_{1}=\q\{\delta_1\leq\delta_2\leq\cdots\leq \delta_\nu\w\}$; these make up the principle examples
we are interested in.

We suppose we are given $\nu$-parameter dilations on $\R^N$.  That is, 
we are
given $e=\q( e_1,\ldots, e_N\w)$, with each $0\ne e_j= \q( e_j^1,\ldots, e_j^{\nu}\w)\in \q[0,\infty\w)^\nu$.
For $\delta\in \q[0,\infty\w)^\nu$ and $t=\q( t_1,\ldots, t_N\w)\in \R^N$,
we define,\footnote{$\delta^{e_j}$ is defined by standard multi-index
notation: $\delta^{e_j}= \prod_{\mu} \delta_\mu^{e_j^\mu}$.}
\begin{equation}\label{EqnDefndeltat}
\delta t = \q( \delta^{e_1}t_1,\ldots, \delta^{e_{N}}t_N \w),
\end{equation}
thereby obtaining $\nu$-parameter dilations on $\R^N$.
For each $\mu$, $1\leq \mu\leq \nu$, let $t_\mu$ denote those coordinates
$t_j$
of $t=\q(t_1,\ldots, t_N \w)\in \R^N$ such that $e_j^\mu\ne 0$.
For $j=\q(j_1,\ldots, j_\nu\w)\in \Z^\nu$, define $2^j=\q(2^{j_1},\ldots, 2^{j_\nu}\w)$.

The class of distributions we will define depends on 
$N$, $a$, $e$, $\mu_0$, and $\nu$. 
Define,
\begin{equation*} 
\lA_{\mu_0}=\q\{j\in \N^\nu : 2^{-j}\in \sA_{\mu_0}\w\} = \q\{j\in \N^\nu : j_{\mu_0}\geq j_{\mu_0+1}\geq \cdots \geq j_{\nu}\w\}.
\end{equation*}
Given a function $\vsig$ on $\R^N$, and $j\in \N^\nu$, define,
\begin{equation*}
\dil{\vsig}{2^j}\q( t\w) = 2^{j\cdot e_1+\cdots + j\cdot e_N} \vsig\q( 2^j t\w).
\end{equation*}
Note that $\dil{\vsig}{2^j}$ is defined in such a way that,
\begin{equation*}
\int \dil{\vsig}{2^j} \q( t\w) \: dt = \int \vsig\q( t\w) \: dt.
\end{equation*}

\begin{defn}\label{DefnsK}
We define $\sK=\sK\q( N,e,a,\mu_0,\nu\w)$ to be the set of all distributions,
$K$, of the form
\begin{equation}\label{EqnSumDefK}
K=\sum_{j\in \lA_{\mu_0}} \dil{\vsig_j}{2^j},
\end{equation}
where $\q\{\vsig_j\w\}_{j\in\lA_{\mu_0}}\subset C_0^\infty\q( \Q^N\q( a\w) \w)$ is a bounded
set, satisfying
\begin{equation*}
\int \vsig_j\q( t\w) \: dt_\mu =0,
\end{equation*}
unless $0=j_\mu$ or $\mu>\mu_0$ and $j_{\mu}=j_{\mu-1}$.
It was shown in \cite{StreetMultiParameterSingRadonLt} that
any sum of the form \eqref{EqnSumDefK} converges in the
sense of distributions.
\end{defn}

See \cite{StreetMultiParameterSingRadonLt} for a more in-depth
discussion of the class $\sK$.

\begin{rmk}
If each $e_j$ is non-zero in only one component,
kernels in $\sK\q(N,e,a,\nu,\nu\w)$ are known as ``product kernels,''
and kernels in $\sK\q(N,e,a,1,\nu\w)$ are known as ``flag kernels.''
The classes we study here are more general:  we do not insist
that each $e_j$ are nonzero in only one component, and we allow
for any $1\leq\mu_0\leq \nu$.
For an illustrative example where the $e_j$ are not non-zero
in only one component, see Section 17.8 of \cite{StreetMultiParameterSingRadonLt}.
For more information on product and flag kernels see, e.g.,  \cite{NagelRicciSteinSingularIntegralsWithFlagKernels}.
\end{rmk}

\section{Multi-parameter Carnot-Carath\'eodory geometry}\label{SectionCCGeom}
At the heart of the definition of the class of $\gamma$
which we will study lies multi-parameter Carnot-Carath\'eodory
geometry.  Thus, before we can even define the class
of $\gamma$, it is necessary to review the relevant
definitions of multi-parameter Carnot-Carath\'eodory balls.
We defer the theorems we will use to deal with these balls
to Section \ref{SectionCCGeomII}.
Our main reference for Carnot-Carath\'eodory geometry is \cite{StreetMultiParameterCCBalls},
and we refer the reader there for a more detailed discussion.

Let $\Omega\subseteq \R^n$ be a fixed open set, and suppose
$X_1,\ldots, X_q$ are $C^\infty$ vector fields on $\Omega$.
We define the Carnot-Carath\'eodory ball of unit radius, centered
at $x_0\in \Omega$, with respect to the list $X$ by
\begin{equation*}
\begin{split}
B_X\q( x_0\w):=\bigg\{y\in \Omega \:\bigg|\: &\exists \gamma:\q[0,1\w]\rightarrow \Omega, \gamma\q( 0\w) =x_0, \gamma\q( 1\w) =y,  \\
&\gamma'\q( t\w) = \sum_{j=1}^q a_j\q( t\w) X_j\q( \gamma\q( t\w)\w),
 a_j\in L^\infty\q( \q[0,1\w]\w), \\
&\Lppn{\infty}{\q[0,1\w]}{\q(\sum_{1\leq j\leq q} \q|a_j\w|^2\w)^{\frac{1}{2}}}<1\bigg\}.
\end{split}
\end{equation*}
Now that we have the definition of balls with unit radius, we may define
(multi-parameter) balls of any radius merely by scaling the vector fields.
To do so, we assign to each vector field, $X_j$, a (multi-parameter) formal
degree $0\ne d_j=\q(d_j^1,\ldots, d_j^\nu \w)\in \q[0,\infty\w)^\nu$.  For $\delta=\q(\delta_1,\ldots, \delta_\nu \w)\in \q[0,\infty\w)^\nu$,
we define the list of vector fields $\delta X$ to be the list
$\q( \delta^{d_1} X_1,\ldots, \delta^{d_q}X_q\w)$.  Here, $\delta^{d_j}$
is defined by the standard multi-index notation:  $\delta^{d_j} =\prod_{\mu=1}^\nu \delta_\mu^{d_j^\mu}$.
We define the ball of radius $\delta$ centered at $x_0\in \Omega$ by
$$\B{X}{d}{x_0}{\delta} := B_{\delta X}\q( x_0\w). $$

At times, it will be convenient to assume that the ball $\B{X}{d}{x_0}{\delta}$
lies ``inside'' of $\Omega$.  To this end, we make the following
definition.
\begin{defn}
Given $x_0\in \Omega$ and $\Omega'\subseteq \Omega$, we say the list of vector fields $X$ satisfies $\sC\q( x_0,\Omega'\w)$
if for every $a=\q( a_1,\ldots, a_q\w) \in \q( L^{\infty}\q( \q[0,1\w]\w)\w)^q$,
with
$$\Lppn{\infty}{\q[0,1\w]}{\q|a\w|}=\Lppn{\infty}{\q[0,1\w]}{\q(\sum_{j=1}^q \q|a_j\w|^2\w)^{\frac{1}{2}}}<1,$$
there exists a solution $\gamma:\q[0,1\w]\rightarrow \Omega'$ to the
ODE
\begin{equation*}
\gamma'\q( t\w) =\sum_{j=1}^q a_j\q( t\w) X_j\q( \gamma\q( t\w)\w), \quad \gamma\q( 0\w) =x_0.
\end{equation*}
Note, by Gronwall's inequality, when this solution exists, it is unique.
Similarly, we say $\q( X,d\w)$ satisfies $\sC\q( x_0,\delta,\Omega'\w)$
if $\delta X$ satisfies $\sC\q(x_0,\Omega'\w)$.
\end{defn}

One of the main points of \cite{StreetMultiParameterCCBalls} was to
provide a detailed study of the balls $\B{X}{d}{x_0}{\delta}$, under
appropriate conditions on the list $\q( X,d\w)$.  To do this, we first
need to pick a subset $\sA\subseteq \q[0,1\w]^\nu$, and
a compact set $\K\Subset \Omega$.  We will
(essentially) be restricting our attention to those balls
$\B{X}{d}{x_0}{\delta}$ such that $x_0\in \K$ and $\delta\in \sA$.
One should think of $\sA=\sA_{\mu_0}$ for some $\mu_0$, as that
case will be the primary one used in this paper.

\begin{defn}\label{DefnsD}
We say $\q( X,d\w)$ satisfies $\sD\q(\K,\sA\w)$ if the following holds:
\begin{itemize}
\item Take
$\Omega'$ with $\K\Subset \Omega'\Subset \Omega$ and
$\xi>0$ such that for every $\delta\in \sA$ and $x\in \K$, $\q( X,d\w)$ satisfies
$\sC\q( x, \xi\delta, \Omega'\w)$.  

\item For every $\delta\in \sA$
and $x\in \K$, we assume
\begin{equation*}
\q[\delta^{d_i} X_i, \delta^{d_j} X_j\w] = \sum_k c_{i,j}^{k,\delta,x} \delta^{d_k} X_k, \text{ on } \B{X}{d}{x}{\xi\delta}.
\end{equation*}

\item For every ordered multi-index $\alpha$ we assume\footnote{We write $\CjN{f}{0}{U}=\sup_{x\in U} \q|f\q(x\w)\w|$, and if we say the norm is finite, we mean (in addition) that $f$ is continuous on $U$.}
\begin{equation*}
\sup_{\substack{\delta\in \sA\\ x\in \K } } \CjN{\q(\delta X\w)^\alpha c_{i,j}^{k,x,\delta}}{0}{\B{X}{d}{x}{\xi\delta}}<\infty.
\end{equation*}
\end{itemize}
If we wish to be explicit about $\Omega'$ and $\xi$, we write $\sD\q( \K,\sA, \Omega', \xi\w)$.
\end{defn}

It is under condition $\sD\q( \K, \sA\w)$ that the balls $\B{X}{d}{x}{\delta}$
were studied in \cite{StreetMultiParameterCCBalls}.
We refer the reader to Section \ref{SectionCCGeomII} for an overview of
the theorems from \cite{StreetMultiParameterCCBalls} that
we shall use.

In what follows, we will not be directly given a list of
vector fields with formal degrees satisfying $\sD\q( \K, \sA\w)$,
\begin{equation*}
\q( X_1,d_1\w), \ldots, \q( X_q,d_q\w),
\end{equation*}
but, rather,  we will be given a list of $C^\infty$ vector fields with
formal degrees which we will assume to ``generate'' such a list.

To understand this, let $\q( X_1,d_1\w),\ldots, \q( X_r,d_r\w)$ be
$C^\infty$ vector fields with associated formal degrees
$0\ne d_j\in \q[0,\infty\w)^\nu$.
For a list $L=\q( l_1,\ldots, l_m\w)$ where $1\leq l_j\leq r$, we define,
\begin{equation*}
\begin{split}
X_L &= \ad{X_{l_1}}\ad{X_{l_2}}\cdots \ad{X_{l_{m-1}}} X_{l_m},\\
d_L &= d_{l_1}+d_{l_2}+\cdots +d_{l_m}.
\end{split}
\end{equation*}
We define $\sS=\q\{\q( X_L,d_L\w) : L \text{ is any such list}\w\}.$

\begin{defn}\label{DefnGeneratesAFiniteList}
We say $\sS$ is {\it finitely generated}
 or that $\q( X_1,d_1\w), \ldots, \q( X_r,d_r\w)$ {\it generates a finite list} if there exists finite subset,
$\sF\subseteq \sS$, such that $\sF$ satisfies $\sD\q( \K,\sA\w)$\footnote{Here,
we are thinking of $\K$ and $\sA$ fixed.}
 and
\begin{equation*}
\q( X_j, d_j\w)\in \sF, \quad 1\leq j\leq r.
\end{equation*}
If we enumerate the vector fields in $\sF$,
\begin{equation*}
\sF=\q\{\q(X_1,d_1\w),\ldots, \q( X_q,d_q\w) \w\},
\end{equation*}
we say that $\q( X_1,d_1\w),\ldots, \q( X_r,d_r\w)$ {\it generates
the finite list} $\q( X_1,d_1\w) ,\ldots, \q( X_q,d_q\w)$.
Note that, if $\sS$ is finitely generated, $\q( X_1,d_1\w),\ldots \q( X_r,d_r\w)$ could generate
many different finite lists.
However, if we let $\q( X,d\w)$ and $\q( X',d'\w)$ be two different such lists
then
either choice will work for our purposes.
In fact, it is shown in \cite{StreetMultiParameterSingRadonLt}
that $\q( X,d\w)$ and $\q(X',d'\w)$ are {\it equivalent}
in a sense that is made precise and discussed at length
in that paper.
It follows that, in every place we use these notions, it will not make
a difference which finite list we use.
Thus, we will unambiguously
say ``$\q( X_1,d_1\w) ,\ldots, \q(X_r,d_r\w)$ generates
the finite list $\q( X_1,d_1\w) ,\ldots, \q( X_q,d_q\w)$,''
to mean that $\q( X_1,d_1\w), \ldots, \q(X_r,d_r\w)$ generates
a finite list and $\q( X_1,d_1\w),\ldots, \q( X_q,d_q\w)$ can
be any such list.
\end{defn}

\section{Surfaces}\label{SectionCurves}
In this section, we define the class of $\gamma$ for which
we will study operators of the form \eqref{EqnIntroOp}.
This is nothing but a reprise of the definitions in
\cite{StreetMultiParameterSingRadonLt}, and we
refer the reader there for more details.

We assume we are given an open subset $\Omega\subseteq \R^N$, a 
fixed $\mu_0$, $1\leq \mu_0\leq \nu$, 
and dilations $e$
as in Section \ref{SectionKernels}.  

\begin{defn}
Given a multi-index $\alpha\in \N^N$, we define 
$$\deg\q(\alpha\w)=\sum_{j=1}^N \alpha_j e_j\in \q[0,\infty\w)^\nu.$$
\end{defn}

Let $\K\Subset\Omega'\Subset \Omega''\Subset \Omega$ be subsets of $\Omega$ 
with $\K$ compact and $\Omega'$ and $\Omega''$ open but relatively compact
in $\Omega$.
Our goal in this section is to define a class of $C^\infty$ functions
$$\gamma\q( t,x\w) : \Q^N\q(\rho\w) \times \Omega''\rightarrow \Omega$$
such that $\gamma\q( 0,x\w)=x$.
Here $\rho>0$ is a small number.
This class of functions will depend on $\mu_0$, $N$, $e$, and $\Omega$
(nominally, the class will also depend on $\K$, $\Omega'$, and $\Omega''$, but 
this will not be an essential point).
This class will be such that if $\psi$ is a $C_0^\infty$
function supported in the interior of $\K$, then there is an $a>0$ sufficiently
small such that the operator given by \eqref{EqnIntroOp}
is bounded on $L^p$ ($1<p<\infty$) for every $K\in \sK\q( N,e,a,\mu_0,\nu\w)$. 

Note, by possibly shrinking $\rho>0$ we may assume that, for each $t\in \Q^N\q( \rho\w)$, $\gamma\q( t, \cdot\w)\big|_{\Omega'}$ is a diffeomorphism (in the $x$-variable) onto its image.
From now on we will assume this.  Also, as in the introduction,
we will write $\gamma_t\q( x\w)= \gamma\q( t,x\w)$.

Unlike the work in \cite{ChristNagelSteinWaingerSingularAndMaximalRadonTransforms},
we separate the condition on $\gamma_t$ into two aspects.
For the first, suppose we are given a list of $C^\infty$ vector fields
on $\Omega'$, $X_1,\ldots, X_q$, with associated $\nu$-parameter formal
degrees, $d_1,\ldots d_q$, satisfying $\sD\q( \K, \sA_{\mu_0}, \Omega',\xi\w)$
for some $\xi>0$ (we will see later where these vector fields will come from).
We denote the list $\q(X_1,d_1\w),\ldots,\q(X_q,d_q\w)$ by $\q(X,d\w)$.

\begin{defn}\label{DefnControl}
Suppose we are given a $C^\infty$ vector field on $\Omega'$, 
depending smoothly on $t\in \Q^N\q( \rho\w)$, $W\q( t,x\w)\in T_x\Omega'$.
We say $\q( X,d\w)$
{\it controls} $W\q( t,x\w)$ if there exists a $\rho_1\leq \rho$
and $\tau_1\leq \xi$ such that for every $x_0\in \K$, $\delta\in \sA_{\mu_0}$
there exist functions $c_l^{x_0,\delta}$ on $\Q^N\q( \rho_1\w)\times \B{X}{d}{x_0}{\tau_1\delta}$ satisfying
\begin{itemize}
\item $W\q( \delta t,x\w) = \sum_{l=1}^q c_l^{x_0,\delta}\q( t,x\w) \delta^{d_l} X_l\q( x\w)$ on 
$\Q^N\q( \rho_1\w)\times \B{X}{d}{x_0}{\tau_1\delta}$, 
where $\delta t$ is defined as in \eqref{EqnDefndeltat}.
\item $\sup_{\substack{x_0\in \K\\\delta\in \sA_{\mu_0} } }\sum_{\q|\alpha\w|+\q|\beta\w|\leq m} \CjN{\q(\delta X\w)^{\alpha} \partial_t^\beta c_l^{x_0,\delta}}{0}{\Q^N\q( \rho_1\w)\times \B{X}{d}{x_0}{\tau_1\delta}}<\infty,$ for every
$m$.
\end{itemize}
\end{defn}

\begin{defn}\label{DefnControlCurve}
We say $\q( X,d\w)$ {\it controls} $\gamma_t\q( x\w)$ if 
$\q( X,d\w)$ controls $W$ where
\begin{equation*}
W\q(t,x\w)=\frac{d}{d\epsilon}\bigg|_{\epsilon=1} \gamma_{\epsilon t}\circ \gamma_t^{-1} \q( x\w).
\end{equation*}
Here, $\epsilon\q( t_1,\ldots, t_N\w) = \q( \epsilon t_1,\ldots, \epsilon t_N\w)$,
and so is unrelated to the dilations $e$.
\end{defn}

Part of our assumption will be that a particular family of vector fields 
$\q( X,d\w)$ controls $\gamma_t$.  Where these vector fields come from
constitutes the other part of our assumption on $\gamma$.

Let $W$ be as in Definition \ref{DefnControlCurve}.  Let
$X_{\alpha}\q( x\w)$ be the Taylor coefficients of
$W$ when the Taylor series is taken in the $t$ variable:
\begin{equation}\label{EqnDefnXjalpha}
W\q( t,x\w) \sim \sum_{\q|\alpha\w|>0} t^\alpha X_{\alpha}\q( x\w),
\end{equation}
so that $X_{\alpha}$ is a $C^\infty$ vector field on $\Omega'$.

Our assumption on $\gamma_t$ is that if we take the set of vector fields
with formal degrees:
\begin{equation}\label{EqnCurvesDefnS}
\sS=\q\{\q(X_{\alpha},\deg\q(\alpha\w) \w):\deg\q( \alpha\w)\text{ is non-zero in only one component}\w\},
\end{equation}
then there is a finite subset $\sF\subseteq \sS$ such that
$\sF$ generates a finite list $\q( X,d\w) = \q( X_1,d_1\w),\ldots, \q(X_q,d_q\w)$
and this finite list controls $\gamma_t$.


\begin{rmk}
The list of vector fields $\q( X,d\w)$ depends on a few choices we have
made in the above:  it depends on the chosen finite subset $\sF$
and it depends on the chosen list generated by $\sF$.  However,
neither of these choices affects $\q( X,d\w)$ in an essential way.
This is discussed in detail in \cite{StreetMultiParameterSingRadonLt}.
\end{rmk}


\section{Statement of Results}
Fix $\Omega\subseteq \R^n$ open, and $\K\Subset \Omega'\Subset\Omega''\Subset\Omega$
with $\K$ compact (with nonempty interior) and $\Omega'$ and $\Omega''$
open but relatively compact in $\Omega$.
Let
\begin{equation*}
\gamma\q( t,x\w):\Q^N\q( \rho\w)\times \Omega''\rightarrow \Omega
\end{equation*}
be a $C^\infty$ function such that $\gamma\q( 0,x\w) \equiv x$.
Here, $\rho>0$ is a small number.

Fix $\nu\in \N$ positive, and $\mu_0$, $1\leq \mu_0\leq \nu$.
Furthermore, let $e=\q( e_1,\ldots, e_N\w)$ be given, with $0\ne e_j\in \q[0,\infty\w)^\nu$.
{\bf We suppose $\gamma$ satisfies the assumptions of Section \ref{SectionCurves}
with this $\K$, $\mu_0$, and $e$.}

\begin{thm}\label{ThmMainThmFirstPass}
For every $\psi\in C_0^\infty\q( \R^n\w)$ supported in the interior
of $\K$, there exists $a>0$ such that for every $K\in \sK\q( N,e,a,\mu_0,\nu\w)$
the operator
\begin{equation*}
f\mapsto \psi\q( x\w) \int f\q( \gamma_t\q( x\w) \w)K\q( t\w) \: dt
\end{equation*}
extends to a bounded operator $L^p\q( \R^n\w)\rightarrow L^p\q( \R^n\w)$,
for every $1<p<\infty$.
\end{thm}

Actually, as shown in \cite{StreetMultiParameterSingRadonLt},
Theorem \ref{ThmMainThmFirstPass} follows directly from the following, slightly more general,
theorem.

\begin{thm}\label{ThmMainThmSecondPass}
There exists $a>0$ such that for every $\psi_1,\psi_2\in C_0^\infty\q( \R^n\w)$
supported on the interior of $\K$, every $K\in \sK\q( N,e,a,\mu_0,\nu\w)$
and every $C^\infty$ function
\begin{equation*}
\kappa\q( t,x\w):\Q^N\q( a\w)\times \Omega''\rightarrow \C
\end{equation*}
the operator
\begin{equation*}
T f\q( x\w)= \psi_1\q( x\w) \int f\q( \gamma_t\q( x\w) \w) \psi_2\q( \gamma_t\q( x\w)\w) \kappa\q( t,x\w) K\q( t\w) \: dt
\end{equation*}
extends to a bounded operator $L^p\q( \R^n\w) \rightarrow L^p\q( \R^n\w)$
for every $1<p<\infty$.
\end{thm}

\begin{rmk}
We focus on the more general Theorem \ref{ThmMainThmSecondPass}.  The importance
of the form of the operator in Theorem \ref{ThmMainThmSecondPass} is that
the class of operators is closed under taking $L^2$ adjoints, which is not
true of the class of operators in Theorem \ref{ThmMainThmFirstPass}.
See Section 12.3 of \cite{StreetMultiParameterSingRadonLt}
for the proof of this.
\end{rmk}

We also study maximal functions.  Define, for $\psi_1,\psi_2\in C_0^\infty\q( \R^n\w)$,
supported on the interior of $\K$ with $\psi_1\geq 0$,
\begin{equation*}
\sM f\q( x\w) = \sup_{\delta\in \sA_{\mu_0}} \psi_1\q( x\w)\int_{\q|t\w|<a} \q|f\q( \gamma_{\delta t}\q( x\w) \w)\psi_2\q( \gamma_{\delta t}\q( x\w)\w)\w|\: dt,
\end{equation*}
where $a>0$ is a fixed, small real number (this should be considered
as the same $a$ as in Theorem \ref{ThmMainThmSecondPass}).
Then, we have,
\begin{thm}\label{ThmMainMaxThm}
Under the same assumptions as Theorem \ref{ThmMainThmSecondPass},
\begin{equation*}
\LpN{p}{\sM f}\lesssim \LpN{p}{f},
\end{equation*}
for every $1<p\leq \infty$.
\end{thm}

Most of the paper will be devoted to exhibiting the proofs of Theorems
\ref{ThmMainThmSecondPass} and \ref{ThmMainMaxThm} in the case when
$\mu_0=\nu$.  That is, when $\sA_{\mu_0}=\q[0,1\w]^\nu$.  This case
contains all the main ideas, but allows for simpler notation.
We then describe the modifications needed to attack the
more general situation in Section \ref{SectionMoreKernels}.

\begin{rmk}\label{RmkNoCZTheory}
A major difficulty in the proofs of Theorems \ref{ThmMainThmSecondPass}
and \ref{ThmMainMaxThm} is, there is no appropriate {\it a priori} multi-parameter
theory analogous to the standard theory of Calder\'on-Zygmund singular
integrals to fall back on.
Indeed, in the single parameter ($\nu=1$) case, one can
``smooth out'' the operators in question enough to apply
the standard Calder\'on-Zygmund theory and obtain
some $L^p$ estimates (see Section 18 of 
\cite{ChristNagelSteinWaingerSingularAndMaximalRadonTransforms}).
To get around this issue, we will use square function techniques, that
will allow us to introduce the Calder\'on-Zygmund theory in a
more round about manner.
Once Theorem \ref{ThmMainThmSecondPass} is proved, we will
actually obtain, {\it a posteriori,} a prototype for
some aspects of 
a multi-parameter Calder\'on-Zygmund type theory.
See Section \ref{SectionSingularIntegrals}.
\end{rmk}

\begin{rmk}\label{RmkABetterMaxForRealAnal}
While the maximal results in this paper are new, the class of $\gamma$
we consider was more motivated by Theorem \ref{ThmMainThmSecondPass}
than by Theorem \ref{ThmMainMaxThm}.
Indeed, we will see in Section \ref{SectionMaximalComment}
that there are choices of $\gamma$ where the $L^p$ boundedness of the singular
integral fails, but the $L^p$ boundedness of the maximal function holds.
We have not attempted to state Theorem \ref{ThmMainMaxThm}
in such a way to include these choices.  This deficiency will
be partially rectified in
\cite{StreetMultiParameterSingRadonAnal}.
Indeed, in \cite{StreetMultiParameterSingRadonAnal}, we will obtain
the following result.
Let $\gamma_t\q(x\w):\R^N_0\times \R^n_0\rightarrow \R^n$ be
a {\bf real analytic function} defined on a neighborhood
of $\q(0,0\w)\in \R^N\times \R^n$, and satisfying $\gamma_0\q(x\w)\equiv x$;
assuming {\it no additional hypotheses}.
Define a maximal operator by,
\begin{equation*}
\sMt f\q(x\w) = \sup_{\delta=\q(\delta_1,\ldots, \delta_N\w)\in \q[0,1\w]^N}\psi\q(x\w) \int_{\q|t\w|\leq a} \q|f\q(\gamma_{\delta_1t_1,\ldots, \delta_N t_N}\q(x\w)\w)\w|\: dt.
\end{equation*}
Then, $\sMt$ is bounded on $L^p$ ($1<p\leq \infty$), provided $\psi$ is
supported on a sufficiently small neighborhood of $0$, and $a$
is sufficiently small.
Note that this result is clearly not a special case of Theorem \ref{ThmMainMaxThm}.
See, also, \cite{SteinStreetMultiparameterSingularRadonTransformsAnnounce}
for a further discussion of this point.
\end{rmk}

\begin{rmk}
The hypotheses we put on $\gamma$ depend on the choice of the dilators $e$ and on $\mu_0$.  In particular, if all the $e_j$
are nonzero in only one component, the hypotheses still depend on $\mu_0$.  In fact, our hypotheses are weaker when $K$
is a ``flag kernel'' ($\mu_0=1$ and all $e_j$ are nonzero in only one component) than when $K$ is a ``product kernel'' ($\mu_0=\nu$ and all
$e_j$ are nonzero in only one component).  This is discussed in detail, with examples, in Sections 17.5 and 17.6 of
\cite{StreetMultiParameterSingRadonLt}.
\end{rmk}

\section{Basic Notation}
Throughout the paper, for $v=\q(v_1,\ldots, v_n\w)\in \R^n$, we write 
$\q|v\w|$ for $\q( \sum_j \q|v_j\w|^2 \w)^{\frac{1}{2}}$, and we write
$\q|v\w|_\infty$ for $\sup_j \q|v_j\w|$.
$B^{n}\q( \eta\w)$ will denote
the ball of radius $\eta>0$ in the $\q|\cdot\w|$ norm.
For two numbers $a,b\in \R$ we write
$a\vee b$ for the maximum of $a$ and $b$ and $a\wedge b$ for the minimum.
If instead, $a=\q( a_1,\ldots, a_n\w), b=\q( b_1,\ldots, b_n\w)\in \R^n$, we write $a\vee b$ (respectively, $a\wedge b$)
for $\q(a_1\vee b_1,\ldots, a_n \vee b_n \w)$ (respectively, $\q(a_1\wedge b_1,\ldots, a_n \wedge b_n\w)$).

For a vectors $\delta=\q(  \delta_1,\ldots, \delta_\nu\w), d=\q( d_1,\ldots, d_\nu\w)\in \R^\nu$, we define $\delta^d$ by the standard multi-index notation.
I.e., $\delta^d=\prod_{\mu=1}^{\nu} \delta_\mu^{d_\mu}$.  Also we will
write $2^d = \q(2^{d_1},\ldots,2^{d_\nu}\w)$.  

Given a, possibly arbitrary, set $U\subseteq \R^n$ and a continuous
function $f$ defined on a neighborhood of $U$, we write
$$\CjN{f}{j}{U} = \sum_{\q|\alpha\w|\leq j}\sup_{x\in U} \q|\partial_x^\alpha f\q( x\w)\w|,$$
and if we state that $\CjN{f}{j}{U}$ is finite, we mean that the partial
derivatives up to order $j$ of $f$ exist on $U$, are continuous, and the above norm
is finite.
If $f$ is replaced by a vector field $Y=\sum_k a_k\q( x\w)\partial_{x_k}$, then we write,
$$\CjN{Y}{j}{U} = \sum_{k}\CjN{a_k}{j}{U}.$$

Given a matrix $A$, we write $\q\| A\w\|$ for the usual operator norm.
Given two integers $1\leq m\leq n$, we let $\sI{m}{n}$ denote the set
of all lists of integers $\q( i_1,\ldots, i_m\w)$ such that
$$1\leq i_1<i_2<\cdots<i_m\leq n.$$
Furthermore, suppose $A$ is an $n\times q$ matrix, and suppose
$1\leq n_0\leq n\wedge q$.
For $I\in \sI{n_0}{n}$, $J\in \sI{n_0}{q}$ we let $A_{I,J}$ denote
the $n_0\times n_0$ matrix given by taking the rows from $A$ which
are listed in $I$ and the columns from $A$ which are listed in $J$.
We define
$$\Det{n_0} A = \q( \det A_{I,J} \w)_{\substack{I\in \sI{n_0}{n}\\J\in \sI{n_0}{q}}},$$
so that, in particular, $\Det{n_0} A$ is a {\it vector} (it will not
be important to us in which order the coordinates are arranged).
$\Det{n_0} A$ comes up when one changes variables.  Indeed, suppose
$\Phi$ is a $C^1$ diffeomorphism from an open subset $U\subset \R^{n_0}$
mapping to an $n_0$ dimensional submanifold of $\R^n$, where
this submanifold is given the induced Lebesgue measure $dx$.  Then, we have
$$\int_{\Phi\q( U\w)} f\q( x\w) \: dx = \int_U f\q( \Phi\q( t\w)\w) \q|\Det{n_0} d\Phi\q( t\w)\w|\: dt.$$

If $A=\q( A_1,\ldots, A_q\w)$ is a list of, possibly non-commuting, operators,
we will use ordered multi-index notation to define $A^\alpha$, where
$\alpha$ is a list of numbers $1,\ldots, q$.  $\q|\alpha\w|$ will
denote the length of the list.  For instance, if
$\alpha=\q( 1,4,4,2,1\w)$, then $\q| \alpha\w|=5$ and $A^\alpha = A_1A_4A_4A_2A_1$.  Thus, if $A_1,\ldots A_q$ are vector fields, then $A^\alpha$ is an
$\q|\alpha\w|$ order partial differential operator.

If $f:\R^{n}\rightarrow \R^m$ is a map, then we write,
\begin{equation*}
d f\q( x\w) \q( \frac{\partial}{\partial x_j}\w),
\end{equation*}
for the differential of $f$ at the point $x$ applied to the vector field
$\frac{\partial}{\partial x_j}$.  If $f$ is a function of two variables,
$f\q( t,x\w):\R^N\times \R^n\rightarrow \R^m$, and we wish to view $d f$ as a linear transformation
acting on the vector space spanned by $\frac{\partial}{\partial_{t_j}}$ ($1\leq j\leq N$),
then we instead write,
\begin{equation*}
\frac{\partial f}{\partial t} \q( t,x\w)
\end{equation*}
to denote this linear transformation.  Hence, it makes sense to write,
\begin{equation*}
\det_{n_0\times n_0} 
\frac{\partial f}{\partial t} \q( t,x\w),
\end{equation*}
where $n_0\leq m\wedge N$.

If $\psi_1,\psi_2\in C_0^\infty\q( \R^n\w)$, we write $\psi_1\prec \psi_2$
to denote that $\psi_2\equiv 1$ on a neighborhood of the support
of $\psi_1$.

We will have occasion to use vector valued functions.
We denote by $L^p\q( \ell^q\q( \N^{\nu} \w)\w)$ the set of sequences
of measurable functions $\q\{f_j\w\}_{j\in\N^\nu}$ such that,
$$\LpN{p}{\q(\sum_{j\in \N^\nu} \q|f_j \w|^q \w)^{1/q}}<\infty.$$

Finally, we will devote a good deal of notation to multi-parameter
Carnot-Carath\'eodory geometry.  
See Sections \ref{SectionCCGeom} and \ref{SectionCCGeomII}.

\section{Multi-parameter Carnot-Carath\'eodory geometry Revisited}\label{SectionCCGeomII}
In this section, we present the results that allow us to deal
with Carnot-Carath\'eodory geometry.  The results we
outline here are contained in Section 4 of
\cite{StreetMultiParameterCCBalls}.
The heart of this theory is the ability to ``rescale'' vector fields.
This rescaling is obtained by conjugating by a particular diffeomorphism,
which will be denoted by $\Phi$ in what follows.

Before we can enter into details, we must explain the connection
between multi-parameter balls and single-parameter balls.
We assume we are given $C^\infty$ vector fields
$X_1,\ldots, X_q$ with associated formal degrees $0\ne d_1,\ldots,d_q\in \q[0,\infty\w)^\nu$.  Here, $\nu\in \N$ is the number of parameters.
Given the multi-parameter degrees, we obtain corresponding
single parameter degrees, which we denote
by $\sd$ and are defined by $\q( \sd\w)_j:=\sum_{\mu=1}^\nu d_j^\mu=\q|d_j\w|_1$.
Let $\delta\in \q[0,\infty\w)^\nu$, and suppose we wish to study
the ball
\begin{equation*}
\B{X}{d}{x_0}{\delta}.
\end{equation*}
Decompose $\delta=\delta_0\delta_1$ where $\delta_0\in \q[0,\infty\w)$
and $\delta_1\in\q[0,\infty\w)^\nu$ (of course this decomposition
is not unique).
Then, directly from the definition, we obtain:
\begin{equation*}
\B{X}{d}{x_0}{\delta}=\B{\delta_1 X}{\sd}{x_0}{\delta_0}=\B{\delta X}{\sd}{x_0}{1}.
\end{equation*}
Thus, studying a ball of radius $\delta$ corresponding to $\q( X,d\w)$
is the same as studying a ball of radius $1$ corresponding
to $\q( \delta X, \sd\w)$.
For this reason, taking $\K$ as in Section \ref{SectionCCGeom}
and assuming $\q( X,d\w)$ satisfies $\sD\q( \K,\q[0,1\w]^\nu\w)$,
we will fix $x_0\in \K$ and $\delta\in \q[0,1\w]^\nu$ and study balls
of radius $\approx 1$ centered at $x_0$ corresponding to the
vector fields with single-parameter formal degrees $\q( \delta X, \sd\w)$.
In what follows, it will be important that all of the implicit constants
are independent of $x_0\in \K$ and $\delta\in \q[0,1\w]^\nu$.

We now turn to stating a theorem about a list of $C^\infty$ vector fields
$Z_1,\ldots, Z_q$ defined on an open set $\Omega\subseteq \R^n$, with associated single parameter formal degrees
$\dt_1,\ldots, \dt_q\in \q(0,\infty\w)$.  The special case
we are interested in is the case when
$\q( Z,\dt\w) = \q( \delta X, \sd\w)$; i.e., when
$Z_j=\delta^{d_j} X_j$ and $\dt_j=\q|d_j\w|_1$.

Fix $x_0\in \Omega$ and $1\geq \xi>0$.\footnote{In our primary example,
one takes $x_0\in \K$ and $\xi$ as in $\sD\q( \K,\q[0,1\w]^\nu\w)$.}
Let $n_0=\dspan{Z_1\q( x_0\w), \ldots, Z_q\q( x_0\w)}$.
For $\q( j_1,\ldots, j_{n_0}\w)\in \sI{n_0}{q}$, let $Z_J$ denote the list of vector fields
$Z_{j_1},\ldots, Z_{j_{n_0}}$.
Fix $J_0\in \sI{n_0}{q}$ such that
\begin{equation*}
\q| \det_{n_0\times n_0} Z_{J_0}\q( x_0\w)\w|_\infty = \q|\det_{n_0\times n_0} Z\q( x_0\w) \w|_\infty,
\end{equation*}
where we have identified $Z\q( x_0\w)$ with the $n\times q$ matrix
whose columns are given by $Z_1\q( x_0\w),\ldots, Z_q\q( x_0\w)$ and
similarly for $Z_{J_0}\q( x_0\w)$.
We assume $\q(Z,\dt\w)$ satisfies $\sC\q( x_0,\xi,\Omega\w)$.
In addition we suppose that there are functions $c_{i,j}^k$ on
$\B{Z}{\dt}{x_0}{\xi}$ such that
\begin{equation*}
\q[Z_i,Z_j\w]=\sum c_{i,j}^k Z_k, \text{ on }\B{Z}{\dt}{x_0}{\xi}.
\end{equation*}
We assume that:
\begin{itemize}
\item $\CjN{Z_j}{m}{\B{Z}{\dt}{x_0}{\xi}}<\infty$ for every $m$.
\item $\sum_{\q|\alpha\w|\leq m} \CjN{ Z^\alpha c_{i,j}^k}{0}{\B{Z}{\dt}{x_0}{\xi}}<\infty,$
for every $m$ and every $i,j,k$.
\end{itemize}
We say that $C$ is an $m$-admissible constant if $C$ can be chosen to
depend only on upper bounds for the above two quantities (for that particular 
choice of $m$),
$m$,
upper and lower bounds for $\dt_1,\ldots, \dt_{q}$, an
upper bound for $n$ and $q$, and a lower bound for
$\xi$.
Note that, in our primary example $\q( Z,\dt\w)=\q( \delta X, \sd\w)$, 
$m$-admissible constants can be chosen independent of
$x_0\in \K$ and $\delta\in \q[0,1\w]^\nu$.
We write $A\lesssim_m B$ if $A\leq C B$, where $C$ is an $m$-admissible
constant, and we write $A\approx_m B$ if $A\lesssim_m B$ and $B\lesssim_m A$.
Finally, we say $\tau=\tau\q( \kappa\w)$ is an $m$-admissible constant
if $\tau$ can be chosen to depend on all the parameters an $m$
admissible constant may depend on, and $\tau$ may also depend on $\kappa$.

\begin{rmk}
Under the above assumptions, the classical Frobenius theorem applies to show that (near $x_0$) there is a submanifold passing through $x_0$,
of dimension $n_0$, whose tangent space is spanned by $Z_1,\ldots, Z_q$.  The balls $\B{Z}{\dt}{x_0}{\delta}$ are open subsets of
this submanifold, and we use the notation $\Vol{\B{Z}{d}{x_0}{\delta}}$ to denote the volume of this ball in the sense of the induced Lebesgue measure
on this submanifold.  See \cite{StreetMultiParameterCCBalls} for more details.
\end{rmk}

\begin{thm}\label{ThmMainCCThm}
There exist $2$-admissible constants $\eta_1,\xi_1>0$ such that
if the map $\Phi:B^{n_0}\q( \eta_1\w)\rightarrow \B{Z}{\dt}{x_0}{\xi}$
is defined by
\begin{equation*}
\Phi\q( u\w) = e^{u\cdot Z_{J_0}} x_0,
\end{equation*}
we have
\begin{itemize}
\item $\Phi:B^{n_0}\q( \eta_1\w) \rightarrow  \B{Z}{\dt}{x_0}{\xi}$ is injective.
\item $\B{Z}{\dt}{x_0}{\xi_1}\subseteq \Phi\q( B^{n_0}\q( \eta_1\w)\w)$.
\item For all $u\in B^{n_0}\q( \eta_1\w)$, $\q|\det_{n_0\times n_0} d\Phi\q( u\w)\w|\approx_2 \q|\det_{n_0\times n_0} Z\q( x_0\w)\w|$.
\item $\Vol{\B{Z}{\dt}{x_0}{\xi_1}}\approx_2 \q|\det_{n_0\times n_0} Z\q( x_0\w)\w|$.
\end{itemize}
Furthermore, if we let $Y_j$ be the pullback of $Z_j$ under the map
$\Phi$, then we have, for $m\geq 0$,
\begin{equation}\label{EqnRescaledSmooth}
\CjN{Y_j}{m}{B^{n_0}\q( \eta_1\w)}  \lesssim_{m\vee 2} 1,
\end{equation}
\begin{equation}\label{EqnRescaledCjN}
\CjN{f}{m}{B^{n_0}\q(\eta_1\w)} \approx_{\q(m-1\w)\vee 2} \sum_{\q|\alpha\w|\leq m} \CjN{Y^\alpha f}{0}{B^{n_0}\q( \eta_1\w)}.
\end{equation}
Finally,
\begin{equation}\label{EqnRescaledSpan}
\q|\det_{n_0\times n_0} Y_{J_0}\q( u\w)\w|\approx_{2} 1,\quad \forall u\in B^{n_0}\q( \eta_1\w).
\end{equation}
\end{thm}

Note that, in light of \eqref{EqnRescaledSmooth} and \eqref{EqnRescaledSpan},
pulling back by the map $\Phi$ allows us to rescale the vector fields
$Z$ in such a way that the rescaled vector fields, $Y$, are smooth
and span the tangent space (uniformly in any relevant parameters).
We will also need the following 
technical result.

\begin{prop}\label{PropExtraCCStuff}
Suppose $\xi_2,\eta_2>0$ are given.  Then there exist $2$-admissible constants
$\eta'=\eta'\q( \xi_2\w)>0$, $\xi'=\xi'\q( \eta_2\w)>0$ such that,
\begin{equation*}
\Phi\q( B^{n_0}\q( \eta'\w)\w) \subseteq \B{Z}{\dt}{x_0}{\xi_2},
\end{equation*}
\begin{equation*}
\B{Z}{\dt}{x_0}{\xi'}\subseteq \Phi\q( B^{n_0}\q( \eta_2\w)\w).
\end{equation*}
\end{prop}
\begin{proof}
The existence of $\eta'$ can be seen by 
applying Theorem \ref{ThmMainCCThm} with $\xi$ replaced by $\xi\wedge \xi_2$.
The existence of $\xi'$ can be shown by
combining the proof
of Proposition 3.21 of \cite{StreetMultiParameterCCBalls} with
the proof of Proposition 4.16 of \cite{StreetMultiParameterCCBalls}.
\end{proof}

\begin{rmk}
With a slight abuse of notation, when we say $m$-admissible constant,
where $m<2$, we will take that to mean a $2$-admissible constant.
Using this new notation, the $\vee$ in \eqref{EqnRescaledSmooth}
and \eqref{EqnRescaledCjN} may be removed.
\end{rmk}

\begin{rmk}
It is not hard to see that the single-parameter formal degrees
$\dt$ do not play an essential role in the above
(see Remark 3.3 of \cite{StreetMultiParameterCCBalls}).
In fact, one could state Theorem \ref{ThmMainCCThm},
taking all the formals degrees $\dt_j=1$ and that would
be sufficient for our purposes.  Moreover, in every place
we use the single-parameter formal degrees $\dt$,
they are inessential.
We have stated the result as above, though, to allow
us to transfer seamlessly between the vector fields
$\q( X,d\w)$ and $\q( Z,\dt\w)$, without any hand-waving
about the formal degrees.
\end{rmk}

\section{Some special single-parameter operators}\label{SectionCZOps}
In this section, we describe a certain single-parameter (i.e., $\nu=1$)
special case of our main theorem.  This special
case will be easy to obtain using the theory described in
Section \ref{SectionCCGeomII}, along with the classical
Calder\'on-Zygmund theory of singular integrals.
We will then, in Section \ref{SectionSquareFunc}, use the operators developed in this section
to create an appropriate Littlewood-Paley theory adapted to the more
general operators of this paper.

We suppose we are given $\K\Subset \Omega'\Subset \Omega''\Subset \Omega$ as in Section \ref{SectionCCGeom}
and a list of $C^\infty$ vector fields $X_1,\ldots X_q$ on $\Omega'$
with {\it single}-parameter formal degrees $d_1,\ldots, d_q\in \q(0,\infty\w)$.
We assume that there exists an $\xi>0$ such that
$\q( X,d\w)$ satisfies 
$\sD\q( \K, \q[0,1\w],\Omega', \xi\w)$.

\begin{rmk}\label{RmkSingleParamHomogType}
In the case above, for $\delta\leq \frac{\xi_1}{2}$,\footnote{Here,
$\xi_1$ is as in Theorem \ref{ThmMainCCThm}.}
 we have, using  the notation and results from Theorem \ref{ThmMainCCThm},
\begin{equation}\label{EqnSingleParamHomogType}
\begin{split}
\Vol{\B{X}{d}{x_0}{2\delta}}&\approx \q|\det_{n_0\times n_0} \q(2\delta\w) X\q( x_0\w)\w| \\
&\approx \q|\det_{n_0\times n_0}\delta X\q( x_0\w)\w|\\
&\approx \Vol{\B{X}{d}{x_0}{\delta}},
\end{split}
\end{equation}
where we have used our usual notation that $\delta X$ denotes the matrix
whose columns are given by $\delta^{d_1} X_1,\ldots, \delta^{d_q}X_q$.
\eqref{EqnSingleParamHomogType} is the fundamental estimate involved in
showing that the balls $\B{X}{d}{\cdot}{\cdot}$ form a space
of homogeneous type.  Thus, if $X_1,\ldots, X_q$ spanned the
tangent space (i.e., if $n_0=n$), then the above Carnot-Carath\'eodory balls
would be open subsets of $\R^n$ and would endow $\K$ with
the structure of a space of homogeneous type.  However, we are interested
in the case when $X_1,\ldots, X_q$ do not, necessarily, span the tangent
space.  In this case, the classical Frobenius theorem applies to show
that $X_1,\ldots, X_q$ foliate the $\K$ into leaves,\footnote{The involutive
distribution generated by $X_1,\ldots, X_q$ is finitely generated as a $C^\infty$
module in light of condition $\sD$.}
and each leaf is a space of homogeneous type.  Using the
coordinate charts ($\Phi$) developed in Section \ref{SectionCCGeomII},
we will be able to exploit this fact in what follows.
This idea was also used in Section 6.2 of \cite{StreetMultiParameterCCBalls}.
\end{rmk}

We consider the function,
\begin{equation*}
\gamma\q( t,x\w):\Q^q\q( \rho\w)\times \Omega''\rightarrow \Omega,
\end{equation*}
given by,
\begin{equation*}
\gamma_{\q( t_1,\ldots, t_q\w)}\q( x\w) = e^{t_1X_1+\cdots+ t_q X_q}x.
\end{equation*}

We define dilations on $\R^q$ by, for $\delta>0$,
\begin{equation*}
\delta \q( t_1,\ldots, t_q\w) = \q( \delta^{d_1} t_1,\ldots, \delta^{d_q} t_q\w),
\end{equation*}
and we define for $\vsig: \R^q\rightarrow \C$, and $j\geq 0$,
\begin{equation*}
\dil{\vsig}{2^j} \q( t\w) = 2^{j\q( d_1+\cdots+ d_q\w)} \vsig\q( 2^j t\w).
\end{equation*}

Fix $a>0$ a small number.  Let $\q\{\vsig_j\w\}_{j\in \N}\subset C_0^\infty \q( \Q^q\q( a\w)\w)$ be a bounded subset satisfying,
\begin{equation*}
\int \vsig_j \q( t\w) \: dt=0, \text{ if } j>0,
\end{equation*}
where we are including $0\in \N$.  Let $K$ be
the distribution defined by,
\begin{equation*}
K\q( t\w) =\sum_{j\in \N} \dil{\vsig_j}{2^j}\q( t\w),
\end{equation*}
where the sum is taken in the sense of distributions.
I.e., $K\in \sK\q( q,e,a,1,1\w)$, where we are taking 
$e=\q( d_1,\ldots, d_q\w)\in \q(0,\infty\w)^q$.

Let $\kappa:\Q^q\q( a\w) \times \Omega''\rightarrow \C$ be a $C^\infty$
function and let $\psi_1,\psi_2\in C_0^\infty\q( \R^n\w)$ be supported
on the interior of $\K$.  Define the operator,
\begin{equation*}
Tf\q( x\w) =\psi_1\q( x\w) \int f\q( \gamma_t\q( x\w) \w) \psi_2\q( \gamma_t\q( x\w)\w) \kappa\q( t,x\w) K\q( t\w) \: dt.
\end{equation*}

\begin{thm}\label{ThmSingleParamSingInt}
There is an $a>0$ such that $T$ (as defined above) is bounded on $L^p$, $1<p<\infty$.
\end{thm}

In addition we will need a maximal theorem.  Take $\psi_1,\psi_2\in C_0^\infty\q( \R^n\w)$ supported on the interior of $\K$ with $\psi_1\geq 0$, and define
\begin{equation*}
\sM f\q( x\w) = \sup_{\delta\in \q(0,1\w]} \psi_1\q( x\w)\int_{\q|t\w|<a} \q| f\q( \gamma_{\delta t}\q( x\w)\w) \psi_2\q( \gamma_{\delta t}\q( x\w)\w)\w|\: dt.
\end{equation*}
Then, we have,
\begin{thm}\label{ThmSingleParamMax}
For $a>0$ sufficiently small, $\LpN{p}{\sM f} \lesssim \LpN{p}{f}$, for $1<p\leq\infty$.
\end{thm}

Note that Theorems \ref{ThmSingleParamSingInt} and \ref{ThmSingleParamMax}
are special cases of Theorems \ref{ThmMainThmSecondPass} and
\ref{ThmMainMaxThm}, respectively (this is proved in Section 17.1
of \cite{StreetMultiParameterSingRadonLt}).  However, we will
see that Theorems \ref{ThmSingleParamSingInt} and
\ref{ThmSingleParamMax} can be proven by reduction to
the classical Calder\'on-Zygmund theory.  We will then use
these theorems to develop an appropriate Littlewood-Paley
theory, with which to prove Theorems \ref{ThmMainThmSecondPass}
and \ref{ThmMainMaxThm}.

We separate the proofs of Theorem \ref{ThmSingleParamSingInt} and
\ref{ThmSingleParamMax} into two cases.  In the first case, we prove
the results under the assumption that $X_1,\ldots, X_q$ span the tangent
space at each point of $\K$; this is covered in Section \ref{SectionCZSpan}.
Then, we use the results in Section \ref{SectionCZSpan} to prove
the more general case when $X_1,\ldots, X_q$ do not necessarily
span the tangent space at each point; this is covered in
Section \ref{SectionCZNoSpan}.

	\subsection{When the vector fields span}\label{SectionCZSpan}
In this section, we prove Theorems \ref{ThmSingleParamSingInt}
and \ref{ThmSingleParamMax} under the additional
assumption that $\inf_{x\in \K} \q|\det_{n\times n} X\q( x\w) \w|\gtrsim 1$.
We will then be able to apply this special case to each leaf,
in order to obtain the more general statement
of Theorems \ref{ThmSingleParamSingInt} and \ref{ThmSingleParamMax}.

\begin{rmk}
In this particular case, Theorems \ref{ThmSingleParamSingInt}
and \ref{ThmSingleParamMax} (and the methods
used in this section) are already well understood.
If fact, as we will see, when the vector fields span the operators
in Theorem \ref{ThmSingleParamSingInt} are just Calder\'on-Zygmund singular integrals corresponding
to the space of homogeneous type given by the balls $\B{X}{d}{x}{\delta}$.  The maximal function
$\sM$ is comparable to the usual maximal function associated to this space of homogeneous type.
See \cite{NagelRosaySteinWaingerEstimatesForTheBergmanAndSzegoKernels,KoenigOnMaximalSobolevAndHolderEstimatesForTheTangentialCR}
for similar methods.  
One could also use the methods of \cite{ChristNagelSteinWaingerSingularAndMaximalRadonTransforms} to prove the results
in this section, but those methods are much stronger than are needed for this simple special case.
In any case, we do not know of a reference that has these
results in the exact form we need them and so include
the short proof here.
\end{rmk}

We focus now on the proof of Theorem \ref{ThmSingleParamSingInt}
and will explain the proof of Theorem \ref{ThmSingleParamMax}
at the end of the section.
Thus, let $T$ be as in Theorem \ref{ThmSingleParamSingInt}.
We already know from the theory in \cite{StreetMultiParameterSingRadonLt}
that
\begin{equation*}
\LpOpN{2}{T}\lesssim 1.
\end{equation*}
Our goal is to show that $T$ is a Calder\'on-Zygmund singular
integral operator and it will follow that $T$ is bounded on
$L^p$ for $1<p\leq 2$.  Since the class of operators
discussed in Theorem \ref{ThmSingleParamSingInt} is self-adjoint 
it will follow that $T$ is bounded on $L^p$ for $1<p<\infty$.

\begin{rmk}
Actually, it is not hard to see that, instead of using the $L^2$ theory
in \cite{StreetMultiParameterSingRadonLt}, we could apply the $T\q( b\w)$
theorem to obtain the $L^p$ boundedness of $T$.  We leave this approach
to the interested reader.
\end{rmk}

Let $\rho\q( x,y\w)$ be the Carnot-Carath\'eodory metric corresponding
to the vector fields with formal degrees $\q(X_1,d_1\w),\ldots, \q(X_q,d_q\w)$.  That is,
\begin{equation*}
\rho\q( x,y\w) = \inf \q\{\delta>0 : y\in \B{X}{d}{x}{\delta}\w\}.
\end{equation*}
Let $\Kt\q( x,y\w)$ denote the Schwartz kernel of $T$.  We wish to show
that 
\begin{equation}\label{EqnToShowSingInt}
\int_{\B{X}{d}{y_1}{2\delta}^{c}} \q| \Kt\q(x,y_1 \w)-\Kt\q( x,y_2\w) \w|\: dx\lesssim 1, \text{ if }y_2\in \B{X}{d}{y_1}{\delta},
\end{equation}
and the $L^p$ boundedness ($1<p\leq 2$) of $T$ will follow from the classical
theory of Calder\'on-Zygmund singular integrals (see, e.g., Theorem 3
on page 19 of \cite{SteinHarmonicAnalysis}).  This uses the
fact that the balls $\B{X}{d}{\cdot}{\cdot}$ form a space of
homogeneous type, as discussed in Remark \ref{RmkSingleParamHomogType}. 

We now turn to proving \eqref{EqnToShowSingInt}.
As is well known, it suffices to prove the inequality,
\begin{equation}\label{EqnToShowSmoothSingInt}
\q|X_x^{\alpha} X_y^{\beta} \Kt\q( x,y\w)\w|\lesssim \frac{\rho\q( x,y\w)^{-\deg\q(\alpha\w)-\deg\q(\beta\w)} }{\Vol{\B{X}{d}{x}{\rho\q( x,y\w)} }},
\end{equation}
where $X_x$ denotes the list of vector fields $\q( X_1,\ldots, X_q\w)$ thought
of a partial differential operators in the $x$ variable, $\alpha$ denotes
an ordered multi-index, and
\begin{equation*}
\deg\q( \alpha\w) = \sum_{j=1}^q k_j d_j,
\end{equation*}
where $k_j$ is the number of times $j$ appears in the ordered multi-index
$\alpha$.
Similarly for $X_y$ and $\beta$.  Actually,
it would suffice to prove \eqref{EqnToShowSmoothSingInt} in the special
case $\q|\alpha\w|=0$, $\q|\beta\w|=1$, but this is no simpler to prove.

For $j\in \N$, let $T_j$ be the operator given by
\begin{equation*}
T_j f\q( x\w) = \psi_1\q( x\w) \int f\q( \gamma_t\q( x\w) \w) \psi_2\q( \gamma_t\q( x\w) \w) \kappa\q( t,x\w) \dil{\vsig}{2^j}\q( t\w) \: dt.
\end{equation*}
Let $\Kt_j\q( x,y\w)$ be the Schwartz kernel of $T_j$.  Thus,
$\Kt = \sum_{j\in \N} \Kt_j$.
To prove \eqref{EqnToShowSmoothSingInt}, it suffices to show that there
is an $a>0$ (independent of $j$) such that, when $\Kt_j$ is defined as above,
we have,
\begin{itemize}
\item $\Kt_j\q(x,y\w)$ is supported on $\q\{\q( x,y\w): \rho\q( x,y\w) \leq \xi_1 2^{-j}\w\}$, where $\xi_1$ is a constant, independent of $j$, and
\item $\q|\q(2^{-j}X_x \w)^\alpha \q(2^{-j} X_y \w)^{\beta} \Kt_j\q( x,y\w)\w|\lesssim \frac{1}{\Vol{\B{X}{d}{x}{2^{-j}}}}$, where, as usual, $\delta X$ denotes
the list of vector fields $\delta^{d_1} X_1,\ldots, \delta^{d_q} X_q$.
\end{itemize}

To prove the above we apply Theorem \ref{ThmMainCCThm} to the list of
vector fields $\q( 2^{-j}X,d\w)$ with $x_0\in \K$.  Note that all
of the assumptions in that section hold uniformly for $x_0\in \K$
and $j\in \N$.  Thus we obtain, $\eta_1,\xi_1>0$ and for each
$x_0\in \K$ and $j\in \N$ a map,
\begin{equation*}
\Phi_{j,x_0}:B^{n}\q( \eta_1\w) \rightarrow \B{X}{d}{x_0}{\xi 2^{-j}},
\end{equation*}
as in Theorem \ref{ThmMainCCThm}.
To prove the claim about the support of $\Kt_j$ it suffices to show
that for $x_0\in \K$ and $\q|t\w|\leq a$, we have,
\begin{equation}\label{EqnSingIntToShowSupport}
e^{t_1 2^{-jd_1} X_1+\cdots+ t_q 2^{-jd_q} X_q}x_0\in \B{X}{d}{x_0}{\xi_1 2^{-j}}.
\end{equation}
Let $Y_1,\ldots, Y_q$ denote the pullbacks of $X_1,\ldots, X_q$ via
$\Phi_{j,x_0}$.
Pulling \eqref{EqnSingIntToShowSupport} back via $\Phi_{j,x_0}$ it suffices to show,
\begin{equation*}
e^{t_1 Y_1+\cdots+t_q Y_q} 0 \in \B{Y}{d}{x_0}{\xi_1}.
\end{equation*}
Take $\eta'>0$ so small that $B^{n}\q( \eta'\w) \subseteq \B{Y}{d}{0}{\xi_1}$.
It is easy to see that this is possible, since $\xi_1\gtrsim 1$ and $\inf_{u\in B^{n}\q( \eta_1\w) }\q|\det_{n\times n} Y\q( u\w)\w|\gtrsim 1$.
Since $Y_1,\ldots, Y_q\in C^\infty$ uniformly in $j$ and $x_0$ (see Theorem \ref{ThmMainCCThm}),
it is follows that for $\q|t\w|\leq a$, with $a>0$ sufficiently small,
we have,
\begin{equation*}
e^{t_1Y_1+\cdots+ t_qY_q}0\in B^{n}\q( \eta'\w) \subseteq \B{Y}{d}{0}{\xi_1};
\end{equation*}
which completes the proof of the support of $\Kt_j$.

Since $\q|\det d\Phi_{j,x_0}\q( u\w)\w|\approx \Vol{\B{X}{d}{x_0}{2^{-j}}}$
for $u\in B^{n}\q( \eta_1\w)$ (and in light of the support of $\Kt_j$),
to prove the differential inequalities on $\Kt_j$, it suffices to show,
\begin{equation*}
\q|\q(\det d\Phi_{j,x_0}\q( v\w) \w) Y_u^\alpha Y_v^\beta \Kt_j\q(\Phi_{j,x_0}\q( u\w), \Phi_{j,x_0}\q( v\w) \w)\w|\lesssim 1,
\end{equation*}
where $u,v\in B^{n}\q( \eta_1\w)$.  Using that $Y_1,\ldots, Y_q$ and
$\Phi_{j,x_0}$ are $C^\infty$ uniformly in any relevant parameters, it suffices
to show that for all multi-indices $\alpha$ and $\beta$ (no longer ordered),
\begin{equation*}
\q|\partial_u^{\alpha} \partial_v^{\beta} \q( \Kt_j\q(\Phi_{j,x_0}\q( u\w), \Phi_{j,x_0}\q( v\w) \w)  \det d\Phi_{j,x_0}\q( v\w)\w) \w|\lesssim 1;
\end{equation*}
that is, that $ \Kt_j\q(\Phi_{j,x_0}\q( u\w), \Phi_{j,x_0}\q( v\w) \w)  \det d\Phi_{j,x_0}\q( v\w)$ is $C^\infty$ uniformly in any relevant parameters.

Let $\Phi^{\#}$ denote the map $\Phi^{\#} g = g\circ \Phi$.
Then, $ \Kt_j\q(\Phi_{j,x_0}\q( u\w), \Phi_{j,x_0}\q( v\w) \w) \det d\Phi_{j,x_0}\q( v\w)$
is the Schwartz kernel of the map
\begin{equation*}
\Tt_j = \Phi_{j,x_0}^{\#} T_j \q(\Phi_{j,x_0}^{\#}\w)^{-1}.
\end{equation*}
It is easy to see that,
\begin{equation*}
\Tt_j g\q( u\w) = \psi_1\q( \Phi_{j,x_0}\q( u\w) \w) \int g\q( \gt_t\q( u\w) \w) \psi_2\q( \gt_t\q( u\w)\w) \kappa\q( 2^{-j} t, \Phi_{j,x_0}\q( u\w) \w) \vsig_j\q( t\w) \: dt,
\end{equation*}
where $\gt_t\q( u\w) = e^{t_1Y_1+\cdots + t_qY_q} u$.
From Theorem \ref{ThmMainCCThm}, we have that,
\begin{equation*}
\q|\det Y_{J_1}\q( 0\w) \w|\gtrsim 1,
\end{equation*}
for some $J_1\in \sI{n}{q}$.  Without loss of generality, by reordering
the coordinates, we may assume $J_1=\q( 1,\ldots, n\w)$.
Recall that $\vsig_j$ is supported in $\Q^q\q( a\w)$.
For each $u$ and $t_{n+1},\ldots, t_q$ fixed, define the map,
\begin{equation*}
\Psi_{u,t_{n+1},\ldots, t_q}\q( t_1,\ldots, t_n\w)= e^{t_1 Y_1+\cdots +t_q Y_q}u.
\end{equation*}
Using the $C^\infty$ bounds for $Y_1,\ldots, Y_q$, we have that for
$\q|t\w|\leq a$ (with $a>0$ sufficiently small),
\begin{equation*}
\q|\det d \Psi_{u,t_{n+1},\ldots,t_q}\q( t_1,\ldots, t_n\w)  \w|\approx 1.
\end{equation*}
Applying the change of variables $v=\Psi_{u,t_{n+1},\ldots,t_q}\q( t_1,\ldots, t_n\w)$,
it is immediate to see that the Schwartz kernel of $\Tt_j$ is $C^\infty$
uniformly in any relevant parameters.
This completes the proof of Theorem \ref{ThmSingleParamSingInt}
in the case when $X_1,\ldots, X_q$ span the tangent space.

The proof of Theorem \ref{ThmSingleParamMax}, in this case, is merely
a simpler reprise of the above.  Indeed, the standard Calder\'on-Zygmund
theory shows that the maximal function,
\begin{equation*}
\sMt f\q( x\w) = \sup_{\delta\in \q(0,1\w]} \psi_1\q( x\w)\frac{1}{\Vol{\B{X}{d}{x}{\delta}} }\int_{y\in \B{X}{d}{x}{\delta}} \q|f\q( y\w) \psi_2\q( y\w)\w|\: dy,
\end{equation*}
is bounded on $L^p$ ($1<p\leq \infty$).
Hence we need only show that pointwise bound,
\begin{equation}\label{EqnToShowSpanMax}
\sM f\q( x\w) \lesssim \sMt f\q( x\w).
\end{equation}
Let 
\begin{equation*}
A_\delta f\q( x\w) = \psi_1\q( x\w) \int_{\q|t\w|\leq a} f\q( \gamma_{\delta t}\q( x\w)\w) \q|\psi_2\q( \gamma_{\delta t}\q( x\w)\w) \w|\: dt,
\end{equation*}
so that $\sM f\q( x\w) = \sup_{\delta\in \q(0,1\w]} A_\delta \q|f\w|$. 
To show \eqref{EqnToShowSpanMax}, it suffices to show, for $a>0$ sufficiently
small, independent of $\delta$,
\begin{itemize}
\item If $\Kt_\delta\q( x,y\w)$ is the Schwartz kernel of $A_\delta$, then $\Kt_\delta\q( x,y\w)$ is supported on $\q( x,y\w)$ such that
$y\in \B{X}{d}{x}{\delta}$.
\item $\q|\Kt_\delta\q( x,y\w)\w|\lesssim \frac{1}{\Vol{\B{X}{d}{x}{\delta} } }$.
\end{itemize}
This follows just as above.

	\subsection{When the vector fields do not span}\label{SectionCZNoSpan}
In this section, we complete the proof of Theorems
\ref{ThmSingleParamSingInt} and \ref{ThmSingleParamMax}
by proving the case when $X_1,\ldots, X_q$ do not span
the tangent space.
The idea, as outlined in Remark \ref{RmkSingleParamHomogType},
is to use the fact that the involutive
distribution generated by $X_1,\ldots, X_q$ is finitely generated
as a $C^\infty$ module.  In fact, in light of
$\sD\q( \K, \q[0,1\w], \Omega, \xi\w)$, $X_1,\ldots, X_q$ are
generators of this distribution (as a $C^\infty$ module).  
Because of this, the classical Frobenius theorem applies to foliate the ambient space into leaves;
$X_1,\ldots, X_q$ spanning the tangent space to each leaf.
The goal is to apply
the theory of Section \ref{SectionCZSpan} to each leaf.
We will be able to do this by utilizing the coordinate charts
on each leaf given to us by Theorem \ref{ThmMainCCThm}.

Let $n_0\q( x\w) = \dim\Span{X_1\q( x\w),\ldots X_q\q( x\w)}$.
Then, there exist $\eta_1, \xi_1>0$ such that 
for each $x\in \K$, we obtain a map
$$\Phi_{x}:B^{n_0 \q(x\w)}\q( \eta_1\w)\rightarrow \B{X}{d}{x}{\xi},$$
as in Theorem \ref{ThmMainCCThm}, by applying Theorem \ref{ThmMainCCThm}
to the vector fields $\q( Z,\dt\w) =\q( X,d\w)$.

Let $K_2\Subset K_0$ be the support of $\psi_1$,
and let $K_1$ be such that $K_2\Subset K_1\Subset K_0$.  Here $A \Subset B$
denotes that $A$ is a relatively compact subset of the interior of $B$.
For a function $f$ defined on $\Omega$, $\delta\leq \xi$, and $x\in \K$, let
\begin{equation*}
A_{\B{X}{d}{\cdot}{\delta}} f\q( x\w) = \frac{1}{\Vol{\B{X}{d}{x}{\delta}} } \int_{\B{X}{d}{x}{\delta}} f\q( y\w) \: dy,
\end{equation*}
where $\Vol{\B{X}{d}{x}{\delta}}$ denotes the induced Lebesgue measure 
of $\B{X}{d}{x}{\delta}$ on the leaf in
which $x$ lies.

We restate Proposition 6.17 of \cite{StreetMultiParameterCCBalls}.
\begin{prop}[Proposition 6.17 of \cite{StreetMultiParameterCCBalls}]\label{PropIntOfAvgs}
There exists a constant $\xi_0>0$, $\xi_0<\xi$, such that
for every $\xi'\leq \xi_0$ and every measurable function $f$ with $f\geq 0$,
we have
\begin{equation*}
\int_{K_2} f\q( x\w) \: dx\lesssim \int_{K_1} A_{\B{X}{d}{\cdot}{\xi'}} f\q( x\w) \: dx \lesssim \int_{\K} f\q( x\w) \: dx,
\end{equation*}
where the implicit constants may depend on a lower bound for $\xi'$.
\end{prop}

We will prove,
\begin{prop}\label{PropLpAvgBound}
Let $\xi'=\frac{\xi_1}{4}\wedge \frac{\xi_0}{2}$, then we have the pointwise
bound for $1<p<\infty$, $x\in \K$,
\begin{equation*}
A_{\B{X}{d}{\cdot}{\xi'}} \q| T f\w|^p \q( x\w) \lesssim A_{\B{X}{d}{\cdot}{2\xi'}} \q|f\w|^p\q( x\w),
\end{equation*}
where the implicit constant may depend on $p$, and we have taken $a>0$
sufficiently small, in the definition of $T$.
\end{prop}
Before we prove Proposition \ref{PropLpAvgBound}, let us first see
how it yields Theorem \ref{ThmSingleParamSingInt}.

\begin{proof}[Proof of Theorem \ref{ThmSingleParamSingInt} given Propositions \ref{PropIntOfAvgs} and \ref{PropLpAvgBound}]
Letting $\xi'$ be as in Proposition \ref{PropLpAvgBound}, we have,
\begin{equation*}
\begin{split}
\LpN{p}{Tf}^p &= \int_{K_2} \q|T f\q( x\w)\w|^p \: dx\\
&\lesssim \int_{K_1} A_{\B{X}{d}{\cdot}{\xi'}} \q| T f\w|^p \q( x\w)\: dx\\
&\lesssim \int_{K_1} A_{\B{X}{d}{\cdot}{2\xi'}} \q|f\w|^p\q( x\w)\: dx\\
&\lesssim \int_{K_0} \q|f\q( x\w)\w|^p \: dx\\
&\lesssim \LpN{p}{f}^p,
\end{split}
\end{equation*}
completing the proof.
\end{proof}

We now turn to the proof of Proposition \ref{PropLpAvgBound}.
It suffices to show that for each $x_0\in \K$, we have,
\begin{equation*}
A_{\B{X}{d}{\cdot}{\xi'}}\q|Tf\w|^p \q( \Phi_{x_0}\q( 0\w)\w) \lesssim A_{\B{X}{d}{\cdot}{2\xi'}}\q|f\w|^p\q( \Phi_{x_0}\q( 0\w)\w),
\end{equation*}
since $\Phi_{x_0}\q( 0\w) = x_0$.

Fix $x_0$ and let $Y_1,\ldots, Y_q$ be the pullbacks of $X_1,\ldots, X_q$
via the map $\Phi_{x_0}$ to $B^{n_0\q( x_0\w)}\q( \eta_1\w)$.  We have,
\begin{lemma}\label{LemmaCZNospan}
For $f\geq 0$ a measurable function defined on $\Omega$,
\begin{equation*}
A_{\B{X}{d}{\cdot}{\xi'}} f \q( x_0\w) \approx \int_{\B{Y}{d}{\cdot}{\xi'}} f\circ \Phi_{x_0} \q( u\w)\: du.
\end{equation*}
\end{lemma}
\begin{proof}
We apply a change of variables $x=\Phi_{x_0}\q( u\w)$, using that $\Phi_{x_0}\q( \B{Y}{d}{0}{\xi'}\w) = \B{X}{d}{x_0}{\xi'}$ and that $\q|\det_{n_0\q( x\w)\times n_0\q( x\w)} d\Phi_{x_0}\q( u\w)\w|\approx \Vol{\B{X}{d}{x_0}{\xi}}$.  See (B.2) of
\cite{StreetMultiParameterCCBalls} for details on this sort of change of variables.  It follows that,
\begin{equation*}
\begin{split}
\int_{\B{Y}{d}{0}{\xi'}} f\q( \Phi_{x_0}\q( u\w)\w) \: du &\approx \frac{1}{\Vol{\B{X}{d}{x_0}{\xi'}}} \int_{\B{X}{d}{x_0}{\xi'}} f\q( x\w) \: dx\\
& =A_{\B{X}{d}{\cdot}{\xi'}} f\q( x_0\w),
\end{split}
\end{equation*}
completing the proof.
\end{proof}

\begin{proof}[Completion of the proof of Proposition \ref{PropLpAvgBound}]
In light of Lemma \ref{LemmaCZNospan}, it suffices to show
\begin{equation*}
\int_{\B{Y}{d}{0}{\xi'}} \q|\q[\Phi_{x_0}^{\#} T \q(\Phi_{x_0}^{\#}\w)^{-1}\w] \Phi_{x_0}^{\#} f \q( u\w)\w|^p\: du \lesssim \int_{\B{Y}{d}{0}{2\xi'}} \q| \Phi_{x_0}^{\#} f\q( u\w)\w|^p\: du.
\end{equation*}
This will follow from,
\begin{equation}\label{EqnCZNospanToShowPulledBack}
\q\| \Phi_{x_0}^{\#} T \q(\Phi_{x_0}^{\#}\w)^{-1}  \w\|_{L^p\q(\B{Y}{d}{0}{\xi'} \w)\rightarrow L^p\q(\B{Y}{d}{0}{2\xi'} \w)}\lesssim 1,
\end{equation}
for $1<p<\infty$, with the implicit constant independent of $x_0$.

To prove \eqref{EqnCZNospanToShowPulledBack}, we apply the theory
in Section \ref{SectionCZSpan} to the operator $ \Phi_{x_0}^{\#} T \q(\Phi_{x_0}^{\#}\w)^{-1} $.
Note that,
\begin{equation*}
 \Phi_{x_0}^{\#} T \q(\Phi_{x_0}^{\#}\w)^{-1} g \q( u\w) = \psi_1\q( \Phi_{x_0}\q( u\w)\w) \int g\q( \gt_t\q( u\w) \w) \psi_2\q( \gt_t\q( u\w)\w) \kappa\q( t,\Phi\q( u\w)\w) K\q( t\w) \: dt,
\end{equation*}
where $\gt_t\q( u\w) = e^{t_1Y_1+\cdots+t_qY_q}u$.
We have, from Theorem \ref{ThmMainCCThm}, that
$\q|\det_{n_0\q( x\w) \times n_0\q( x\w)} Y\q( u\w)\w|\approx 1,$ for
$u\in B^{n_0\q( x\w)}\q( \eta_1\w)$.  That is, that $Y_1,\ldots, Y_q$
span the tangent space (uniformly in $x_0$).
It is easy to see that the methods in Section \ref{SectionCZSpan}
apply to the operator $\Phi_{x_0}^{\#} T \q(\Phi_{x_0}^{\#}\w)^{-1}$
uniformly in $x_0$, establishing \eqref{EqnCZNospanToShowPulledBack}
and completing the proof of Proposition \ref{PropLpAvgBound}.
\end{proof}

The proof of Theorem \ref{ThmSingleParamMax} follows by a simpler reprise
of the above.  See also Section 6.2 of \cite{StreetMultiParameterCCBalls}.

\section{Auxiliary operators}\label{SectionAuxOps}
In this section, we introduce a number of operators, which will be
useful in the proof of
Theorems \ref{ThmMainThmSecondPass} and \ref{ThmMainMaxThm}.
Before we begin, we pick four $C_0^\infty$ cut-off functions
$\denum{\psi}{0},\denum{\psi}{-1},\denum{\psi}{-2},\denum{\psi}{-3}\geq 0$, supported on the interior of $\K$ with
\begin{equation*}
\psi_1,\psi_2\prec \denum{\psi}{0}\prec \denum{\psi}{-1} \prec \denum{\psi}{-2}\prec \denum{\psi}{-3}.
\end{equation*}

In the statement of Theorem \ref{ThmMainThmSecondPass},
we took $K\in \sK\q( N,e,a,\nu,\nu\w)$ (recall, we are first
presenting the proof in the case $\mu_0=\nu$, and in Section \ref{SectionMoreKernels}
will present the necessary modifications to treat general $\mu_0$).  Thus,
\begin{equation*}
K\q( t\w) = \sum_{j\in \N^\nu} \dil{\vsig_j}{2^j}\q( t\w),
\end{equation*}
where $\q\{\vsig_j\w\}\subseteq C_0^\infty \q( \Q^N\q( a\w)\w)$
is a bounded set and the $\vsig_j$ satisfy certain
cancellation conditions (see Section \ref{SectionKernels} for details).
Hence, there is a corresponding decomposition of $T$.
We define, for $j\in \N^\nu$,
\begin{equation*}
T_j f\q( x\w) = \psi_1\q( x\w)  \int f\q( \gamma_t\q( x\w)\w) \psi_2\q( \gamma_t\q( x\w)\w) \kappa\q( t,x\w) \dil{\vsig_j}{2^j}\q( t\w) \: dt.
\end{equation*}
We have,
\begin{equation*}
\sum_{j\in \N^\nu} T_j = T.
\end{equation*}

We now turn to the operators which we will use to construct our Littlewood-Paley
theory.  For each $\mu$, $1\leq \mu\leq \nu$, we obtain a 
list of vector fields with single-parameter formal degrees $\q( X^\mu, d^\mu\w)$,
by letting $X^\mu_1,\ldots, X^\mu_{q_\mu}$ be those vector fields
$X_j$ such that $d_j$ is non-zero in only the $\mu$th component.
We then assign the formal degree to be $d_j^\mu$ (i.e., the value
of the non-zero component).  Using this definition,
\begin{equation*}
\q( \delta_\mu X^\mu, d^\mu\w) = \q( \deltah X, \sd\w),
\end{equation*}
where $\deltah$ is $\delta_\mu$ in the $\mu$th component and $0$
in all other components, and we have suppressed the vector fields
that are equal to $0$.  As a consequence, $\q( X^\mu, d^\mu\w)$
satisfies $\sD\q( \K, \q[0,1\w], \Omega', \xi\w)$, since
$\q( X,d\w)$ satisfies $\sD\q( \K, \q[0,1\w]^\nu, \Omega', \xi\w)$.

We define (single-parameter) dilations on $\R^{q_\mu}$ by,
\begin{equation}\label{EqnRqmuDil}
\delta \q( t_1,\ldots, t_{q_\mu}\w) = \q(\delta^{d_1^\mu}t_1, \ldots,\delta^{d_{q_\mu}^\mu} t_{q_\mu} \w).
\end{equation}
Let $\phi_\mu\in C_0^\infty\q( \Q^{q_\mu}\q( a\w) \w)$ be such that
$\int \phi_\mu =1$, and assume $\phi_\mu\geq 0$.  Define,
\begin{equation*}
\phi_{\mu,j} =
\begin{cases}
\phi_\mu & \text{if }j=0,\\
\dil{\phi_\mu}{2}-\phi_\mu & \text{if }j>0.
\end{cases}
\end{equation*}
Here, as usual,
$\dil{\phi_\mu}{2^{j}}\q( t\w) = 2^{j\q( d_1^\mu+\cdots+d_{q_\mu}^\mu\w)}\phi_\mu\q( 2^j t\w)$.  Define,
\begin{equation*}
\gh_{\q( t_1,\ldots, t_{q_\mu}\w)}^\mu \q( x\w) = e^{t_1 X_1^\mu + \cdots + t_{q_\mu} X_{q_\mu}^\mu }x.
\end{equation*}

For $j\in \N$, define,
\begin{equation*}
D_j^\mu f\q( x\w) = \denum{\psi}{-3}\q( x\w) \int f\q(\gh_t^\mu\q( x\w)\w) \denum{\psi}{-3}\q( \gh_t^\mu\q( x\w) \w) \dil{\phi_{\mu,j}}{2^j}\q( t\w)\: dt;
\end{equation*}
so that
$\sum_{j\in \N} D_j^\mu = \denum{\psi}{-3}^2.$
For $j=\q( j_1,\ldots, j_\nu\w) \in \N^\nu$, define,
\begin{equation}\label{EqnDefnD}
D_j = D_{j_1}^1 D_{j_2}^2\cdots D_{j_\nu}^\nu,
\end{equation}
so that,
\begin{equation*}
\sum_{j\in \N^\nu} D_j = \denum{\psi}{-3}^{2\nu}.
\end{equation*}
In Section \ref{SectionSquareFunc}, we will use the operators $D_j$
to create an appropriate Littlewood-Paley square function.

Now we turn to the operators which will be at the
basis of the study of the maximal function.
The study of the maximal function will proceed
by induction on the number of parameters ($\nu$),
with the base case being the trivial
case $\nu=0$ (we will explain this more in what follows).
In what follows, we introduce operators that will
facilitate this induction.

Let $\Ninf = \N\cup \q\{\infty\w\}$.
For a subset $E\subseteq \nuset$ and 
$j=\q( j_1,\ldots, j_\nu\w) \in \N^\nu$,
define $j_E\in \Ninf^\nu$ to be equal to $j_\mu$
in those components $\mu\in E$, and equal to $\infty$
in the rest of the components.
For $t\in \R^N$, we dilate $2^{-j_E} t$ in the usual way,
where we identify $2^{-\infty}=0$; thus,
$2^{-j_E} t$ is zero in every coordinate $t_j$ such
that $e_j^\mu\ne 0$ for some $\mu\in E^{c}$.
We may think of these dilations as $\q|E\w|$-parameter
dilations acting on a lower dimensional space
consisting of those coordinates which are not mapped
to $0$ under this dilation.
Notice that $j_{\nuset}=j$ and $j_{\emptyset} =\q(\infty,\infty,\cdots,\infty \w)$.

Let $\sigma\in C_0^\infty\q( \Q^N\q( a\w)\w)$ satisfy
$\sigma\geq 0$ and $\sigma\geq 1$ on a neighborhood of $0$.
We assume, further, that $\sigma$ is of the form,
\begin{equation*}
\sigma\q( t_1,\ldots, t_N\w) = \sigma_0\q( t_1\w)\cdots \sigma_0\q( t_N\w).
\end{equation*}
where $\sigma_0\in C_0^\infty\q( \R\w)$, is supported near $0$, is $\geq 0$,
and is $\geq 1$ on a neighborhood of $0$.
We define for $j\in \Ninf^\nu$,
\begin{equation*}
M_{j} f\q( x\w) = \denum{\psi}{0}\q( x\w) \int f\q(\gamma_{2^{-j}t }\q( x\w) \w) \denum{\psi}{0}\q( \gamma_{2^{-j} t}\q(x\w)\w) \sigma\q( t\w) \: dt.
\end{equation*}
Notice, 
\begin{equation}\label{EqnMemptyset}
M_{j_{\emptyset}} f\q( x\w) = \denum{\psi}{0}^2\q( x\w) \q[\int \sigma\q( t\w) \: dt \w] f\q( x\w).
\end{equation}
It is immediate to see,
\begin{equation}\label{EqnsMBoundDisc}
\sM f\q( x\w) \lesssim \sup_{j\in \N^\nu} M_j \q|f\w|\q( x\w) +  \denum{\psi}{0}\q( x\w) \int_{\q|t\w|\leq a} \q|f\q( \gamma_t\q( x\w)\w)\w|\denum{\psi}{0}\q( \gamma_t\q( x\w)\w) \: dt.
\end{equation}
The second term on the left hand side of \eqref{EqnsMBoundDisc}
is easy to control, and so to prove
Theorem \ref{ThmMainMaxThm}, it suffices to prove
the following proposition.
\begin{prop}\label{PropDiscMaxBound}
\begin{equation*}
\LpN{p}{\sup_{j\in \N^{\nu}} \q|M_j f\w|}\lesssim \LpN{p}{f},
\end{equation*}
for $1<p<\infty$.
\end{prop}

Indeed, to deduce Theorem \ref{ThmMainMaxThm} merely apply Proposition
\ref{PropDiscMaxBound} to $\q|f\w|$
and use \eqref{EqnsMBoundDisc}.
The difficulty in Proposition \ref{PropDiscMaxBound} is that,
unlike the operators $T_j$, the operators
$M_j$ do not have any cancellation
to take advantage of.  We now turn to reducing
Proposition \ref{PropDiscMaxBound} to an equivalent result
where there will be cancellation
to take advantage of.

We begin by explaining our induction.  Given $E\subseteq \nuset$,
separate $t\in \R^N$ into two variables $t=\q( t_1^E, t_2^E\w)$:
$t_2^E$ will be those coordinates that are mapped to $0$ under
$2^{-j_E}t$, and $t_1^E$ will be the rest of the coordinates.
I.e., $t_2^E$ are those coordinates $t_j$ such that $e_j^\mu\ne 0$
for some $\mu\in E^c$.
With an abuse of notation, we write $2^{-j_E} t_1^E$ as the $t_1^E$
coordinate of $2^{-j_E} t$, and so $2^{-j_E} t_1^E$ defines
$\q|E\w|$-parameter dilations on $t_1^E$.
Furthermore, with another abuse of notation, we write
$\sigma\q( t\w) = \sigma\q( t_1^E\w)\sigma\q( t_2^E\w)$,
where $\sigma\q( t_1^E\w)$ is a product of $\sigma_0\q( t_j\w)$
such that $t_j$ is a coordinate of $t_1^E$, and similarly
for $t_2^E$.
We may rewrite $M_{j_E}$ as follows,
\begin{equation*}
M_{j_E} f\q( x\w) = \q[\denum{\psi}{0}\q( x\w) \int f\q( \gamma_{2^{-j_E}t_1^E}\q( x\w) \w) \denum{\psi}{0}\q(\gamma_{2^{-j_E} t_1^E}\q( x\w)  \w) \sigma\q( t_1^E\w) \: dt_1^E\w] \q[\int \sigma\q(t_2^E \w)\: dt_2^E\w].
\end{equation*}

The term $\int \sigma\q(t_2^E \w)\: dt_2^E$ is a constant.  It is 
easy to see from our assumptions that
$\gamma_{2^{-j_E} t_1^E}$ is of the same form as $\gamma_{2^{-j}t}$
with $\nu$ replaced by $\q|E\w|$.  I.e., $\gamma_{t_1^E}$ satisfies
the hypotheses of Theorem \ref{ThmMainMaxThm} with $\nu$
replaced by $\q|E\w|$.
As a consequence, $M_{j_E}$ is a constant times an operator
of the same form as $M_j$, with $\nu$ replaced
by $\q|E\w|$.

We will prove Proposition \ref{PropDiscMaxBound} by induction
on $\nu$.  Due to the above discussion, our inductive hypothesis implies,
\begin{equation}\label{EqnInductHypo}
\LpN{p}{\sup_{j\in \N^{\nu}}\q|M_{j_E} f\w| }\lesssim \LpN{p}{f},
\end{equation}
for $E\subsetneq \nuset$ and $1<p\leq\infty$.
The base case of our induction will correspond to $E=\emptyset$.
In light of \eqref{EqnMemptyset}, the base case is trivial.

For each $\mu$, $1\leq \mu\leq \nu$, and each $j\in \Ninf$, define
the operator,
\begin{equation*}
A_j^\mu f\q( x\w) = \denum{\psi}{-1}\q( x\w) \int_{t\in\R^{q_\mu}} f\q(\gh_{2^{-j}t}^\mu\q( x\w) \w) \denum{\psi}{-1}\q(\gh_{2^{-j}t}^\mu\q( x\w) \w)\sigma\q( t\w) \: dt,
\end{equation*}
where we have used the dilations on $\R^{q_\mu}$ defined
in \eqref{EqnRqmuDil} and we have identified $2^{-\infty}=0$;
so that $A_\infty^\mu  = \q[\int \sigma\q( t\w)\: dt\w]\denum{\psi}{-1}^2$.
Here we have abused notation and viewed $\sigma$ as a function on
$\R^{q_{\mu}}$.  By this we mean,
$\sigma\q( t_1,\ldots, t_{q_\mu}\w)= \prod_{j=1}^{q_{\mu}} \sigma_0\q( t_j\w)$.
Define the maximal operator,
\begin{equation*}
\sM^\mu f\q( x\w) = \sup_{\delta\in \q[0,1\w]} \denum{\psi}{-3}\q( x\w) \int_{\q|t\w|\leq a} \q|f\q(\gh_{\delta t}^\mu\q( x\w)\w)\w| \denum{\psi}{-3}\q(\gh_{\delta t}^\mu\q( x\w) \w) \: dt.
\end{equation*}
Note that Theorem \ref{ThmSingleParamMax} shows that,
\begin{equation*}
\LpN{p}{\sM^\mu f}\lesssim \LpN{p}{f}, \quad 1<p\leq \infty.
\end{equation*}
Also it is elementary to verify the pointwise inequality
\begin{equation}\label{EqnSupAjBound}
\sup_{j\in \Ninf} \q|A_j^\mu f\q( x\w) \w|\lesssim \sM^\mu f\q( x\w),
\end{equation}
and so we have,
\begin{equation*}
\LpN{p}{\sup_{j\in \Ninf} \q|A_j^\mu f\q( x\w)\w|}\lesssim \LpN{p}{f}, \quad 1<p\leq \infty.
\end{equation*}

For $j=\q(j_1,\ldots, j_\nu \w)\in \Ninf^\nu$ define
\begin{equation}\label{EqnDefnA}
A_j = A_{j_1}^1 A_{j_2}^2\cdots A_{j_\nu}^\nu.
\end{equation}
Notice that 
\begin{equation*}
A_{\q(\infty, \infty,\cdots, \infty \w)} = \q[\int \sigma\q( t\w) \: dt\w]^{\nu} \denum{\psi}{-1}^{2\nu}.
\end{equation*}
And since $\denum{\psi}{-1} M_j= M_j=M_{j_{\nuset}}$, we see
to prove Proposition \ref{PropDiscMaxBound} it suffices to prove,
\begin{equation}\label{EqnMaxSTSWithA}
\LpN{p}{\sup_{j\in \N}\q|A_{j_\emptyset} M_{j_{\nuset}}f\w|}\lesssim \LpN{p}{f},\quad 1\leq p\leq \infty.
\end{equation}

For $E\subsetneq \nuset$, combining \eqref{EqnInductHypo} and \eqref{EqnSupAjBound}, we see,
\begin{equation}\label{EqnInductHypWithA}
\LpN{p}{\sup_{j\in \N} \q|A_{j_{E^c}} M_{j_{E}} f \w| }\lesssim \LpN{p}{f},\quad 1<p\leq \infty.
\end{equation}
For $j\in \N^\nu$, define the operator,
\begin{equation*}
B_j = \sum_{E\subseteq \nuset} \q( -1\w)^{\q|E\w|} A_{j_{E^c}} M_{j_E}.
\end{equation*}
From \eqref{EqnInductHypWithA} we see that to prove \eqref{EqnMaxSTSWithA}
(and hence to prove Proposition \ref{PropDiscMaxBound} and Theorem \ref{ThmMainMaxThm}) 
it suffices
to prove,
\begin{prop}
\begin{equation*}
\LpN{p}{\sup_{j\in \N^\nu} \q|B_j f\w|}\lesssim \LpN{p}{f},
\end{equation*}
for $1<p\leq \infty$.
\end{prop}

\section{Preliminary $L^2$ results}
In this section we describe some $L^2$ results concerning
the operators defined in Section \ref{SectionAuxOps}.
These results, along with the results in Section \ref{SectionCZOps},
make up the main technical results on which our theory is based.
All of the results in this section will follow
from the results
in \cite{StreetMultiParameterSingRadonLt} (after some reductions).

\begin{thm}\label{ThmL2Thm}
For $j_1,\ldots, j_r\in \N^\nu$, define,
\begin{equation*}
\diam{j_1,\ldots, j_r} = \max_{1\leq l,m\leq r} \q|j_l-j_m\w|.
\end{equation*}
If we take $a>0$ sufficiently small,\footnote{Recall, all
of the operators in Section \ref{SectionAuxOps} were defined in terms
of some small $a>0$.}
 then there exists $\epsilon_2>0$ such that,
\begin{itemize}
\item $\LpOpN{2}{B_{j_1}D_{j_2}}\lesssim 2^{-\epsilon_2\diam{j_1,j_2}}$,
\item $\LpOpN{2}{D_{j_1}T_{j_2}D_{j_3}}\lesssim 2^{-\epsilon_2 \diam{j_1,j_2,j_3}}$,
\item $\LpOpN{2}{D_{j_1}^{*}D_{j_2}^{*}D_{j_3}D_{j_4}}\lesssim 2^{-\epsilon_2\diam{j_1,j_2,j_3,j_4}}$,
\item $\LpOpN{2}{D_{j_1}D_{j_2}D_{j_3}^{*}D_{j_4}^{*}}\lesssim 2^{-\epsilon_2\diam{j_1,j_2,j_3,j_4}}$.
\end{itemize}
Here, $j_1,j_2,j_3$, and $j_4$ are arbitrary elements of $\N^\nu$.
\end{thm}

The rest of this section is devoted to the proof of Theorem \ref{ThmL2Thm}.
We will see that each part of Theorem \ref{ThmL2Thm} follows from
an application of the same general result.  This result
is proved in \cite{StreetMultiParameterSingRadonLt},
and we review the statement of the result in
in Section \ref{SectionGenL2}.  In Section
\ref{SectionL2Reduce} we show how to reduce each part of
Theorem \ref{ThmL2Thm} to the result in Section \ref{SectionGenL2}.

\begin{rmk}
Using methods similar to the ones in this section, one
can prove $\LpOpN{2}{T_{j_1}^{*}T_{j_2}}, \LpOpN{2}{T_{j_1}T_{j_2}^{*}}\lesssim 2^{-\epsilon \diam{j_1,j_2}}$.
This shows, via the Cotlar-Stein lemma, that $T$ is bounded on $L^2$.
This is the proof used in \cite{StreetMultiParameterSingRadonLt}.
\end{rmk}

	\subsection{A general $L^2$ result}\label{SectionGenL2}
In this section, we review the main technical result
from \cite{StreetMultiParameterSingRadonLt}; this result
will imply Theorem \ref{ThmL2Thm}.

The setting is as follows.  We are given operators
$S_1,\ldots, S_L$, and $R_1$, $R_2$, and a real number
$\zeta\in \q[0,1\w]$.  We will present conditions
on these operators such that there
exists $\epsilon>0$ with,
\begin{equation}\label{EqnToShowGenL2}
\LpOpN{2}{S_1\cdots S_L \q( R_1-R_2\w)}\lesssim \zeta^\epsilon.
\end{equation}
In Section \ref{SectionL2Reduce}, we will show that
the assumptions of this section hold uniformly
in $j_1,j_2,j_3,j_4$ for the operators
in Theorem \ref{ThmL2Thm} (with an appropriate choice
of $\zeta$), and Theorem \ref{ThmL2Thm} will follow.
The term $R_1-R_2$ is how we make use of the cancellation
implicit in the operators in Theorem \ref{ThmL2Thm}.

To describe the operators above, suppose we are given
$C^\infty$ vector fields $Z_1,\ldots, Z_q$ on $\Omega$
with {\it single}-parameter formal degrees
$\dt_1,\ldots, \dt_q$.  We will be taking
$\q( Z,\dt\w) = \q( \delta X, \sd\w)$ for some
$\delta\in \q[0,1\w]^\nu$, to prove Theorem \ref{ThmL2Thm}.
We assume that $\q( Z,\dt\w)$ satisfies the
assumptions of Theorem \ref{ThmMainCCThm} uniformly
for $x_0\in \K$, for some fixed $\xi>0$.

We assume that $r$ of the vector fields $Z_1,\ldots, Z_r$
generate $Z_1,\ldots, Z_q$ in the sense that there is an\footnote{The implicit
constants in \eqref{EqnToShowGenL2} may depend on
$M_1$ and $\xi$.}
 $M_1$
such that for every $j$, $r+1\leq j\leq q$,
$Z_j$ may be written in the form,
\begin{equation*}
Z_j = \ad{Z_{l_1}} 
\ad{Z_{l_2}} \cdots \ad{Z_{l_m}} Z_{l_{m+1}}, \quad 1\leq l_k\leq r, \quad 0\leq m\leq M_1-1.
\end{equation*}

\begin{defn}
Let $\gh:\Q^N\q( \rho\w)\times \Omega''\rightarrow \Omega$ be a
$C^\infty$ function, satisfying $\gh_0\q( x\w) \equiv x$.  We say
that $\gh$ is controlled by $\q( Z,\dt\w)$ at the unit scale
if the following holds.
Define the vector field $\Wh\q( t,x\w)$ by,
\begin{equation*}
\Wh\q( t,x\w) = \frac{d}{d\epsilon}\bigg|_{\epsilon=1} \gh_{\epsilon t}\circ \gh_{t}^{-1}\q( x\w).
\end{equation*}
We suppose, there exist $\rho_1,\tau_1>0$ such that for every $x_0\in \K$,
\begin{itemize}
\item $\Wh\q( t,x\w) = \sum_{l=1}^q c_l\q( t,x\w) Z_l\q( x\w)$, on $\B{Z}{\dt}{x_0}{\tau_1}$,
\item $\sum_{\q|\alpha\w|+\q|\beta\w|\leq m} \CjN{Z^\alpha \partial_t^\beta c_l}{0}{\Q^N\q(\rho_1\w)\times \B{Z}{\dt}{x_0}{\tau_1}}<\infty$, for every $m$.
\end{itemize}
\end{defn}

\begin{rmk}
Note that the assumption that $\q( X,d\w)$ controls $\gamma$ can
be restated as $\q( \delta X, \sd\w)$ controls $\gamma_{\delta t}$
at the unit scale
for every $\delta\in \q[0,1\w]^\nu$, uniformly in $\delta$.
\end{rmk}

We now turn to defining the operators $S_j$.  We assume, for each $j$,
we are given a $C^\infty$ function $\gh_j:\Q^{N_j}\q(\rho\w)\times \Omega''\rightarrow \Omega$ with $\gh_j\q(0,x\w)\equiv x$, and that this
function is controlled by $\q( Z,\dt\w)$ at the unit scale.
As usual, we restrict our attention to $\rho>0$ small,
so that $\gh_{j,t}^{-1}$ makes sense wherever we use it.
We suppose we are given $\psi_{j,1},\psi_{j,2}\in C_0^\infty\q( \R^n\w)$
supported on the interior of $\K$ and
$\kappa_j\in C^\infty\q(\overline{\Q^{N_j}\q( a\w)}\times \overline{\Omega' } \w)$.
Finally, we suppose we are given $\vsig_j\in C_0^\infty\q( \Q^{N_j}\q( a\w)\w)$.
We define,
\begin{equation*}
S_j f\q( x\w) = \psi_{j,1}\q( x\w) \int f\q( \gh_{j,t}\q( x\w)\w) \psi_{j,2}\q(\gh_{j,t}\q( x\w)\w) \kappa_j\q( t,x\w) \vsig_j\q( t\w)\: dt.
\end{equation*}

\begin{defn}
If $S_j$ is of the above form, we say $S_j$ is controlled by $\q( Z,\dt\w)$
at the unit scale.
\end{defn}

\begin{rmk}\label{RmkAdjointControl}
If $S_j$ is controlled by $\q( Z,\dt\w)$ at the unit scale, then so
is $S_j^{*}$.  This is shown in
\cite{StreetMultiParameterSingRadonLt}.
Furthermore, a simple change of variables shows that if $S_j$
is controlled at the unit scale by $\q( Z,\dt\w)$, then
$\LpOpN{2}{S_j}\lesssim 1$.
\end{rmk}

We assume further, that for each $l$, $1\leq l\leq r$, there is a $j$,
$1\leq j\leq L$, and a multi index $\alpha$ (with $\q|\alpha\w|\leq M_2$
for some\footnote{The implicit constants
in \eqref{EqnToShowGenL2} are allowed to depend on $M_2$.}
$M_2$), such that,
\begin{equation*}
Z_l\q( x\w) = \frac{1}{\alpha!} \frac{\partial}{\partial t}^\alpha\bigg|_{t=0} \frac{d}{d\epsilon}\bigg|_{\epsilon=1} \gh_{j,\epsilon t}\circ \gh_{j,t}^{-1}\q( x\w).
\end{equation*}
This concludes our assumptions on $S_1,\ldots, S_L$.

We now turn to the operators $R_1$ and $R_2$.  It is here
where $\zeta$ plays a role.
We assume we are given a $C^\infty$ function $\gt_{t,s}$ 
which is controlled by $\q( Z,\dt\w)$ at the unit scale:
\begin{equation*}
\gt_{t,s}\q( x\w) : \Q^{\Nt}\q( \rho\w)\times \q[-1,1\w]\times \Omega''\rightarrow \Omega, \quad \gt_{0,0}\q( x\w) \equiv x.
\end{equation*}

\begin{rmk}
Here we are thinking of $\q( t,s\w)$ as playing the role of the
$t$ variable in the definition of control.
\end{rmk}

We suppose we are given $\kapt\q( t,s,x\w) \in C^{\infty}\q(\overline{\Q^{\Nt}\q( a\w)} \times \q[-1,1\w]\times \Omega'' \w)$,
$\vsigt\in L^1\q( \Q^N\q( a\w)\w)$, and $\psit_1,\psit_2\in C^\infty_0\q( \R^n\w)$ supported on the interior of $\K$.  We define, for $\xi\in \q[-1,1\w]$,
\begin{equation*}
R^{\xi} f\q(x\w) = \psit_1\q( x\w) \int f\q( \gt_{t,\xi}\q( x\w)\w) \psit_2\q( \gt_{t,\xi}\q( x\w)\w) \kapt\q( t,\xi,x\w) \vsigt\q( t\w)\: dt.
\end{equation*}
We set $R_1=R^{\zeta}$ and $R_2=R^{0}$.

\begin{thm}[Theorem 14.5 of \cite{StreetMultiParameterSingRadonLt}]\label{ThmGenL2Thm}
In the above setup, if $a>0$ is sufficiently small, we have,
\begin{equation*}
\LpN{2}{S_1\cdots S_L \q( R_1-R_2\w)}\lesssim \zeta^{\epsilon},
\end{equation*}
for some $\epsilon>0$.
\end{thm}

\begin{rmk}\label{RmkUnifAssump}
It is important in our applications of Theorem \ref{ThmGenL2Thm}
that the various constants can be chosen independent
of any relevant parameters.
I.e., that if all of the hypotheses of this section
hold ``uniformly'' then so does Theorem \ref{ThmGenL2Thm}.
Indeed, this is the case, and is discussed further
and made precise in \cite{StreetMultiParameterSingRadonLt}.  
In this paper, we merely say that
in our proof of Theorem \ref{ThmL2Thm}, all of our applications
of Theorem \ref{ThmGenL2Thm} will satisfy the hypotheses of
this section uniformly in the appropriate sense, and we
leave the straight-forward verification of this fact to the reader.
\end{rmk}

	\subsection{Reduction to the general $L^2$ result}\label{SectionL2Reduce}
This section is devoted
to proving Theorem \ref{ThmL2Thm} by applying
Theorem \ref{ThmGenL2Thm}.
We will be implicitly choosing $a>0$ by choosing it small
enough that Theorem \ref{ThmGenL2Thm} applies.
Since the assumptions of Theorem \ref{ThmGenL2Thm} will hold
uniformly in $j_1,j_2,j_3,j_4$, we will have that $a>0$, $\epsilon>0$, and
the implicit constant in Theorem \ref{ThmGenL2Thm} can all
be chosen independent of $j_1,j_2,j_3,j_4\in \N^\nu$.
See Remark \ref{RmkUnifAssump} and \cite{StreetMultiParameterSingRadonLt} for more details on this.

Recall the list of vector fields
$\q( X,d\w) = \q( X_1,d_1\w),\ldots, \q( X_q,d_q\w)$
satisfying $\sD\q( \K, \q[0,1\w]^\nu\w)$
defined in Section \ref{SectionCurves}.

Our assumptions on $\gamma$ can be restated (by
possibly reordering $\q( X_1,d_1\w),\ldots, \q( X_q,d_q\w)$)
as there exists $r\leq q$ such that,
\begin{enumerate}
\item $\q( X,d\w)$ controls $\gamma$.
\item For $1\leq l\leq r$, $d_l$ is nonzero in only one component.
\item Every $\q( X_j,d_j\w)$, with $r<j\leq q$, can be written as,
\begin{equation*}
X_j = \ad{X_{l_1}}\ad{X_{l_2}}\cdots \ad{X_{l_m}}X_{l_{m+1}},
\end{equation*}
\begin{equation*}
d_j= d_{l_1}+d_{l_2}+\cdots+d_{l_{m+1}},
\end{equation*}
with $1\leq l_k\leq r$.

\item Every $\q( X_l, d_l\w)$, $1\leq l\leq r$, is of the form,
\begin{equation}\label{EqnHowTheXjsCome}
X_l = \frac{1}{\alpha!}\frac{\partial}{\partial t}^{\alpha}\bigg|_{t=0} \frac{d}{d\epsilon}\bigg|_{\epsilon=1} \gamma_{\epsilon t}\circ \gamma_t^{-1}\q( x\w),
\end{equation}
\begin{equation}\label{EqnHowThedjsCome}
d_l = \deg\q( \alpha\w).
\end{equation}
\end{enumerate}

We now describe how the above assumptions come into play in what follows.
Take $\delta\in \q[0,1\w]$.  Define the list of vector
fields with single-parameter formal degree $\q( Z,\dt\w)=\q( \delta X,\sd\w)$.
Note that $Z_1,\ldots, Z_r$ generate $Z_1,\ldots, Z_q$, in the
sense of that every $Z_j$ ($r+1\leq j\leq q$) can be written in the form,
\begin{equation*}
Z_j = \ad{Z_{l_1}}\ad{Z_{l_2}}\cdots \ad{Z_{l_m}}Z_{l_{m+1}},
\end{equation*}
with $1\leq l_k\leq r$ for every $k$.

Fix $\delta_1=\q( \delta_1^1,\ldots, \delta_{1}^\nu\w), \delta_2=\q( \delta_2^1,\ldots, \delta_2^\nu\w)\in \q[0,1\w]^\nu$ and assume
$\delta_1^\mu\leq \delta_2^\mu$ for every $\mu$.
Then, using the fact that $\q( X,d\w)$ controls $\gamma$, 
we have $\q( \delta_2 X, \sd\w)$ controls $\gamma_{\delta_1 t}$
at the unit scale, uniformly in $\delta_1,\delta_2$.
Furthermore, suppose that $\delta_1^\mu=\delta_2^\mu$ for some
fixed $\mu$.  Suppose further that for $j_0$ fixed ($j_0\leq r$),
$d_{j_0}$ is nonzero in only the $\mu$th coordinate.
Define $\gt_t = \gamma_{\delta_1 t}$.  We then have,
\begin{equation*}
\delta_2^{d_{j_0}} X_{j_0} = \delta_1^{d_{j_0}} X_{j_0} 
=\frac{1}{\alpha!}\frac{\partial}{\partial t}^{\alpha}\bigg|_{t=0} \frac{d}{d\epsilon}\bigg|_{\epsilon=1} \gt_{\epsilon t}\circ \gt_t^{-1}\q( x\w),
\end{equation*}
for some $\alpha$.

We now turn to the proof of Theorem \ref{ThmL2Thm}.  We describe, in detail,
the proof for $B_j D_k$ as it is the most complicated (here we have
replaced $j_1$ with $j$ and $j_2$ with $k$ for notational convenience).
We then indicate the modifications necessary to study all of the other operators.
Let $\jz=j\wedge k$.
Define $\q( Z,d\w) = \q( 2^{-\jz} X,\sd\w)$, and
$\iinf = \q|j-k\w|_\infty$.  Our goal is to show that
there exists $\epsilon>0$ such that,
\begin{equation}\label{EqnToShowBjDk}
\LpOpN{2}{B_j D_k}\lesssim 2^{-\epsilon \iinf}.
\end{equation}
Note that this is trivial if $\iinf=0$ and so we assume $\iinf>0$
in what follows.

We separate the proof into two cases.  The first case
is when there is a $\mu_1$ ($1\leq \mu_1\leq \nu$)
such that $\iinf = k_{\mu_1}-j_{\mu_1}$.
In this case, we need only use the cancellation in the operator $D_k$,
and so it suffices to prove,
\begin{equation}\label{EqnToShowL21}
\LpOpN{2}{A_{j_{E^c}} M_{j_{E}} D_{k}}\lesssim 2^{-\epsilon \iinf},
\end{equation}
for every $E\subseteq \nuset$.

To prove \eqref{EqnToShowL21}, it suffices to prove,
\begin{equation}\label{EqnToShowL22}
\LpOpN{2}{ \q[ D_{k}^{*} M_{j_E}^{*} A_{j_{E^c}}^{*} A_{j_{E^c}} M_{j_E} D_{k}   \w]^2 }\lesssim 2^{-\epsilon\iinf},
\end{equation}
where we have changed $\epsilon$.

Using that $\LpOpN{2}{D_k^{*}}, \LpOpN{2}{M_{j_E}^{*}}, \LpOpN{2}{A_{j_{E^c}}^{*}}\lesssim 1$, to prove \eqref{EqnToShowL22} it suffices to show,
\begin{equation}\label{EqnToShowL23}
\LpOpN{2}{A_{j_{E^c}} M_{j_E} D_k D_k^{*} M_{j_E}^{*}A_{j_{E^c}}^{*}  A_{j_{E^c}} M_{j_E} D_k}\lesssim 2^{-\epsilon \iinf}.
\end{equation}
We now expand the last $D_k$ in \eqref{EqnToShowL23} into
$D_k = D_{k_1}^1 D_{k_2}^2\cdots D_{k_\nu}^\nu$.
Using that $\LpOpN{2}{D_{k_\mu}^\mu}\lesssim 1$ for every $\mu$,
to prove \eqref{EqnToShowL23} it suffices to show,
\begin{equation}\label{EqnToShowL24}
\LpOpN{2}{A_{j_{E^c}} M_{j_E} D_k D_k^{*} M_{j_E}^{*}A_{j_{E^c}}^{*}  A_{j_{E^c}} M_{j_E} D_{k_1}^1 D_{k_2}^2\cdots D_{k_{\mu_1}}^{\mu_1}}\lesssim 2^{-\epsilon \iinf}.
\end{equation}

We will prove \eqref{EqnToShowL24} by applying Theorem \ref{ThmGenL2Thm}
with,
\begin{equation*}
S_1\cdots S_L = A_{j_{E^c}} M_{j_E} D_k D_k^{*} M_{j_E}^{*}A_{j_{E^c}}^{*}  A_{j_{E^c}} M_{j_E} D_{k_1}^1 D_{k_2}^2\cdots D_{k_{\mu_1-1}}^{\mu_1-1},
\end{equation*}
\begin{equation*}
R_1=D_{j_{\mu_1}}^{\mu_1}, \quad R_2=0.
\end{equation*}
In the above, we are thinking of $D_{j}$ and $A_{j_{E^{c}}}$
as a product of terms (see \eqref{EqnDefnD} and \eqref{EqnDefnA}),
and assigning an $S_l$ to each term in the product, similarly
for the adjoints.

First we verify that each $S_l$ is controlled at the unit
scale by $\q( Z,\dt\w)$.
We will begin by showing that $A_{j_\mu}^\mu$, $A_\infty^\mu$, and $D_{k_\mu}^\mu$
are controlled at the unit scale.  We will also
show that $M_{j_E}^E$ is controlled at the unit scale.
It will then follow that $D_{k}^{*}$, $M_{j_E}^{*}$,
and $A_{j_{E^c}}^{*}$ are all products of operators
which are controlled at the unit scale, since if
$S_l$ is controlled at the unit scale, so is $S_l^{*}$ (Remark \ref{RmkAdjointControl}).

Consider,
\begin{equation*}
A_{j_\mu}^{\mu} f\q( x\w) =\denum{\psi}{-1}\q( x\w) \int_{t\in \R^{q_\mu}} f\q( \gh_{2^{-j_\mu} t}^\mu \q( x\w) \w) \denum{\psi}{-1}\q( \gh_{2^{-j_\mu}t}^\mu\q( x\w)\w) \sigma\q( t\w) \: dt,
\end{equation*}
and so to show that $A_{j_\mu}^\mu$ is controlled at the unit scale by $\q( Z,\dt\w)$,
it suffices to show that $\gh_{2^{-j_\mu}}^\mu$ is controlled
at the unit scale by $\q( Z,\dt\w)$.
However,
\begin{equation*}
\gh_{2^{-j_\mu} t}^\mu \q( x\w) = \exp\q( 2^{-j_\mu d_1^\mu} t_1 X_1^\mu  + \cdots + 2^{-j_\mu d_{q_\mu}^\mu} t_{q_\mu} X_{q_\mu}^\mu \w)x.
\end{equation*}
By definition, the list of vector fields $\q( 2^{-j_\mu} X^\mu, d^\mu\w)$
is a sublist of the list of vector fields
$\q( 2^{-j} X,\sd\w)$.  Since $j\geq \jz$ coordinatewise, it follows immediately from
Lemma 12.18 of \cite{StreetMultiParameterSingRadonLt}
that $\gh_{2^{-j_\mu} t}^{\mu}$ is controlled at the unit scale
by $\q( Z,d\w) = \q( 2^{-\jz} X,\sd\w)$.
It is trivial that $A_{\infty}^\mu$ is controlled at the unit scale
by $\q( Z,\dt\w)$ since $\gh_t\q( x\w) \equiv x$ is controlled
at the unit scale by $\q(Z,\dt\w)$, trivially.
The proof that $D_{k_\mu}^\mu$ is controlled at the unit scale
by $\q( Z,\dt\w)$ follows just as the proof for $A_{j_\mu}^\mu$.

We now turn to $M_{j_E}$.  Since,
\begin{equation*}
M_{j_E} f\q( x\w) = \denum{\psi}{0}\q( x\w) \int f\q( \gamma_{2^{-j_E}t}\q( x\w)\w) \denum{\psi}{0}\q( \gamma_{2^{-j_E}t}\q( x\w)\w) \sigma\q( t\w)\: dt,
\end{equation*}
we need only show that $\gamma_{2^{-j_E}t}$ is controlled at the unit
scale by $\q( Z,\dt\w)$.  As discussed at the beginning of this section,
$\gamma_{2^{-j}t}$ is controlled by $\q( Z,\dt\w)$.  Since
$\gamma_{2^{-j_E}t}$ is the same as $\gamma_{2^{-j}t}$ except
with some of the coordinates of $t$ set to $0$, the result follows.

This completes the proof that each $S_p$ is controlled at the unit
scale by $\q( Z,\dt\w)$.
To complete our discussion of the $S_p$, we need to show that for
each $l$, $1\leq l\leq r$, $Z_l$ is of the form,
\begin{equation}\label{EqnToShowZAppears}
Z_l\q( x\w) = \frac{1}{\alpha!}\frac{\partial}{\partial t}^\alpha\bigg|_{t=0}\frac{d}{d\epsilon}\bigg|_{\epsilon=0} \gt_{\epsilon t}\circ \gt_t^{-1}\q( x\w),
\end{equation} 
for some $\alpha$, where $\gt$ is one of the functions
defining the maps $S_1,\ldots, S_L$.

Fix $l$, $1\leq l\leq r$.  Recall, $Z_l= 2^{-\jz \cdot d_l} X_l$,
and $d_l$ is nonzero in precisely one component.
Let us suppose that $d_l$ is nonzero in only the $\mu_2$
component.  Thus, $Z_l= 2^{-\jz_{\mu_2} d_l^{\mu_2}} X_l$.
There are two possibilities.  Either $\jz_{\mu_2}= j_{\mu_2}$
or $\jz_{\mu_2}=k_{\mu_2}$.  We deal with the second case first.

Suppose $\jz_{\mu_2}= k_{\mu_2}$.  Pick $p$ such that
$S_p = D_{k_{\mu_2}}^{\mu_2}$.
We have,
\begin{equation*}
D_{k_{\mu_2}}^{\mu_2} f\q( x\w) = \denum{\psi}{-3}\q( x\w) \int f\q( \gt_t\q( x\w)\w) \denum{\psi}{-3}\q( \gt_t\q( x\w)\w) \phi_{\mu_2,k_{\mu_2}}\q( t\w) \: dt,
\end{equation*}
where,
\begin{equation*}
\gt_t\q( x\w) = \gh_{2^{-k_{\mu_2}} t}^{\mu_2}\q( x\w) = \exp\q( 2^{-k_{\mu_2} d_1^{\mu_2}} t_1 X_1^{\mu_2} + \cdots + 2^{-k_{\mu_2} d_{q_{\mu_2}}^{\mu_2}} t_{q_{\mu_2}} X_{q_{\mu_2}}^{\mu_2}\w)x.
\end{equation*}
By the definition of $\q( X^{\mu_2}, d^{\mu_2}\w)$, $\q( X_l, d_l^{\mu_2}\w)$
appears in the list $\q( X^{\mu_2},d^{\mu_2}\w)$ (this uses the fact
that $d_l$ is nonzero in only the $\mu_2$ coordinate).
Hence, $Z_l$ is of the form 
$2^{-\jz_{\mu_2} d_{m}^{\mu_2}} X_m^{\mu_2}=2^{-k_{\mu_2} d_{m}^{\mu_2}} X_m^{\mu_2}$ for some $m$.
It follows that,
\begin{equation*}
Z_l \q( x\w) = \frac{\partial}{\partial t_m}\bigg|_{t=0} \frac{d}{d\epsilon}\bigg|_{\epsilon =1} \gt_{\epsilon t}\circ \gt_{t}^{-1} \q(x \w),
\end{equation*}
which completes the proof of \eqref{EqnToShowZAppears}
in this case.

We now turn to the case when $\jz_{\mu_2}=j_{\mu_2}$.  We separate
this case into two cases:  when $\mu_2\in E$ and when $\mu_2\in E^{c}$.
First we deal with the case when $\mu_2\in E^{c}$.  In that case,
we use $p$ such that $S_p=A_{j_{\mu_2}}^{\mu_2}$,
and the proof of \eqref{EqnToShowZAppears} proceeds
just as in the case for $D_{k_{\mu_2}}^{\mu_2}$ above.

We turn to the case when $\mu_2\in E$.  In this case,
we use $p$ such that $S_p= M_{j_E}$.
We have,
\begin{equation*}
M_{j_E} f\q( x\w) = \denum{\psi}{0}\q( x\w)\int f\q( \gamma_{2^{-j_E}t}\q( x\w)\w) \denum{\psi}{0}\q( \gamma_{2^{-j_E}t}\q( x\w)\w) \sigma\q(t\w)\: dt.
\end{equation*}
Hence, we take $\gt_t= \gamma_{2^{-j_E}t}$.  Note, if
$\gt_t$ were instead taken to be equal to $\gamma_{2^{-j}t}$,
then \eqref{EqnToShowZAppears} would follow immediately from \eqref{EqnHowTheXjsCome}
and \eqref{EqnHowThedjsCome}.
$\gamma_{2^{-j_E}t}$ is just $\gamma_{2^{-j}t}$ with some of the coordinates
set to $0$.  Thus, to prove \eqref{EqnToShowZAppears}, we need
only show that $\partial_t^{\alpha}$ in \eqref{EqnHowTheXjsCome}
only involves those coordinates of $2^{-j_E}t$ which are not
identically $0$.  Let $\alpha$ be the multi-index from \eqref{EqnHowTheXjsCome}.
We know that $\deg\q( \alpha\w) = d_l$ and therefore is nonzero
in only the $\mu_2$ coordinate.  Thus if $t_m$ is a coordinate
appearing in $\partial_t^{\alpha}$, then $e_m$ must be nonzero
in only the $\mu_2$ coordinate.  Since $j_E$ is $\infty$
only in those coordinates $\mu\in E^{c}$, it follows
that $2^{-j_E}t$ is not identically $0$ in any of the coordinates
appearing in $\partial_t^{\alpha}$.  \eqref{EqnToShowZAppears} follows.

In conclusion, the operators $S_1,\ldots, S_L$ satisfy
all the hypotheses of Theorem \ref{ThmGenL2Thm}.
We now turn to $R_1,R_2$.
Recall, we are presently considering the case
$\iinf = k_{\mu_1}-j_{\mu_1}$.
We have,
\begin{equation*}
\begin{split}
D_{k_\mu}^\mu f\q( x\w) &= \denum{\psi}{-3}\q( x\w) \int f\q( \gh_{2^{-k_{\mu_1}}t}^{\mu_1}\q( x\w)\w) \denum{\psi}{-3}\q(\gh_{2^{-k_{\mu_1}}t}^{\mu_1}\q( x\w) \w) \phi_{\mu_1,k_{\mu_1}} \q( t\w) \: dt\\
&\quad -\denum{\psi}{-3}\q( x\w) \int f\q( \gh_{0}^{\mu_1}\q( x\w)\w) \denum{\psi}{-3}\q(\gh_{0}\q( x\w) \w) \phi_{\mu_1,k_{\mu_1}} \q( t\w) \: dt\\
&=:R_1f\q( x\w) - R_2f\q(x\w),
\end{split}
\end{equation*}
where we have used that $\int \phi_{\mu_1,k_{\mu_1}}=0$ (since $k_{\mu_1}>0$, which follows from our assumption that $\iinf>0$) and therefore
$R_2=0$.

Let $c_0=\min_{1\leq m\leq q_{\mu}} d_{m}^{\mu}>0$.
Define $\zeta=2^{-c_0 \iinf}$.  Let,
\begin{equation*}
\gt_{t,s}\q( x\w) = \exp\q( 2^{-\jz_{\mu_1} d_1^{\mu_1}} 2^{-\q(k_{\mu_1}-\jz_{\mu_1}\w) \q(d_1^{\mu_1}-c_0\w) } s t_1 + \cdots+2^{-\jz_{\mu_1} d_{q_{\mu_1}}^{\mu_1}} 2^{-\q(k_{\mu_1}-\jz_{\mu_1}\w) \q(d_{q_{\mu_1}}^{\mu_1}-c_0\w) } st_{q_{\mu_1}}  \w)x.
\end{equation*}
It is simple to verify that $\q( Z,\dt\w)$ controls
$\gt_{t,s}$ at the unit scale.  In particular, this follows from
the fact that $\q( Z,\dt\w)$ controls $\gh_{2^{-\jz_{\mu_1}} t}^{\mu_1}$
at the unit scale (see Lemma 12.18 of \cite{StreetMultiParameterSingRadonLt}), and $\gt_t$ is just $\gh_{2^{-\jz_{\mu_1}} t}$
with $t_m$ replaced by $2^{-\q(k_{\mu_1}-\jz_{\mu_1}\w) \q(d_m^{\mu_1}-c_0\w) } s t_m$.

Let $\zeta = 2^{-c_0\q(k_{\mu_1}-\jz_{\mu_1} \w)}=2^{-c_0\iinf}$.
Note that,
\begin{equation*}
\gh_{2^{-k_{\mu_1}}t}^{\mu_1}= \gt_{t, \zeta}.
\end{equation*}
We therefore have,
\begin{equation*}
\begin{split}
\q( R_1-R_2\w) f\q( x\w) &= \denum{\psi}{-3}\q( x\w) \int f\q( \gt_{t,\zeta}\q( x\w)\w) \denum{\psi}{-3}\q(\gt_{t,\zeta}\q( x\w) \w) \phi_{\mu_1,k_{\mu_1}} \q( t\w) \: dt\\
&\quad -\denum{\psi}{-3}\q( x\w) \int f\q( \gt_{t,0}\q( x\w)\w) \denum{\psi}{-3}\q(\gt_{t,0}\q( x\w) \w) \phi_{\mu_1,k_{\mu_1}} \q( t\w) \: dt.
\end{split}
\end{equation*}
This completes the proof that $R_1-R_2$ has the desired form.

Theorem \ref{ThmGenL2Thm} applies to show that there exists $\epsilon>0$
such that,
\begin{equation*}
\LpOpN{2}{S_1\cdots S_L\q( R_1-R_2\w)}\lesssim \zeta^{\epsilon} = 2^{-\epsilon' \iinf},
\end{equation*}
for some $\epsilon'>0$.
This completes the proof of \eqref{EqnToShowL24}
and therefore shows,
\begin{equation*}
\LpOpN{2}{B_j D_k}\lesssim 2^{-\epsilon'' \q|j-k\w|},
\end{equation*}
in this case.

We return to deal with the second case:  there is a $\mu_1$
such that $\iinf = j_{\mu_1}-k_{\mu_1}$.
We wish to show,
\begin{equation*}
\LpOpN{2}{B_j D_k}\lesssim 2^{-\epsilon\iinf}.
\end{equation*}
Applying the triangle inequality, it suffices to show
for every $E\subseteq \nuset\setminus \q\{\mu_1\w\}$,
\begin{equation*}
\LpOpN{2}{ \q[A_{j_{\Ewmuc}}M_{j_{\Ewmu}} - A_{j_{E^c}}M_{j_E} \w]D_k }\lesssim 2^{-\epsilon \iinf}.
\end{equation*}
Let $O_{j,k,E}  =\q[A_{j_{\Ewmuc}}M_{j_{\Ewmu}} - A_{j_{E^c}}M_{j_E} \w]D_k $.
Thus, we wish to show,
\begin{equation*}
\LpOpN{2}{O_{j,k,E}^{*} O_{j,k,E}}\lesssim 2^{-\epsilon \iinf}.
\end{equation*}
Applying the triangle inequality to the term $O_{j,k,E}^{*}$,
we see that it suffices to show for every $F\subseteq \nuset$,
\begin{equation*}
\LpOpN{2}{D_{k}^{*} M_{j_F}^{*} A_{j_{F^c}}^{*} O_{j,k,E}}\lesssim 2^{-\epsilon \iinf}.
\end{equation*}
Using that $\LpOpN{2}{O_{j,k,E}}\lesssim 1$, we see,
\begin{equation*}
\begin{split}
\LpOpN{2}{D_k^{*} M_{j_F}^{*} A_{j_{F^c}}^{*} O_{j,k,E}}^2 &= \LpOpN{2}{O_{j,k,E}^{*} A_{j_{F^{c}}} M_{j_F} D_k D_k^{*} M_{j_F}^{*} A_{j_{F^c}}^{*} O_{j,k,E} }\\
&\lesssim \LpOpN{2}{ A_{j_{F^{c}}} M_{j_F} D_k D_k^{*} M_{j_F}^{*} A_{j_{F^c}}^{*} O_{j,k,E}}.
\end{split}
\end{equation*}
Thus, it suffices to show,
\begin{equation}\label{EqnToShowL2PO}
\LpOpN{2}{P_{j,k,F} O_{j,k,E}}\lesssim 2^{-\epsilon \iinf},
\end{equation}
where $P_{j,k,F}=  A_{j_{F^{c}}} M_{j_F} D_k D_k^{*} M_{j_F}^{*} A_{j_{F^c}}^{*} $.

We now use that $\LpOpN{2}{D_k}, \LpOpN{2}{M_{j_E}}, \LpOpN{2}{M_{j_{\Ewmu}}}\lesssim 1$, to see,
\begin{equation*}
\begin{split}
\LpOpN{2}{P_{j,k,F}O_{j,k,E}} \lesssim &\LpOpN{2}{P_{j,k,F} \q[ A_{j_{\Ewmuc }}-A_{j_{E^c}}\w] }\\
&+\sum_{G\in \q\{E^{c}, \Ewmuc \w\} } \LpOpN{2}{P_{j,k,F} A_{j_G}\q[M_{j_{\Ewmu}}-M_{j_E}\w] }
\end{split}
\end{equation*}

Thus it suffices to show, for every $F,G\subseteq \nuset$ and every
$E\subseteq \nuset\setminus \q\{\mu_1\w\}$,
\begin{equation}\label{EqnToShowL2MCancel}
\LpOpN{2}{P_{j,k,F} A_{j_G} \q[M_{j_{\Ewmu}}-M_{j_E}\w]}\lesssim 2^{-\epsilon \iinf},
\end{equation}
\begin{equation}\label{EqnToShowL2ACancel}
\LpOpN{2}{P_{j,k,F}\q[A_{j_{\Ewmu}} - A_{j_E}\w]}\lesssim 2^{-\epsilon \iinf},
\end{equation}
where we have reversed the roles of $E$ and $E^c$ in \eqref{EqnToShowL2ACancel}.

We begin with \eqref{EqnToShowL2ACancel}.
Write $j_E=\q( j_E^1,\ldots, j_E^\nu\w) \in \Ninf^{\nu}$.
Note that,
\begin{equation*}
A_{j_{\Ewmu}}-A_{j_E} = A_{j_E^1}^1 A_{j_E^2}^2\cdots A_{j_{E}^{\mu_1-1}}^{\mu_1-1} \q[ A_{j_{\mu_1}}^{\mu_1} - A_{\infty}^{\mu_1}\w] A_{j_E^{\mu_1+1}}^{\mu_1+1}\cdots A_{j_E^\nu}^\nu.
\end{equation*}
Using the fact that $\LpOpN{2}{A_{j_E^{\mu}}^{\mu}}\lesssim 1$ for every $\mu$,
to prove \eqref{EqnToShowL2ACancel} it suffices to show,
\begin{equation}\label{EqnToShowL2ACancel2}
\LpOpN{2}{P_{j,k,F} A_{j_E^1}^1 \cdots A_{j_{E}^{\mu_1-1}}^{\mu_1-1} \q[ A_{j_{\mu_1}}^{\mu_1} - A_{\infty}^{\mu_1}\w]}\lesssim 2^{-\epsilon \iinf}.
\end{equation}

To prove \eqref{EqnToShowL2ACancel2} we will apply Theorem
\ref{ThmGenL2Thm} with,
\begin{equation*}
S_1\cdots S_L = P_{j,k,F} A_{j_E^1}^1 \cdots A_{j_{E}^{\mu_1-1}}^{\mu_1-1} 
=A_{j_{F^c}} M_{j_F} D_k D_k^{*} M_{j_F}^{*} A_{j_{F^c}}^{*} A_{j_E^1}^1 \cdots A_{j_{E}^{\mu_1-1}}^{\mu_1-1},
\end{equation*}
\begin{equation*}
R_1 = A_{j_{\mu_1}}^{\mu_1},\quad R_2 = A_\infty^{\mu_1}.
\end{equation*}
As before, we take $\q( Z,\dt\w) = \q( 2^{-\jz}X, \sd\w)$.
The proof that the term $S_1\cdots S_L$ is of the proper
form follows just as before.  We, therefore, concern
ourselves only with showing that $R_1$ and $R_2$
have the proper form.

We have,
\begin{equation*}
R_1 f\q( x\w) = \denum{\psi}{-1}\q( x\w) \int f\q( \gh_{2^{-j_\mu}t}^\mu\q( x\w)\w) \denum{\psi}{-1}\q( \gh_{2^{-j_\mu} t}\q(x\w)\w)\sigma\q( t\w) \: dt,
\end{equation*}
\begin{equation*}
R_2 f\q( x\w) = \denum{\psi}{-1}^2\q( x\w)\q[\int \sigma\w] f\q( x\w)=\denum{\psi}{-1}\q( x\w) \int f\q( \gh_{0}^\mu\q( x\w)\w) \denum{\psi}{-1}\q( \gh_{0}\q(x\w)\w)\sigma\q( t\w) \: dt.
\end{equation*}
Setting $c_0=\min_{1\leq m\leq q_{\mu_1}} d_l^{\mu_1}$ and
$\zeta=2^{-c_0\iinf}$, the proof that $R_1-R_2$ is of the proper
form follows just as in the previous case.
Theorem \ref{ThmGenL2Thm} applies to show
\eqref{EqnToShowL2ACancel2}, thereby establishing
\eqref{EqnToShowL2ACancel}.

We turn, finally, to showing \eqref{EqnToShowL2MCancel}.  We will
apply Theorem \ref{ThmGenL2Thm} with
\begin{equation*}
S_1\cdots S_L = P_{j,k,F} A_{j_G} = A_{j_{F^c}} M_{j_F} D_k D_k^{*} M_{j_F}^{*} A_{j_{F^c}}^{*} A_{j_G},
\end{equation*}
\begin{equation*}
R_1 = M_{j_{\Ewmu}}, \quad R_2=M_{j_E}.
\end{equation*}
As before, we take $\q( Z,\dt\w)= \q( 2^{-\jz} X,\sd\w)$.
That $S_1\cdots S_L$ has the proper form to apply
Theorem \ref{ThmGenL2Thm} follows just as before.
We therefore only concern ourselves with $R_1$ and $R_2$.

Consider,
\begin{equation*}
M_{j_E} f\q( x\w) = \denum{\psi}{0}\q( x\w) \int f\q( \gamma_{2^{-j_E}t}\q( x\w)\w) \denum{\psi}{0}\q( \gamma_{2^{-j_E}t}\q( x\w)\w) \sigma\q( t\w)\: dt,
\end{equation*}
with a similar formula for $M_{j_{\Ewmu}}$.
For $t\in \R^N$, separate $t$ into two variables: $t=\q( t_1,t_2\w)$.
$t_1$ will consist of those coordinates $t_l$ such that 
$e_l^{\mu_1}\ne 0$, and $t_2$ will denote the rest of the coordinates.
Write $t_1=\q( t_{l_1},\ldots, t_{l_{N_1}}\w)$.  
Let $i\in \Ninf^{\nu}$ be given by,
\begin{equation*}
i_\mu =
\begin{cases}
j_E^{\mu} & \text{if }\mu\ne \mu_1,\\
k_{\mu_1} & \text{if }\mu=\mu_1.
\end{cases}
\end{equation*}
Notice,
\begin{equation*}
2^{-j_{\Ewmu}} \q( t_1,t_2\w) = 2^{-i}\q( \q( 2^{\q(k_{\mu_1}-j_{\mu_1}\w) e_{l_1}^{\mu_1} }t_{l_1},\cdots, 2^{\q(k_{\mu_1}-j_{\mu_1}\w) e_{l_{N_1}}^{\mu_1} }t_{l_{N_1}} \w), t_2 \w),
\end{equation*}
\begin{equation*}
2^{-j_{E}} \q( t_1,t_2\w) = 2^{-i}\q( 0, t_2\w).
\end{equation*}
Define $c_0=\min\q\{e_{l_1}^{\mu_1},\cdots, e_{l_{N_1}}^{\mu_1}\w\}>0$,
and $\zeta= 2^{-c_0\q(j_{\mu_1}-k_{\mu_1}\w)} = 2^{-c_0\iinf}$.
Thus,
\begin{equation*}
2^{-j_{\Ewmu}} \q( t_1,t_2\w) = 2^{-i}\q( \q( 2^{\q(k_{\mu_1}-j_{\mu_1}\w) \q(e_{l_1}^{\mu_1}-c_0\w) }\zeta t_{l_1},\cdots, 2^{\q(k_{\mu_1}-j_{\mu_1}\w) \q(e_{l_{N_1}}^{\mu_1}-c_0\w) }\zeta t_{l_{N_1}} \w), t_2 \w).
\end{equation*}

Define,
\begin{equation*}
\gt\q(\q( t_1,t_2,s\w) x\w) = \gamma\q( 2^{-i}\q( 
\q( 2^{\q(k_{\mu_1}-j_{\mu_1}\w) \q(e_{l_1}^{\mu_1}-c_0\w) }s t_{l_1},\cdots, 2^{\q(k_{\mu_1}-j_{\mu_1}\w) \q(e_{l_{N_1}}^{\mu_1}-c_0\w) }s t_{l_{N_1}} \w), t_2 
 \w) , x\w).
\end{equation*}
We claim that $\q( Z,\dt\w)$ controls $\gt$ at the unit scale.
Indeed, we know $\q(Z,\dt\w)$ controls
$\gamma_{2^{-\jz} t}$ at the unit scale,\footnote{As discussed before,
this follows directly from our assumptions on $\gamma$.}
 and since $i\geq \jz$ coordinatewise,
we have that $\q( Z,\dt\w)$ controls $\gamma_{2^{-i}t}$ at the unit scale.
Now the result follows easily from the definition of control, since $\q(k_{\mu_1}-j_{\mu_1}\w)\q(e_{l_m}^{\mu_1}-c_0\w)<0$ for every $m$.

Note that,
\begin{equation*}
R_1 f\q( x\w) = \denum{\psi}{0}\q( x\w) \int f\q( \gt_{t,\zeta}\q( x\w)\w) \denum{\psi}{0}\q(\gt_{t,\zeta}\q( x\w) \w)\sigma\q( t\w) \: dt,
\end{equation*}
\begin{equation*}
R_2 f\q( x\w) = \denum{\psi}{0}\q( x\w) \int f\q( \gt_{t,0}\q( x\w)\w) \denum{\psi}{0}\q(\gt_{t,0}\q( x\w) \w)\sigma\q( t\w) \: dt.
\end{equation*}
Thus, $R_1-R_2$ has the proper form for Theorem \ref{ThmGenL2Thm}.
Theorem \ref{ThmGenL2Thm} applies to show,
\begin{equation*}
\LpOpN{2}{S_1\cdots S_L\q( R_1-R_2\w)}\lesssim \zeta^{\epsilon}= 2^{-\epsilon' \iinf},
\end{equation*}
which establishes \eqref{EqnToShowL2MCancel} and completes
the proof of
\eqref{EqnToShowBjDk}.

We now make comments on the modifications of the above
necessary to deal with the other parts
of Theorem \ref{ThmL2Thm}.  When considering
$D_{j_1}T_{j_2}D_{j_3}$ one takes
$\jz=j_1\wedge j_2\wedge j_3$ and $\iinf = \max_{1\leq k,l\leq 3} \q|j_k-j_l\w|_\infty\approx \diam{j_1,j_2,j_3}$.
Also, we set $\q( Z,\dt\w) =\q( 2^{-\jz} X,\sd\w)$.  One proceeds
in essentially the same manner as $B_{j_1}D_{j_2}$.  The $D_j$
terms behave just as before.  When $T_{j_2}$ appears
as a $S_p$ for some $p$, it can be treated
just as $M_{j_{\nuset}}$ was treated above.  When
$\iinf=j_2^{\mu_1}-j_1^{\mu_1}$ or $\iinf=j_2^{\mu_1}-j_3^{\mu_1}$, for some $\mu_1$,
$T_{j_2}$ must also be used as $R_1-R_2$ in the above argument.
In that case, one uses that $\int \vsig_{j_2}\q( t\w) d t_{\mu_1}=0$
and setting $R_2=0$, one can
write $T_{j_2}=R_1=R_1-R_2$ in a form that works just as
$M_{j_{\nuset}} - M_{j_{\nuset\setminus \q\{\mu_1\w\}}}$ did above.
See also \cite{StreetMultiParameterSingRadonLt} for details on this.

When considering, instead, $D_{j_1}D_{j_2}D_{j_3}^{*}D_{j_4}^{*}$
or $D_{j_1}^{*} D_{j_2}^{*} D_{j_3} D_{j_4}$,
one takes $\jz=j_1\wedge j_2\wedge j_3\wedge j_4$ and
$\iinf = \max_{1\leq k,l\leq 4} \q| j_k-j_l \w|_\infty \approx \diam{j_1,j_2,j_3,j_4}$.
Also, one takes $\q( Z,\dt\w) =\q( 2^{-\jz}X,\sd\w)$.  With these
choices everything proceeds as above with simple modifications.
We leave the details to the interested reader.

\section{Square functions and the reproducing formula}\label{SectionSquareFunc}
Using the operators $D_j$ defined in Section \ref{SectionAuxOps}
we develop, in this section, a Littlewood-Paley square function
and a Calder\'on-type ``reproducing formula'' which will
be essential to our proof of Theorems
\ref{ThmMainThmSecondPass} and \ref{ThmMainMaxThm}.

Recall,
\begin{equation*}
\sum_{j\in \N^{\nu}} D_j = \denum{\psi}{-3}^{2\nu}.
\end{equation*}
For notational convenience, we define $D_j=0$ for $j\in \Z^{\nu}\setminus \N^{\nu}$.
For $M\in \N$, define,
\begin{equation*}
U_M = \sum_{\substack{j\in \N^{\nu}\\ \q|l\w|\leq M}} D_j D_{j+l},
\end{equation*}
\begin{equation*}
R_M = \sum_{\substack{j\in \N^{\nu}\\ \q|l\w|> M}} D_j D_{j+l};
\end{equation*}
so that $U_M+R_M=\denum{\psi}{-3}^{4\nu}$.  The main results
of this section are:
\begin{thm}[A Calder\'on-type ``reproducing formula'']\label{ThmReproduce}
Fix $p_0$, $1<p_0<\infty$.  There exists $M=M\q( p_0\w)$,
and a bounded map $V_M:L^{p_0}\rightarrow L^{p_0}$ such that,
\begin{equation*}
\denum{\psi}{-2} U_M V_M = \denum{\psi}{-2} = V_M U_M \denum{\psi}{-2}.
\end{equation*}
\end{thm}

\begin{rmk}
Strictly speaking, Theorem \ref{ThmReproduce} does not give a reproducing
formula (one would need $V_M=I$ for it to be a reproducing formula).
However, we will use it in the same way that one often
uses the Calder\'on reproducing formula, which is why we have
labeled it such.
\end{rmk}

\begin{thm}[The Littlewood-Paley square function]\label{ThmSquare}
For every $p$, $1<p<\infty$,
\begin{equation}\label{EqnfLessSquare}
\LpN{p}{\denum{\psi}{-2} f} \lesssim \LpN{p}{\q(\sum_{j\in \N^\nu}\q|D_j\denum{\psi}{-2} f\w|^2\w)^{\frac{1}{2}}},
\end{equation}
\begin{equation}\label{EqnSquareLessf}
\LpN{p}{\q(\sum_{j\in \N^\nu}\q|D_j f\w|^2\w)^{\frac{1}{2}}}\lesssim \LpN{p}{f}.
\end{equation}
\end{thm}

The rest of this section is devoted to the proofs
of Theorems \ref{ThmReproduce} and \ref{ThmSquare}.
We begin with the proof of \eqref{EqnSquareLessf}:
\begin{lemma}\label{LemmaSquareLessf}
For every $p$, $1<p<\infty$, we have:
\begin{equation}\label{EqnLemmaSquareLessf}
\LpN{p}{\q(\sum_{j\in \N^\nu}\q|D_j f\w|^2\w)^{\frac{1}{2}}}\leq C_p\LpN{p}{f}.
\end{equation}
The same result holds with $D_j$ replaced with $D_j^{*}$.
\end{lemma}
\begin{proof}
As is well known, to prove \eqref{EqnLemmaSquareLessf}, it suffices to prove
for every set of $\nu$ sequences
$\q\{\epsilon_{j_1}^1\w\}_{j_1\in \N},\ldots, \q\{\epsilon_{j_\nu}^\nu\w\}_{j_\nu\in \N}$, taking values of $\pm 1$, the operator
\begin{equation}\label{EqnToShowIIDDbound}
\sum_{j_1,\ldots, j_\nu\in \N} \epsilon_{j_1}^1\epsilon_{j_2}^2\cdots \epsilon_{j_\nu}^\nu D_{\q(j_1,\ldots, j_\nu\w)}
\end{equation}
is bounded on $L^p$, with bound independent of the choice of the sequence.
To see why \eqref{EqnToShowIIDDbound} is enough, see for example,
page 267 of \cite{SteinHarmonicAnalysis} and Chapter 4,
Section 5 of \cite{SteinSingularIntegrals}.
We have,
\begin{equation*}
\sum_{j_1,\ldots, j_\nu\in \N} \epsilon_{j_1}^1\epsilon_{j_2}^2\cdots \epsilon_{j_\nu}^\nu D_{\q(j_1,\ldots, j_\nu\w)} = \q(\sum_{j_1\in \N} \epsilon_{j_1}^1 D_{j_1}^1\w)\cdots \q(\sum_{j_\nu\in \N} \epsilon_{j_\nu}^{\nu} D_{j_\nu}^{\nu}\w).
\end{equation*}

For each $\mu$,
\begin{equation*}
\sum_{j_\mu} \epsilon_{j_\mu}^{\mu} D_{j_\mu}^{\mu}
\end{equation*}
is of the form covered by Theorem \ref{ThmSingleParamSingInt} and
hence bounded on $L^p$ ($1<p<\infty$).  It is easy to see that the
Theorem \ref{ThmSingleParamSingInt} holds uniformly in the
choice of the sequence $\q\{\epsilon_{j_{\mu}}^{\mu}\w\}$ taking
values of $\pm 1$.  The result now follows.

The same proof works with $D_j$ replaced by $D_j^{*}$.
\end{proof}

To prove Theorem \ref{ThmReproduce}, we first state a preliminary lemma.
\begin{lemma}\label{LemmaRMLimit}
For $p_0$ fixed, $1<p_0<\infty$,
\begin{equation*}
\lim_{M\rightarrow \infty} \LpOpN{p_0}{R_M} =0.
\end{equation*}
\end{lemma}

\begin{proof}[Proof of Theorem \ref{ThmReproduce} given Lemma \ref{LemmaRMLimit}]
Note that,
\begin{equation*}
U_M = \denum{\psi}{-3}^{4\nu} - R_M.
\end{equation*}
Take $\psi\in C_0^\infty$ with $\denum{\psi}{-2}\prec \psi\prec \denum{\psi}{-3}$.
Take $M=M\q( p_0\w)$ so large that $\LpOpN{p_0}{R_M \psi}<1$.
Define,
\begin{equation*}
V_M = \sum_{m=0}^\infty \psi\q(R_M \psi\w)^m,
\end{equation*}
with convergence in the uniform operator topology $L^{p_0}\rightarrow L^{p_0}$.
It is direct to verify that $V_M$ satisfies the conclusions of
Theorem \ref{ThmReproduce}.
\end{proof}

Lemma \ref{LemmaRMLimit} follows by interpolating the next two lemmas,
\begin{lemma}\label{LemmaBadLpRM}
For $1<p<\infty$, $M\geq 1$,
\begin{equation*}
\LpOpN{p}{R_M}\leq C_p M^{\nu}.
\end{equation*}
\end{lemma}

\begin{lemma}\label{LemmaGoodL2RM}
\begin{equation*}
\LpOpN{2}{R_M}\lesssim 2^{-\epsilon M},
\end{equation*}
for some $\epsilon>0$.
\end{lemma}

\begin{proof}[Proof of Lemma \ref{LemmaBadLpRM}]
Since $R_M = \denum{\psi}{-3}^{4\nu}-U_M$, it suffices to prove
the result with $U_M$ in place of $R_M$.
Let $q$ be dual to $p$.  Fix $f\in L^p$ and $g\in L^q$
with $\LpN{q}{g}=1$.  Consider,
\begin{equation*}
\begin{split}
\q|\ip{g}{U_M f}\w| &\leq \sum_{\q|l\w|\leq M} \q|\ip{g}{\sum_{j\in \N^\nu} D_j D_{j+l} f}\w|\\
&= \sum_{\q|l\w|\leq M} \q| \sum_{j\in \N^\nu} \ip{D_j^{*} g}{ D_{j+l} f}\w|\\
&\leq \sum_{\q|l\w|\leq M} \int \q(\sum_{j\in \N^\nu} \q|D_j^{*} g\w|^2\w)^{1/2} \q(\sum_{j\in \N^{\nu}} \q|D_j f\w|^2\w)^{1/2}\\
&\lesssim M^\nu \LpN{q}{\q(\sum_{j\in\N^\nu} \q|D_j^{*}g\w|^2\w)^{1/2} }
\LpN{p}{\q(\sum_{j\in \N^\nu} \q|D_j f\w|^2\w)^{1/2}}\\
&\lesssim M^{\nu} \LpN{q}{g} \LpN{p}{f}\\
&= M^{\nu} \LpN{p}{f};
\end{split}
\end{equation*}
where, in the second to last line, we have applied Lemma \ref{LemmaSquareLessf}
twice.  Taking the supremum over all $g$ with $\LpN{q}{g}=1$
yields the result.
\end{proof}

\begin{proof}[Proof of Lemma \ref{LemmaGoodL2RM}]
We wish to apply the Cotlar-Stein lemma to,
\begin{equation*}
\sum_{\substack{j\in \N^\nu\\ \q|l\w|>M}} D_j D_{j+l}.
\end{equation*}

Applying Theorem \ref{ThmL2Thm}, we have,
\begin{equation*}
\LpOpN{2}{D_{j_1} D_{j_1+l_1} D_{j_2}^{*} D_{j_2+l_2}^{*}}, \LpOpN{2}{D_{j_1}^{*} D_{j_1+l_1}^{*} D_{j_2} D_{j_2+l_2}}\lesssim 2^{-\epsilon_2 \diam{j_1,j_1+l_1, j_2,j_2+l_2}}.
\end{equation*}

The Cotlar-Stein lemma states,
\begin{equation*}
\LpOpN{2}{R_M}\lesssim \sup_{\substack{j_1\in \N \\ \q|l_1\w|>M}} \sum_{\substack{ j_2\in \N \\ \q|l_2\w|>M}} 2^{-\epsilon_2\diam{j_1,j_1+l_1, j_2, j_2+l_2}/2} \lesssim 2^{-\epsilon M};
\end{equation*}
completing the proof.
\end{proof}

We have now completed the proof of Theorem \ref{ThmReproduce}.
We end this section with the completion of the proof of
Theorem \ref{ThmSquare}, by proving \eqref{EqnfLessSquare}.

\begin{proof}[Proof of \eqref{EqnfLessSquare}]
Fix $p_0$, $1<p_0<\infty$, and let $M=M\q( p_0\w)$ be as
in Theorem \ref{ThmReproduce}.
Let $q_0$ be dual to $p_0$, so that $V_M^{*}:L^{q_0}\rightarrow L^{q_0}$.
Let $g\in L^{q_0}$ be such that $\LpN{q_0}{g}=1$.
We have for $f\in L^{p_0}$,
\begin{equation*}
\begin{split}
\q|\ip{g}{\denum{\psi}{-2} f}\w| &= \q|\ip{V_M^{*} g}{U_M \denum{\psi}{-2} f}\w|\\
&\leq \sum_{\q|l\w|\leq M} \q| \sum_{j\in \N^\nu} \ip{D_j^{*} V_M^{*} g}{ D_{j+l} \denum{\psi}{-2} f}\w|\\
&\leq \sum_{\q|l\w|\leq M} \LpN{q_0}{\q(\sum_{j\in \N^{\nu}} \q|D_j^{*} V_M^{*} g\w|^2\w)^{1/2} }
\LpN{p_0}{\q( \sum_{j\in \N^{\nu}} \q|D_j \denum{\psi}{-2} f\w|^2  \w)^{1/2} }\\
&\lesssim M^\nu \LpN{q_0}{V_M^{*} g}
\LpN{p_0}{\q( \sum_{j\in \N^{\nu}} \q|D_j \denum{\psi}{-2} f\w|^2  \w)^{1/2} }\\
&\lesssim 
\LpN{p_0}{\q( \sum_{j\in \N^{\nu}} \q|D_j \denum{\psi}{-2} f\w|^2  \w)^{1/2} };
\end{split}
\end{equation*}
where in the second to last line we applied Lemma \ref{LemmaSquareLessf},
and in the last line we used that $M$ is fixed (since $p_0$ is),
$V_M^{*}$ is bounded on $L^{q_0}$, and $\LpN{q_0}{g}=1$.
Taking the supremum over all $g$ with $\LpN{q_0}{g}=1$ yields the result.
\end{proof}

\section{The maximal result (Theorem \ref{ThmMainMaxThm})}
In this section, we prove Theorem \ref{ThmMainMaxThm}.  The proof
proceeds by a bootstrapping argument.
In fact, there are at least two, well-known,
bootstrapping arguments that can be used to prove results like
Theorem \ref{ThmMainMaxThm}.  One can be found in
\cite{NagelSteinWaingerDifferentiationInLacunaryDirections},
another in Section 4 of \cite{GreenleafSeegerWaingerOnXrayTransformsForRigidLineComplexes}.
Either of these arguments will suffice for our purposes.
We proceed using the methods of
\cite{NagelSteinWaingerDifferentiationInLacunaryDirections}.

Let us review a few of the reductions covered in
Section \ref{SectionAuxOps}.  First, for $E\subseteq \nuset$,
define,
\begin{equation*}
\sM_E f\q(x\w) = \sup_{j\in \N^\nu} M_{j_E} \q|f\w| \q( x\w).
\end{equation*}
Then, to prove Theorem \ref{ThmMainMaxThm} it suffices to
prove that $\sM_{\nuset}$ is bounded on $L^p$ ($1<p\leq \infty$).
We proceed by induction on $\nu$.  As discussed
in Section \ref{SectionAuxOps}, the base case ($\nu=0$)
is trivial, and we may assume by our inductive hypothesis
that $\sM_E$ is bounded on $L^p$ for $E\subsetneq \nuset$.
Note that $\sM_{\nuset}$ is clearly bounded on $L^\infty$,
and so our goal is to show that it is bounded on $L^p$, $1<p\leq 2$.

The following lemma was proved in Section \ref{SectionAuxOps} (under
our inductive hypothesis),
\begin{lemma}\label{LemmaEquivBandsM}
For each $p$, $1<p<\infty$,
\begin{equation*}
\LpN{p}{\sM_{\nuset} f}\lesssim \LpN{p}{f}, \quad \forall f\in L^p,
\end{equation*}
if and only if
\begin{equation*}
\LpN{p}{\sup_{j\in \N^{\nu}}\q|B_j f\w|}\lesssim\LpN{p}{f}, \quad \forall f\in L^p.
\end{equation*}
\end{lemma}

\begin{rmk}
Actually, only the if part of Lemma \ref{LemmaEquivBandsM} was shown
in Section \ref{SectionAuxOps}.  The only if part is immediate,
and is also not used in what follows.  We, therefore, leave it
to the interested reader.
\end{rmk}

In what follows, $D_j$ for $j\in \Z^{\nu}\setminus \N^\nu$ is defined to be $0$.
For $k\in \Z^{\nu}$ define a new operator acting on sequences
of measurable functions $\q\{f_j\q(x\w)\w\}_{j\in \N^\nu}$ by,
\begin{equation*}
\sB_k \q\{f_j\w\}_{j\in \N^{\nu}} = \q\{B_j D_{j+k} f_j\w\}_{j\in \N^{\nu}}.
\end{equation*}

\begin{prop}\label{PropSTSsBk}
Fix $p_0$, $1<p_0<\infty$.  If there is $\epsilon>0$ such that
\begin{equation*}
\LplqOpN{p_0}{2}{\sB_k}\lesssim 2^{-\epsilon\q|k\w|},
\end{equation*}
then $\sM_{\nuset}$ is bounded on $L^{p_0}$.
\end{prop}
\begin{proof}
In light of Lemma \ref{LemmaEquivBandsM} and because
\begin{equation*}
\LpN{p_0}{\sup_{j\in \N^{\nu}} \q|B_j f\w| } \leq \LpN{p_0}{ \q(\sum_{j\in \N^{\nu}} \q|B_j f\w|^2\w)^{1/2}},
\end{equation*}
it suffices to show
\begin{equation*}
\LpN{p_0}{ \q(\sum_{j\in \N^{\nu}} \q|B_j f\w|^2\w)^{1/2}}\lesssim \LpN{p_0}{f}.
\end{equation*}
Fix $M=M\q( p_0\w)$ as in Theorem \ref{ThmReproduce}.  Note,
$B_j= B_j \denum{\psi}{-2} = B_j \denum{\psi}{-2} U_M V_M= B_j U_M V_M$.
Let $g=V_M f$.  Note that $\LpN{p_0}{g}\lesssim \LpN{p_0}{f}$.
Thus, it suffices to show,
\begin{equation*}
\LpN{p_0}{\q(\sum_{j\in \N^{\nu}} \q|B_j U_M g\w|^2\w)^{1/2} }\lesssim \LpN{p_0}{g}.
\end{equation*}
Consider, using the triangle inequality,
\begin{equation*}
\begin{split}
\LpN{p_0}{\q(\sum_{j\in \N^{\nu}} \q|B_j U_M g\w|^2\w)^{1/2}} & = \LpN{p_0}{\q(  \sum_{j\in \N^{\nu}} \q|\sum_{\substack{j_0\in \N^{\nu} \\ \q|l\w|\leq M } } B_j D_{j_0} D_{j_0+l} g\w|^2\w)^{1/2} }\\
& = \LpN{p_0}{\q(  \sum_{j\in \N^{\nu}} \q|\sum_{\substack{k\in \Z^{\nu} \\ \q|l\w|\leq M } } B_j D_{j+k} D_{j+k+l} g\w|^2\w)^{1/2} }\\
& \leq \sum_{\substack{k\in \Z^{\nu}\\ \q|l\w|\leq M } }\LpN{p_0}{\q(  \sum_{j\in \N^{\nu}} \q| B_j D_{j+k} D_{j+k+l} g\w|^2\w)^{1/2} }\\
& = \sum_{\substack{k\in \Z^{\nu}\\ \q|l\w|\leq M } } \LplqN{p_0}{2}{\sB_k \q\{D_{j+k+l} g\w\}_{j\in \N^{\nu}} }\\
& \lesssim  \sum_{\substack{k\in \Z^{\nu}\\ \q|l\w|\leq M } }  2^{-\epsilon \q|k\w|} \LplqN{p_0}{2}{\q\{D_j g\w\}_{j\in \N^{\nu} }}\\
&\lesssim \LpN{p_0}{\q(\sum_{j\in \N^{\nu} } \q|D_j g\w|^2\w)^{1/2} }\\
&\lesssim \LpN{p_0}{g},
\end{split}
\end{equation*}
where, in the last line, we have applied Theorem \ref{ThmSquare}.
This completes the proof of the proposition.
\end{proof}

Define,
\begin{equation*}
\sP=\q\{p\in \q(1,2\w]: \exists \epsilon>0, \LplqOpN{p}{2}{\sB_k}\lesssim 2^{-\epsilon \q|k\w|}\w\}.
\end{equation*}
In light of Proposition \ref{PropSTSsBk}, the $L^p$ boundedness
of $\sM_{\nuset}$ will follow directly from
the following proposition.

\begin{prop}\label{PropsPEqual}
$\sP=\q(1,2\w]$.
\end{prop}

Proposition \ref{PropsPEqual}, in turn, follows directly
from the next lemma.

\begin{lemma}\label{LemmasPProp}
\begin{itemize}
\item $2\in \sP$,
\item If $q\in \sP$, then $\q(\frac{2q}{q+1}, 2\w]\subseteq \sP$.
\end{itemize}
\end{lemma}

It is easy to see that any subset of $\q( 1,2\w]$ satisfying
the conclusions of Lemma \ref{LemmasPProp} must equal
$\q(1,2\w]$.  The rest of this section is devoted
to the proof of Lemma \ref{LemmasPProp},
which then completes the proof of Theorem \ref{ThmMainMaxThm}.

\begin{proof}[Proof of Lemma \ref{LemmasPProp}]
That $2\in \sP$ follows directly from Theorem \ref{ThmL2Thm}.
In fact, if $\epsilon_2>0$ is as in Theorem \ref{ThmL2Thm},
we have,
\begin{equation}\label{EqnsBkL2Ineq}
\LplqOpN{2}{2}{\sB_k} \lesssim 2^{-\epsilon_2\q|k\w|};
\end{equation}
merely by interchanging the norms.  Moreover, using that,
\begin{equation*}
\LpOpN{1}{B_j D_{j+k}}\lesssim 1,
\end{equation*}
we have,
\begin{equation}\label{EqnsBkL1Ineq}
\LplqOpN{1}{1}{\sB_k} \lesssim 1;
\end{equation}
also by interchanging the norms.
Interpolating \eqref{EqnsBkL2Ineq} and \eqref{EqnsBkL1Ineq}
shows for $1<p\leq 2$,
\begin{equation}\label{EqnsBkLpIneq}
\LplqOpN{p}{p}{\sB_k} \lesssim 2^{-\epsilon_p \q|k\w|},
\end{equation}
where $\epsilon_p=\q( 2-\frac{2}{p}\w)\epsilon_2>0$.

Now suppose $q\in \sP$.  By Proposition \ref{PropSTSsBk}
$\sM_{\nuset}$ is bounded on $L^q$.
We claim,
\begin{equation}\label{EqnlinfsB}
\LplqOpN{q}{\infty}{\sB_k} \lesssim 1.
\end{equation}

Before we verify \eqref{EqnlinfsB}, recall the maximal functions
$\sM^{\mu}$ defined in Section \ref{SectionAuxOps}.
$\sM^{\mu}$ is bounded on $L^p$ ($1<p\leq \infty$), and
by \eqref{EqnSupAjBound} and the definition of $A_{j_E}$, we have
\begin{equation*}
\q|A_{j_E} f\q(x\w)\w| \lesssim \q[\prod_{\mu\in E} \sM^\mu \w] f\q( x\w),
\end{equation*}
where the product is taken in order of increasing $\mu$.  A similar
result holds for $D_j$ with $E$ replaced by $\nuset$.

We now turn to verifying \eqref{EqnlinfsB}.
\begin{equation*}
\begin{split}
\LpN{q}{\sup_{j\in \N^{\nu}} \q|B_j D_{j+k} f_j\w|} &\leq \sum_{E\subseteq \nuset} \LpN{q}{\sup_{j\in \N^{\nu}} \q|A_{j_{E^c}} M_{j_E} D_{j+k}f_j\w|}\\
&\lesssim \sum_{E\subseteq \nuset} \LpN{q}{\q[\prod_{\mu\in E^c} \sM^{\mu}\w]\sM_E \q[\prod_{\mu=1}^\nu \sM^{\mu}\w] \sup_{j\in \N^{\nu}} \q|f_j\w| }\\
&\lesssim \LpN{q}{\sup_{j\in \N^{\nu}} \q|f_j\w|}.
\end{split}
\end{equation*}
In the last line, we have used our inductive hypothesis when $E\ne \nuset$ and
we used that $\sM_{\nuset}$ is bounded on $L^q$ when $E=\nuset$.
This completes the verification of \eqref{EqnlinfsB}.

Interpolating \eqref{EqnlinfsB} with \eqref{EqnsBkLpIneq}
as $p\rightarrow 1$ proves that $\q(\frac{2q}{q+1},2\w]\subseteq \sP$.
Here, we have implicitly used the fact that (by interpolation)
if $r\in \sP$, then $\q[r,2\w]\subseteq \sP$ (since $2\in \sP$).
\end{proof}

\section{Proof of Theorem \ref{ThmMainThmSecondPass}}
This section is devoted to proving Theorem \ref{ThmMainThmSecondPass}:
$T:L^p\rightarrow L^p$, $1<p<\infty$.
Since the class of operators covered in Theorem \ref{ThmMainThmSecondPass}
is closed under adjoints (see Section 12.3 of \cite{StreetMultiParameterSingRadonLt}), it suffices to prove the
result for $1<p\leq 2$.

We decompose $T=\sum_{j\in \N^{\nu}} T_j$ as in Section \ref{SectionAuxOps}.
In what follows, $T_j$ and $D_j$ for $j\in \Z^\nu\setminus \N^\nu$
are defined to be $0$.
For $k_1,k_2\in \Z^\nu$, define a new operator, acting on sequences of
measurable functions $\q\{f_j\q( x\w)\w\}_{j\in \N^{\nu}}$ by
\begin{equation*}
\sT_{k_1,k_2} \q\{f_j\w\}_{j\in \N^\nu} = \q\{D_j T_{j+k_1} D_{j+k_2} f_j\w\}_{j\in \N^\nu}.
\end{equation*}
Theorem \ref{ThmMainThmSecondPass} follows immediately from a
combination of the following two propositions.

\begin{prop}\label{PropIfsTBound}
Fix $p_0$, $1<p_0<\infty$.  If there exists $\epsilon>0$ such that
\begin{equation*}
\LplqOpN{p_0}{2}{\sT_{k_1,k_2}}\lesssim 2^{-\epsilon\q(\q|k_1\w|+\q|k_2\w|\w)},
\end{equation*}
then $T$ is bounded on $L^{p_0}$.
\end{prop}

\begin{prop}\label{PropsTBound}
For each $p$, $1<p\leq 2$, there is an $\epsilon=\epsilon\q( p\w)>0$ such
that,
\begin{equation*}
\LplqOpN{p}{2}{\sT_{k_1,k_2}} \lesssim 2^{-\epsilon\q( \q|k_1\w|+\q|k_2\w| \w) }.
\end{equation*}
\end{prop}

\begin{proof}[Proof of Proposition \ref{PropIfsTBound}]
Fix $p_0$, $1<p_0<\infty$.  Take $M=M\q( p_0\w)$ as
in Theorem \ref{ThmReproduce}.
We have,
$T=T\denum{\psi}{-2} = T \denum{\psi}{-2}U_M V_M= T U_M V_M$.
Since $V_M$ is bounded on $L^{p_0}$, it suffices to show
$TU_M$ is bounded on $L^{p_0}$.

Consider, using Theorem \ref{ThmSquare},
\begin{equation*}
\begin{split}
\LpN{p_0}{TU_M f} & =\LpN{p_0}{\denum{\psi}{-2}TU_M f}\\
&\approx \LpN{p_0}{\q(\sum_{j\in \N^{\nu}} \q|D_j \denum{\psi}{-2} T U_M f\w|^2 \w)^{1/2}}\\
&= \LpN{p_0}{\q(\sum_{j\in \N^{\nu}} \q|D_j T U_M f\w|^2 \w)^{1/2}}.
\end{split}
\end{equation*}
Thus, to complete the proof, it suffices to show,
\begin{equation*}
\LpN{p_0}{\q(\sum_{j\in \N^{\nu}} \q|D_j T U_M f\w|^2 \w)^{1/2}}\lesssim \LpN{p_0}{f}.
\end{equation*}

We have, by the triangle inequality,
\begin{equation*}
\begin{split}
\LpN{p_0}{\q(\sum_{j\in \N^{\nu}} \q|D_j T U_M f\w|^2 \w)^{1/2}} &= 
\LpN{p_0}{\q( \sum_{j\in \N^{\nu}} \q|\sum_{\substack{j_1,j_2\in \N^{\nu}\\ \q|l\w|\leq M}} D_j T_{j_1} D_{j_2} D_{j_2+l} f\w|^2 \w)^{1/2} }\\
&=\LpN{p_0}{\q( \sum_{j\in \N^{\nu}} \q|\sum_{\substack{k_1,k_2\in \Z^{\nu}\\ \q|l\w|\leq M}} D_j T_{j+k_1} D_{j+k_2} D_{j+k_2+l} f\w|^2 \w)^{1/2} }\\
&\leq \sum_{\substack{k_1,k_2\in \Z^{\nu}\\ \q|l\w|\leq M}}\LpN{p_0}{\q( \sum_{j\in \N^{\nu}} \q| D_j T_{j+k_1} D_{j+k_2} D_{j+k_2+l} f\w|^2 \w)^{1/2} }\\
&= \sum_{\substack{k_1,k_2\in \Z^{\nu}\\ \q|l\w|\leq M}} \LplqN{p_0}{2}{\sT_{k_1,k_2} \q\{D_{j+k_2+l}f\w\}_{j\in \N^{\nu}} }\\
&\lesssim \sum_{\substack{k_1,k_2\in \Z^{\nu}\\ \q|l\w|\leq M}}  2^{-\epsilon\q( \q|k_1\w|+\q|k_2\w|\w)} \LplqN{p_0}{2}{\q\{D_j f\w\}_{j\in \N^{\nu}}}\\
&\lesssim \LpN{p_0}{\q(\sum_{j\in \N^{\nu}} \q|D_j f\w|^2 \w)^{1/2} }\\
&\lesssim \LpN{p_0}{f},
\end{split}
\end{equation*}
where, in the last line, we have applied Theorem \ref{ThmSquare}.
This completes the proof.
\end{proof}

\begin{proof}[Proof of Proposition \ref{PropsTBound}]
We first prove the result for $p=2$.  In this case, the result follows
immediately from Theorem \ref{ThmL2Thm}.
Indeed if $\epsilon_2>0$ is as in Theorem \ref{ThmL2Thm},
we have,
\begin{equation*}
\LplqOpN{2}{2}{\sT_{k_1,k_2}}\lesssim 2^{-\frac{\epsilon_2}{2}\q( \q|k_1\w|+\q|k_2\w| \w)};
\end{equation*}
merely by interchanging the norms.
In addition, since,
\begin{equation}\label{EqnsTL2}
\LpOpN{1}{D_jT_{j+k_1}{D_{j+k_2}}}\lesssim 1,
\end{equation}
we have,
\begin{equation}\label{EqnsTL1}
\LplqOpN{1}{1}{\sT_{k_1,k_2}}\lesssim 1;
\end{equation}
also by interchanging the norms.
Interpolating \eqref{EqnsTL2} and \eqref{EqnsTL1}
shows that for every $p$, $1<p\leq 2$,
\begin{equation}\label{EqnsTLp}
\LplqOpN{p}{p}{\sT_{k_1,k_2}}\lesssim 2^{-\epsilon_p\q(\q|k_1\w|+\q|k_2\w| \w)},
\end{equation}
where $\epsilon_p=\q(1-\frac{1}{p} \w)\epsilon_2>0$.

We claim, for every $p$, $1<p<\infty$,
\begin{equation}\label{EqnsTlinf}
\LplqOpN{p}{\infty}{\sT_{k_1,k_2}}\lesssim 1.
\end{equation}
We use the maximal operators $\sM^\mu$ defined in Section \ref{SectionAuxOps}
along with the maximal operator $\sM$ from Theorem \ref{ThmMainMaxThm}.
We have,
\begin{equation*}
\begin{split}
\LpN{p}{\sup_{j} \q|\sT_{k_1,k_2} \q\{f_j\w\}_{j\in \N^{\nu}}\w|} &= \LpN{p}{\sup_j \q| D_j T_{j+k_1} D_{j+k_2} f_j \w|  }\\
&\lesssim \LpN{p}{\q[\prod_{\mu=1}^\nu \sM^\mu \w]\sM \q[\prod_{\mu=1}^\nu \sM^\mu\w]\sup_j \q|f_j\w|}\\
&\lesssim \LpN{p}{\sup_j \q|f_j\w|}.
\end{split}
\end{equation*}
In the last line, we used the $L^p$ boundedness of the various maximal functions.
\eqref{EqnsTlinf} follows.

Interpolating \eqref{EqnsTLp} and \eqref{EqnsTlinf} yields the result.
\end{proof}

\section{More general kernels}\label{SectionMoreKernels}
In the previous sections, we exhibited the proofs
of Theorems \ref{ThmMainThmSecondPass}
and \ref{ThmMainMaxThm} in the special case $\mu_0=\nu$.
That is, in the case when $\sA_{\mu_0}=\q[0,1\w]^\nu$.
In this section, we describe the modifications necessary
to prove the result for general $\mu_0$.  At the end of the
section we make some remarks about even more general
sets $\sA\subseteq \q[0,1\w]^\nu$ for which our methods apply.

Fix $\mu_0$, $1\leq \mu_0\leq \nu$.  Recall,
\begin{equation*}
\sA_{\mu_0} = \q\{\delta=\q( \delta_1,\ldots, \delta_\nu\w) \in \q[0,1\w]^\nu: \delta_{\mu_0}\leq \delta_{\mu_0+1}\leq \cdots \leq \delta_{\nu}\w\}.
\end{equation*}
Our decomposition of $T$ now takes the form,
\begin{equation*}
T= \sum_{j\in \lA} T_j,
\end{equation*}
where
\begin{equation*}
\lA = \q\{j\in \N^{\nu} : 2^{-j}\in \sA_{\mu_0}\w\}=\q\{j\in \N^{\nu}: j_{\mu_0}\geq j_{\mu_0+1}\geq \cdots \geq j_{\nu}\w\};
\end{equation*}
see Section \ref{SectionKernels} for more details.

The proof in this case remains almost exactly the same, provided
we make a few different choices when defining
the auxiliary operators in Section \ref{SectionAuxOps}.
In the case when $\mu_0=\nu$, we defined the vector fields
with single parameter formal degrees $\q( X^\mu, d^{\mu}\w)$
so that,
\begin{equation}\label{EqnChooseXmu}
\q( \delta_\mu X^{\mu}, d^{\mu}\w) = \q( \deltah X, \sd\w),
\end{equation}
where $\deltah\in \q[0,1\w]^\nu$ was $\delta_\mu$ in the $\mu$
coordinate and $0$ in every other coordinate.
In the case when $\mu_0<\mu$, this choice no longer works.
Instead, we choose $\q( X^\mu, d^\mu\w)$ so that
\eqref{EqnChooseXmu} holds
where $\deltah$ is defined to be $\delta_\mu$ in the $\mu'$ coordinate,
for every $\mu'\geq \mu$, and defined to be $0$ in the rest of the
coordinates.
The operators $A_j^{\mu}$ and $D_j^{\mu}$ are defined with this choice
of $\q( X^\mu, d^\mu\w)$.

For $E\subseteq \nuset$, $j_E\in \Ninf^{\nu}$ must be defined
differently, so that $2^{-j_E}\in \sA$.
For $1\leq \mu\leq \mu_0$, we define,
\begin{equation*}
j_E^{\mu} = \begin{cases}
j_{\mu} & \text{if }\mu\in E,\\
\infty & \text{otherwise.}
\end{cases}
\end{equation*}
For $\mu>\mu_0$, we recursively define,
\begin{equation*}
j_E^{\mu} =\begin{cases}
j_\mu & \text{if }\mu\in E,\\
\min \q\{\infty, j_E^{\mu-1}\w\} & \text{otherwise.}
\end{cases}
\end{equation*}
One defines $A_{j_{E}}$, $M_{j_E}$ in the same manner as before, but with
this choice of $j_E$.


Now the proof goes through just as before.  Whenever
one uses $B_j$ or $T_j$ one must restrict attention to $j\in \lA$.
However, when one considers $D_j$, one allows $j$
to range over $j\in \N^{\nu}$.

In fact, the above methods work for more general $\sA\subseteq \q[0,1\w]^\nu$.
For instance, just by changing $e$ and $a$, studying the operators
associated to
\begin{equation*}
\sA=\q\{\delta=\q( \delta_1,\ldots, \delta_\nu\w)\in \q[0,1\w]^\nu : \delta_{\mu_0}^{b_{\mu_0}}\leq C_{\mu_0+1} \delta_{\mu_0}^{b_{\mu_0+1}}\leq \cdots\leq C_{\nu} \delta_\nu^{b_\nu}  \w\},
\end{equation*}
where $b_\mu$ and $C_\mu$ are positive numbers, is equivalent to studying
the operators associated to
$\sA_{\mu_0}$.
See \cite{StreetMultiParameterSingRadonLt} for the definition
of the class of kernels $\sK$ for more general sets $\sA$.

There were two main reasons that our methods applied to $\sA_{\mu_0}$.
\begin{enumerate}
\item There was a natural choice of $\q( X^\mu, d^\mu\w)$ for each $\mu$.
\item Given $\delta_E$ as above (for $\delta\in \sA$), there was a natural (``minimal'') choice $\deltah\in \sA$
with $\deltah_{\mu}=\delta_\mu$, $\forall\mu\in E$.
\end{enumerate}
There are, of course, many subsets $\sA\subseteq \q[0,1\w]^\nu$ where
no such natural choices can be made.  There are more examples
(which do satisfy the above in an appropriate way) which can
be covered by our methods.  However, we know of no simple general
condition unifying these examples.  Moreover, for all
the applications we have in mind, $\sA_{\mu_0}$ will suffice.
We, therefore, say no more on this issue, here.

\section{Singular integrals not of Radon transform type}\label{SectionSingularIntegrals}
The main point of Section \ref{SectionCZOps} was that
a single-parameter special case of the operators
studied in this paper, fell under the general
Calder\'on-Zygmund singular integral
framework.
The point of this section is to make a few remarks
of the multi-parameter special case of our main theorems
which is analogous to this single-parameter special case.
The operators studied in this section can be considered
as a prototype for a multi-parameter analog 
of parts of the Calder\'on-Zygmund theory.

\begin{rmk}
Often, when one hears of {\it multi-parameter singular integrals},
it is the product theory of singular integrals to which is being
referred.  See, e.g., \cite{FeffermanSingularIntegralsOnProductDomains,
NagelSteinOnTheProductTheoryOfSingularIntegrals}.
The operators discussed in this section are not necessarily
of product type.
\end{rmk}

Suppose we are given $\nu$ families of $C^\infty$ vector fields
with single-parameter formal degrees,
$\q( X^\mu, d^{\mu}\w)= \q( X^\mu_1, d^{\mu}_1\w) ,\ldots, \q( X^{\mu}_{q_\mu}, d^{\mu}_{q_\mu}\w)$, $1\leq \mu\leq \nu$.
We suppose that each $\q( X^\mu, d^\mu\w)$ satisfies $\sD\q( \K, \q[0,1\w]\w)$.

Let $\q( X_1,d_1\w),\ldots, \q( X_r,d_r\w)$ be the list of vector fields
consisting of $\q( X^\mu_j, \dhc_j^\mu\w)$ for every $\mu$ and $j$,
where $\dhc_j^{\mu}\in \q[0,\infty\w)^\nu$ is $d_j^\mu$ in the $\mu$
coordinate and $0$ in every other coordinate.  

Taking
$\mu_0=\nu$ (i.e., $\sA=\q[0,1\w]^\nu$), we suppose $\q( X_1,d_1\w),\ldots, \q( X_r,d_r\w)$ generates a finite list $\q( X_1,d_1\w) ,\ldots, \q( X_q,d_q\w)$.
We take $N=q$ and define $\nu$-parameter dilations on $\R^N$ by,
\begin{equation*}
\delta\q( t_1,\ldots, t_q\w) = \q( \delta^{d_1}t_1,\ldots \delta^{d_q}t_q\w),
\end{equation*}
for $\delta\in \q[0,\infty\w)^\nu$.
Define,
\begin{equation*}
\gamma_{\q( t_1,\ldots, t_q\w)} \q( x\w) = e^{t_1X_1+\cdots + t_qX_q}x.
\end{equation*}
We consider operators, $T$, of the form covered in Theorem
\ref{ThmMainThmSecondPass}, where $K\in \sK\q( q,d,a,\nu,\nu\w)$
for some small $a>0$.
It follows from the remarks in Section 17.1 of \cite{StreetMultiParameterSingRadonLt}
that all of the assumptions of Theorem \ref{ThmMainThmSecondPass}
are satisfied with the above choices.
Hence, $T$ is bounded on $L^p$, $1<p<\infty$.

For $K\in \sK\q( q,d,a,\nu,\nu\w)$ decompose $K$,
\begin{equation*}
K=\sum_{j\in \N^{\nu}} \dil{\vsig_j}{2^j},
\end{equation*}
where $\vsig_j$ is as in Definition \ref{DefnsK}.
Corresponding to this decomposition of $K$, one obtains
a decomposition of $T$, $T=\sum_{j\in \N^{\nu}} T_j$.

Let $T_j\q( x,y\w)$ denote the Schwartz kernel of $T_j$.
It follows directly from Proposition 4.22 of
\cite{StreetMultiParameterCCBalls} that
\begin{itemize}
\item $T_j\q( x,y\w)$ is supported on $y\in \B{X}{d}{x}{2^{-j}}$,
\item $\q|T_j\q( x,y\w)\w|\lesssim \Vol{\B{X}{d}{x}{2^{-j}} }^{-1}$.
\end{itemize}
In the one parameter situation, this is just the fact that
a Calder\'on-Zygmund singular integral can be
decomposed into dyadic scales in the usual way.  
Thus, the above is a prototype for a multi-parameter generalization
of the usual dyadic decomposition of a single-parameter Calder\'on-Zygmund
singular integral operator.

\begin{rmk}
In light of the above, $T_j^{*}T_j$ is no ``smoother'' than $T_j$.
The reader used to the single parameter theory might then
suspect that a $T^{*}T$ type iteration argument will not be helpful in our studies.
However, this is not the case.  Indeed, a $T^{*}T$ type iteration argument
was essential to our proof (in Section \ref{SectionGenL2}).
The idea is that when $j_1,j_2\in \N^{\nu}$ and $j_1\wedge j_2\ne j_1,j_2$,
$\q( T_{j_1}^{*} T_{j_2}\w)^{*} T_{j_1}^{*} T_{j_2}$ is smoother than $T_{j_1}^{*}T_{j_2}$.  
\end{rmk}

We now turn to maximal functions.  With all the same choices
as above, we define $\sM$ as in Theorem \ref{ThmMainMaxThm}.
With $\psi_1,\psi_2$ as in Theorem \ref{ThmMainMaxThm},
define a new maximal operator by,
\begin{equation*}
\sMt f\q( x\w) = \sup_{\q|\delta\w|\leq a'} \psi_1\q( x\w) \frac{1}{\Vol{\B{X}{d}{x}{\delta}}} \int_{\B{X}{d}{x}{\delta}} \q| f\q( y\w) \psi_2\q( y\w)\w| \: dy.
\end{equation*}
It follows directly from Proposition 4.22 of
\cite{StreetMultiParameterCCBalls} that
$\sMt f\q( x\w) \lesssim \sM f\q( x\w)$, provided $a'>0$ is sufficiently small.
Hence, $\sMt$ is bounded on $L^p$, $1<p\leq \infty$.
This generalizes the maximal results of
\cite{StreetMultiParameterCCBalls}.

Reduction to Theorem \ref{ThmMainMaxThm} is not the only
way to prove the $L^p$ boundedness of $\sMt$.
Indeed, for each $\mu$, define the maximal operator,
\begin{equation*}
\sMt_\mu f\q( x\w) = \sup_{0<\delta_\mu\leq a''} \denum{\psi}{0}\q( x\w) \frac{1}{\Vol{\B{X^\mu}{d^\mu}{x}{\delta_\mu}}} \int_{\B{X^\mu}{d^\mu}{x}{\delta_\mu}} \q| f\q( y\w) \denum{\psi}{0}\q( y\w)\w| \: dy.
\end{equation*}
It is shown in Section 6.2 of \cite{StreetMultiParameterCCBalls}
that $\sMt_\mu$ is bounded on $L^p$ ($1<p\leq \infty$) for each $\mu$.
This proceeds in a similar manner to the methods in
Section \ref{SectionCZNoSpan}:  by reduction to the
classical Calder\'on-Zygmund theory.

It can be shown that,
\begin{equation}\label{EqnsMtIneq}
\sMt f\q( x\w) = \q(\sMt_1\cdots \sMt_\nu \w)^M f\q( x\w),
\end{equation}
for some large $M$ (provided $a'$ is sufficiently smaller than $a''$); and the $L^p$ boundedness of $\sMt$ follows.
The proof of \eqref{EqnsMtIneq} is somewhat lengthy and technical,
and does not seem to yield Theorem \ref{ThmMainMaxThm} in the general case.
We, therefore, say no more about this here.

In this situation, we can develop a Littlewood-Paley square function
of an appropriate type.
While the operators $D_j$ from Section \ref{SectionAuxOps} were sufficient to create a Littlewood-Paley
square function to prove the $L^p$ boundedness of $T$, they are
not of the same type as $T_j$--and therefore take us out
of the class of operators we are discussing.

Instead, one uses that the distribution $\delta_0\in \sK\q( q, d, a, \nu,\nu\w)$ (Proposition 16.3 of \cite{StreetMultiParameterSingRadonLt}).
Write,
\begin{equation*}
\delta_0 = \sum_{j\in \N^{\nu}} \dil{\vsig_j}{2^j},
\end{equation*}
where $\vsig_j$ is
as in Definition \ref{DefnsK}.
Define,
\begin{equation*}
\Dt_j f\q( x\w) = \denum{\psi}{-2}\q( x\w) \int f\q(\gamma_t\q( x\w) \w) \denum{\psi}{-2}\q( \gamma_t\q( x\w) \w) \dil{\vsig_j}{2^j}\q( t\w) \: dt.
\end{equation*}
Thus, $\denum{\psi}{-2}^2 = \sum_{j\in \N^{\nu}} \Dt_j$.
One can recreate the theory in Section \ref{SectionSquareFunc}
with $D_j$ replaced by $\Dt_j$,
so long as one uses Theorem \ref{ThmMainThmSecondPass}
instead of Theorem \ref{ThmSingleParamSingInt} throughout.
We therefore obtain a Calder\'on-type ``reproducing formula''
and a Littlewood-Paley square function in terms
of $\Dt_j$.

\section{Some comments on maximal operators}\label{SectionMaximalComment}
There are a number of maximal operators in the literature
which are related to the one discussed in Theorem \ref{ThmMainMaxThm}.
The ones most closely related are those 
discussed in
\cite{ChristTheStrongMaximalFunctionOnANilpotentGroup},
where certain strong maximal functions on nilpotent Lie groups
are discussed.
Of course, our methods also apply to convolution
operators on nilpotent groups.
See Section 17.2 of \cite{StreetMultiParameterSingRadonLt}.
Our results can be used to study some of the maximal operators which were
covered in \cite{ChristTheStrongMaximalFunctionOnANilpotentGroup},
and we discuss this below.
At the end of this section, we discuss the fact that not all of
the maximal operators from \cite{ChristTheStrongMaximalFunctionOnANilpotentGroup} are
covered by our results.
Nevertheless, these maximal operators can be covered by the {\it methods} of this
paper, and this will be taken up (and generalized)
in \cite{StreetMultiParameterSingRadonAnal}.


In what follows, we describe the connection between the results
in this paper and the results in
\cite{ChristTheStrongMaximalFunctionOnANilpotentGroup}
in the special case of the three dimensional Heisenberg group, $\Ho$.
All of the comments that follow work more generally for,
say, stratified nilpotent Lie groups, but we leave those details
to the interested reader.
As a manifold, $\Ho = \C\times \R$, and we give it coordinates
$\q( z,t\w) = \q( x,y,t\w)$.
A basis for the left invariant vector fields on $\Ho$
is $X=\partial_x - 2y\partial_y$, $Y= \partial_y+2x\partial_t$,
$T=\partial_t$.

In \cite{ChristTheStrongMaximalFunctionOnANilpotentGroup}, the following
strong maximal function is considered,
\begin{equation*}
\sMt f\q( \xi\w) = \sup_{\delta_1,\delta_2,\delta_3>0} \int_{\q|\q( x,y,t\w)\w|\leq 1} \q|f\q(e^{x\delta_1 X+ y\delta_2 Y + t\delta_3 T}\xi\w)\w| \: dx\: dy\: dt. 
\end{equation*}
It is shown that $\sMt$ is bounded on $L^p$, $1<p\leq \infty$.
We claim that this result follows from Theorem \ref{ThmMainMaxThm}.
To do this, we show,
\begin{equation*}
\sMt_N f\q( \xi\w) = \sup_{N\geq \delta_1,\delta_2,\delta_3>0} \int_{\q|\q( x,y,t\w)\w|\leq 1} \q|f\q(e^{x\delta_1 X+ y\delta_2 Y + t\delta_3 T}\xi\w)\w| \: dx\: dy\: dt,
\end{equation*}
is bounded on $L^p$ ($1<p<\infty$) with bound independent of $N$.

Indeed, let $\psi\geq 0$ be a $C_0^\infty$ function which equals
$1$ on a neighborhood of $0$.  Define,
\begin{equation*}
\sM f\q( \xi\w) = \sup_{1\geq \delta_1,\delta_2,\delta_3>0} \psi\q( \xi\w) \int_{\q|\q( x,y,t\w)\w|\leq a} \q|f\q(e^{x\delta_1 X+ y\delta_2 Y + t\delta_3 T}\xi\w)\w| \psi\q(e^{x\delta_1 X+ y\delta_2 Y + t\delta_3 T}\xi \w) \: dx\: dy\: dt,
\end{equation*}
where $a>0$ is some small number.
It is easy to verify that Theorem \ref{ThmMainMaxThm} applies (see
Section 17.2 of \cite{StreetMultiParameterSingRadonLt}).
Thus $\sM$ is bounded on $L^p$.
Note, for $f$ with small support near $0$, we trivially have,
\begin{equation*}
\sMt_a f\q( \xi\w) \lesssim \sM f\q( \xi\w),
\end{equation*}
for every $\xi$.

Now consider the one-parameter dilations on $\Ho$ given by,
\begin{equation}\label{EqnOneParamDil}
r\q( x,y,t\w) = \q( rx,ry,r^2 t\w),
\end{equation}
for $r\in \q( 0,\infty\w)$.
Define, for $1<p<\infty$,
\begin{equation*}
\dilp{f}{r}\q( \xi\w) = r^{4/p} f\q( r\xi\w),
\end{equation*}
so that $\LpN{p}{\dilp{f}{r}}=\LpN{p}{f}$.

It is easy to see that,
\begin{equation*}
\dilp{\q(\sMt_{N/r} \dilp{f}{r} \w)}{1/r} = \sMt_{N} f.
\end{equation*}
Fix $f\in L^p$ with compact support.  Taking $r$ so large $N/r\leq a$,
and $\dilp{f}{r}$ has small support, we see,
\begin{equation*}
\LpN{p}{\sMt_N f} = \LpN{p}{\sMt_{N/r} \dilp{f}{r}} \lesssim \LpN{p}{\sM \dilp{f}{r}}\lesssim \LpN{p}{\dilp{f}{r}} = \LpN{p}{f}.
\end{equation*}
Thus, we have,
\begin{equation*}
\LpN{p}{\sMt_N f}\lesssim \LpN{p}{f},
\end{equation*}
for every $f$ with compact support.  A limiting argument completes the proof.

There is another approach which can be used to prove the $L^p$
boundedness of $\sMt$.  Namely, one could simply recreate
the entire proof in this paper, without using cutoff functions,
and instead of restricting to $\delta\in \q[0,1\w]^{\nu}$, one
allows $\delta\in \q[0,\infty\w)^\nu$.  It is easy to see that,
in this special case, all of our methods go through.  This is
due to the fact that one has global one-parameter dilations,
\eqref{EqnOneParamDil},
on $\Ho$ which respect each aspect of our proof.

The proof method 
for the $L^2$ boundedness of $\sMt$
in \cite{ChristTheStrongMaximalFunctionOnANilpotentGroup}
is closely related to the proof in this paper.
One main difference is that (for certain maximal operators
more general than $\sMt$),
\cite{ChristTheStrongMaximalFunctionOnANilpotentGroup}
uses transference methods to lift the problem to a 
higher dimensional maximal function.
This allows \cite{ChristTheStrongMaximalFunctionOnANilpotentGroup}
to deal with certain maximal functions on nilpotent groups which are not
directly applicable by our methods.

It turns out that all of the maximal operators
covered in \cite{ChristTheStrongMaximalFunctionOnANilpotentGroup}
{\it can} be covered by our methods, with some modifications.
This will be discussed in
\cite{StreetMultiParameterSingRadonAnal},
where (among other things) the results of
\cite{ChristTheStrongMaximalFunctionOnANilpotentGroup}
will be generalized.

To understand where Theorem \ref{ThmMainMaxThm} falls short, consider the
function $\gamma: \R^{2}\times \R\rightarrow \R$ given
by $\gamma_{s,t}\q( x\w) = x-st$.  It is easy to see,
using the methods of Section 17.5 of
\cite{StreetMultiParameterSingRadonLt} that there
is a product kernel $K\q( s,t\w)\in \sK\q(2,\q(\q(1,0\w),\q( 0,1\w)\w),a,2,2\w) $
supported on $\Q^2\q( a\w)$ (with $a$ as small as we like)
such that the corresponding singular Radon transform
(as in Theorem \ref{ThmMainThmFirstPass})
is not bounded on $L^2$.  This fact was first noted
in \cite{NagelWaingerL2BoundednessOfHilbertTransformsMultiParameterGroup}.

However, if $\psi\in C_0^\infty$, $\psi\geq 0$, is supported sufficiently close to $0$,
the maximal function,
\begin{equation*}
\sM f\q( x\w) = \sup_{0<\delta_1,\delta_2\leq a} \psi\q( x\w) \int \q|f\q(\gamma_{\delta_1 s, \delta_2 t}\q( x\w)\w) \w|\: ds\: dt,
\end{equation*}
is bounded on $L^p\q( \R\w)$ ($1<p\leq \infty$), for $a>0$ sufficiently small.
Thus, in an ideal world, Theorem \ref{ThmMainMaxThm} would be 
generalized to apply to this choice of $\gamma$ (and other, more complicated,\footnote{Because of the form of $\gamma$, the operator $\sM$ is actually equal to a
one-parameter maximal operator, which is covered by Theorem \ref{ThmMainMaxThm}.
There are other choices of $\gamma$ of the same general type where this is not the case,
yet the maximal operator is still bounded.}
$\gamma$ like it).
However, since we used the {\it same} class of $\gamma$
for Theorem \ref{ThmMainThmFirstPass} and Theorem \ref{ThmMainMaxThm},
and Theorem \ref{ThmMainThmFirstPass} fails for this choice of $\gamma$,
our methods need to be modified to attack this sort of example.

In fact, it is possible to modify our methods in a natural way
to deal with cases the same type as this choice of $\gamma$.
This will be discussed in detail
in \cite{StreetMultiParameterSingRadonAnal}.
The reason we have not done so here, is that in order to include
$\gamma$ in the class of functions we study, we will need to 
strengthen other aspects of our assumptions in a few technical ways.
Thus, the maximal theorem proven in
\cite{StreetMultiParameterSingRadonAnal} will not
be strictly stronger than the one in this paper.
This is an issue, since the proof of
Theorem \ref{ThmMainThmSecondPass} used Theorem \ref{ThmMainMaxThm}.
Thus, we would end up weakening Theorem \ref{ThmMainThmSecondPass},
if we attempted to modify Theorem \ref{ThmMainMaxThm}.
See Remark \ref{RmkABetterMaxForRealAnal} for further details
on the sort of maximal results we will prove
in \cite{StreetMultiParameterSingRadonAnal}.

\bibliographystyle{amsalpha}

\bibliography{radon}



\center{MCS2010: Primary 42B20, Secondary 42B25}

\center{Keywords: Calder\'on-Zygmund theory, singular integrals, singular
Radon transforms, maximal Radon transforms, Littlewood-Paley theory, product
kernels, flag kernels, Carnot-Carath\'eodory geometry}

\end{document}